\documentclass[11pt,oneside,reqno]{amsart}
\usepackage{geometry}
\usepackage[utf8]{inputenc}
\usepackage{amstext,latexsym,amsbsy,amsmath,amssymb,amsthm,mathtools,relsize,geometry,enumerate}
\usepackage{hyperref}
\hypersetup{
  colorlinks   = true, 
  urlcolor     = blue, 
  linkcolor    = red, 
  citecolor   = red 
}
\usepackage{mathtools,enumerate,enumitem}

\usepackage{graphicx}
\usepackage{epstopdf}
\setkeys{Gin}{width=\linewidth,totalheight=\textheight,keepaspectratio}
\graphicspath{{./images/}}

\usepackage{multicol}
\usepackage{booktabs}
\usepackage{mathrsfs}
\usepackage{color}

\newcommand{\nbb}{\mathbb{N}}
\newcommand{\rbb}{\mathbb{R}}

\renewcommand{\L}{\mathcal{L}}

\newcommand{\W}{\mathcal{W}}
\newcommand{\Hcal}{\mathcal{H}}
\newcommand{\Pcal}{\mathcal{P}}
\newcommand{\B}{\mathcal{B}}
\newcommand{\la}{\langle}
\newcommand{\ra}{\rangle}

\newcommand{\taut}{\tilde{\tau}}

\newcommand{\Psionem}{\Psi_{1,m}}
\newcommand{\Psitwom}{\Psi_{2,m}}

\newcommand{\nt}{\notag}

\newcommand{\uo}{u^0}
\newcommand{\um}{u^m}
\newcommand{\Um}{U^m}

\newcommand{\vm}{v^m}

\newcommand{\ebar}{\bar{\varepsilon}}

\newcommand{\umt}{\tilde{u}^m}
\newcommand{\vmt}{\tilde{v}^m}

\newcommand{\ut}{\tilde{u}}
\newcommand{\vt}{\tilde{v}}

\newcommand{\Jmu}{\pi_1 J^m_{0,t}}
\newcommand{\Jmv}{\pi_2 J^m_{0,t}}
\newcommand{\Amu}{\pi_1 A^m_{0,t}}
\newcommand{\Amv}{\pi_2 A^m_{0,t}}
\newcommand{\rhomu}{\pi_1 \rho^m}
\newcommand{\rhomv}{\pi_2 \rho^m}

\newcommand{\num}{\nu^m}

\newcommand{\tautil}{\tilde{\tau}}
\newcommand{\x}{\mathrm{x}}

\newcommand{\vbar}{\bar{v}}

\newcommand{\ubar}{\bar{u}}

\newcommand{\utilde}{\tilde{u}}

\newcommand{\Ut}{\tilde{U}}
\newcommand{\dmnb}{d^{\,m}_{N,\beta}}
\newcommand{\dmtnb}{\tilde{d}^{\,m}_{N,\beta}}

\newcommand{\dmtnbl}{\tilde{d}^{\,m}_{N,\beta,\lambda}}
\newcommand{\dtnb}{\tilde{d}^{\,0}_{N,\beta}}
\newcommand{\dtnbhalf}{\tilde{d}^{\,0}_{N,\beta/2}}
\newcommand{\dtnbtwo}{\tilde{d}^{\,0}_{N,2\beta}}

\newcommand{\mi}{\wedge}

\newcommand{\Tr}{\text{Tr}}

\renewcommand{\d}{\text{d}}

\newcommand{\domain}{\mathcal{O}}

\newcommand{\f}{\varphi}

\newcommand{\grad}{\nabla}
\newcommand{\Dom}{\text{Dom}}

\newcommand{\ddt}{\tfrac{\d}{\d t}}

\newcommand{\Gammau}{\pi_1\Gamma^m}
\newcommand{\Gammav}{\pi_2\Gamma^m}

\newcommand{\E}{\mathbb{E}}

\renewcommand{\P}{\mathbb{P}}

\newcommand{\nbar}{\bar{n}}

\newcommand{\close}{\!\!\!}

\theoremstyle{plain}
\newtheorem{theorem}{Theorem}[section]
\newtheorem{corollary}[theorem]{Corollary}

\newtheorem{lemma}[theorem]{Lemma}
\newtheorem{assumption}[theorem]{Assumption}
\newtheorem{proposition}[theorem]{Proposition}

\theoremstyle{definition}

\newtheorem{remark}[theorem]{Remark}

\numberwithin{equation}{section}

\title{The small mass limit for long time statistics of a stochastic nonlinear damped wave equation}

\author{ Hung D.~Nguyen$^1$}

\address{$^1$ Department of Mathematics, University of California, Los Angeles, California, USA}

\begin{document}

\begin{abstract}
We study the long time statistics of a class of semi--linear damped wave equations with polynomial nonlinearities and perturbed by additive Gaussian noise in dimensions 2 and 3. We find that if sufficiently many directions in the phase space are stochastically forced, the system is exponentially attractive toward its unique invariant measure with a convergent rate that is uniform with respect to the mass. Then, in the small mass limit, we prove the convergence of the first marginal of the invariant measures in a suitable Wasserstein distance toward the unique invariant measure of a stochastic reaction--diffusion equation. Together with the uniform geometric ergodicity, we obtain the validity of the small mass limit for the solutions on the infinite time horizon $[0,\infty)$. This extends previously known results established for the damped wave equations under Lipschitz nonlinearities. 
\end{abstract}

\maketitle

\section{Introduction} \label{sec:intro}

Let $\domain\subset \rbb^d$, $d=2,3$, be a bounded open domain with smooth boundary. We consider the following equation in the unknown variable $\um(t)=\um(x,t):\domain\times[0,\infty)\to\rbb$
\begin{align}
m\,\partial_{tt} \um(t)&=-\partial_t \um(t)+\triangle \um(t) + \f(\um(t))+ Q\, \partial_t w(t),\label{eqn:wave:original}\\
\um(0)&=u_0 \in H^1_0(\domain),\quad \partial_t \um(0)=v_0\in L^2(\domain),\quad \um(t)\big|_{\partial\domain}=0,\nt
\end{align}
which, by Newton's Law, describes the motion of a particle suspended in a randomly continuous medium, under the impact of external forces $\f+Q\partial_t w$ and a damping force proportional to the speed \cite{barbu2002stochastic,carmona1988random,crauel1997random}. In equation~\eqref{eqn:wave:original}, $m>0$ represents the size of the particle mass, $\f:\rbb\to\rbb$ is the nonlinearity satisfying dissipative conditions with polynomial growth, $w(t)$ is a cylindrical Wiener process taking values in $L^2(\domain)$, and $Q:L^2(\domain)\to \L^2(\domain)$ is a symmetric linear bounded map. We note that equation~\eqref{eqn:wave:original} is also subjected to Dirichlet condition. For simplicity, we set all other physical constants to~1.

When setting $m=0$, equation~\eqref{eqn:wave:original} is formally reduced to the following reaction--diffusion equation
\begin{align}
\partial_{t} \uo(t)&=\triangle \uo(t) + \f(\uo(t))+ Q\, \partial_t w(t),\label{eqn:react-diff:original}\\
\uo(0)&=u_0,\quad u(t)\big|_{\partial\domain}=0.\nt
\end{align}
It is well--known that in the regime of $m\to 0$, the so--called Smoluchowski--Kramers approximations \cite{kramers1940brownian,smoluchowski1916drei}, the process $\um$ converges to $\uo$ in any finite time interval $[0,T]$. Results in this direction for stochastic wave equations appeared as early as in the work of \cite{cerrai2006smoluchowski,cerrai2006smoluchowski2}. In particular, for a broad class of nonlinearities, it can be shown that the following holds \cite{cerrai2006smoluchowski,cerrai2006smoluchowski2}
\begin{align} \label{lim:probability}
\lim_{m\to 0}\P(\sup_{t\in[0,T]}\|\um(t)-\uo(t)\|_H>\varepsilon)= 0,\quad \varepsilon>0,\, T>0. 
\end{align}
If we additional assume that $\f$ is Lipschitz, a convergence in $C([0,T];H)$  established in \cite{cerrai2016smoluchowski,salins2019smoluchowski}
\begin{align} \label{lim:L^p}
\lim_{m\to 0}\E\big[\sup_{t\in[0,T]}\|\um(t)-\uo(t)\|_H^p\big]= 0,\quad T>0,
\end{align}
for suitable $p\ge 2$. More interestingly, if the constant friction  $-\partial_t\um(t)$ is replaced by a nonlinear damping term, the limit equation~\eqref{eqn:react-diff:original} must be modified so that results of~\eqref{lim:probability}--\eqref{lim:L^p}--types still hold \cite{cerrai2022smoluchowski}. See also the related work in \cite{fukuizumi2022non,lv2012averaging,shi2021small,zine2022smoluchowski}. Altogether, the results established therein rigorously justify that the motion of a small particle in any finite time interval can be approximated by the parabolic equation \eqref{eqn:react-diff:original} instead of the hyperbolic equation \eqref{eqn:wave:original} \cite{cerrai2006smoluchowski,cerrai2017smoluchowski}.

Although the small mas limits of~\eqref{eqn:wave:original} have been mostly investigated for finite time windows, there are several results concerning long time behaviors. For example, asymptotics of the exit times and the large deviation theory in the context of small mass limits were central in the work of \cite{cerrai2014smoluchowski,cerrai2016smoluchowski,cerrai2022small}. 
It is also a matter of interest to compare the statistically steady states of~\eqref{eqn:wave:original} and \eqref{eqn:react-diff:original}, which have been well--studied on their own. In \cite{cerrai2006smoluchowski}, under the assumption that systems~\eqref{eqn:wave:original} and \eqref{eqn:react-diff:original} are of gradient type, the invariant measure of~\eqref{eqn:wave:original} was explicitly derived. As a consequence, its first marginal was shown to coincide with the invariant measure of~\eqref{eqn:react-diff:original}. It is worth to mention that in general, it is difficult to provide a formula for the invariant measures, letting alone proving that they are the same. Nevertheless, it turns out that in dimension $d=1$, provided that $\f$ satisfies certain dissipative conditions and that \cite[Hypothesis 2]{cerrai2020convergence}
\begin{align*}
\lambda\in[1,3),\quad \f(x)=\text{O}(|x|^\lambda)\,\, \text{as}\,\, |x|\to\infty,
\end{align*}
there exists a Wasserstein distance $\W$ in $\Pcal r(L^2(\domain))$, the space of probability measures in $L^2(\domain)$, such that \cite[Theorem 5.1]{cerrai2020convergence}
\begin{align} \label{lim:Wass}
\lim_{m\to 0} \W (\pi_1\num,\nu^0)=0.
\end{align}
In the above, $\{\num\}_{m>0}$ is an arbitrarily sequence of invariant measures of~\eqref{eqn:wave:original}, $\pi_1\num$ is the first marginal of $\num$ on $L^2(\domain)$, and $\nu^0$ is the unique invariant measure of~\eqref{eqn:react-diff:original}.
In higher spatial dimensions,~\eqref{lim:Wass} was also established in \cite[Theorem 5.1]{cerrai2020convergence} under the additional restriction $\lambda=1$, leaving out the class of non--Lipschitz polynomials. One of our goals here is to push through this threshold. 

More specifically, our main contributions in this note are summarized as follows: in dimensions $d=2,3$, under the assumptions that $\lambda\in[1,2)$ and that sufficiently many directions of the phase space are stochastically forced, we find that for all $m$ sufficiently small,~\eqref{eqn:wave:original} admits a unique invariant measure $\num$ and that it is exponentially attractive toward $\num$ with a convergent rate uniformly with respect to $m$, see Theorem~\ref{thm:geometric-ergodicity:mass} below. Furthermore, we establish limit~\eqref{lim:Wass} for a suitable Wasserstein distance $\W$ in $\Pcal r(L^2(\domain))$, cf. Theorem~\ref{thm:nu^m->nu^0}. As a consequence, given $(\um_0,\vm_0)$ a sequence of initial conditions with sufficient regularity and $f:L^2(\domain)\to \rbb$ an observable satisfying certain Lipschitz conditions related to $\W$, we show that
\begin{align} \label{lim:Wass:f}
\lim_{m\to 0}\sup_{t\ge 0}\big|\E f(\um(t))-\E f(\uo(t))\big| =0.
\end{align}
See Theorem~\ref{thm:m->0:f} for a precise statement of~\eqref{lim:Wass:f}. 

We note that for fixed $m$, statistically steady states of~\eqref{eqn:wave:original} have been well--studied. The existence of invariant measures of~\eqref{eqn:wave:original} for a broad class of nonlinearities was established in the work of \cite{crauel1997random} via the theory of random attractors. Unique ergodicity was then obtained in \cite{barbu2002stochastic} in dimensions $d\le 3$ under the condition that the effect of $\f$ is dominated by the Laplacian and that $
\f(x)=\text{O}(|x|^\lambda)$ as $|x|\to\infty,$ where $\lambda\in[1,3)$. Exponential mixing for~\eqref{eqn:wave:original} was investigated under degenerate noise in \cite{martirosyan2014exponential} making use of an asymptotic coupling argument. On the other hand, ergodicity of~\eqref{eqn:react-diff:original} has now become a classical topic and can be found in many previous literature \cite{cerrai2020convergence,da1996ergodicity,da2014stochastic, glatt2022short,hairer2011theory}. As mentioned above, in this paper, we address the problem of unique ergodicity of~\eqref{eqn:wave:original} under suitable non--Lipschitz nonlinearities in the context of $m\to 0$. Following closely the framework of \cite{butkovsky2020generalized,hairer2006ergodicity,
hairer2008spectral,hairer2011asymptotic,kulik2017ergodic,
kulik2015generalized}, the uniform geometric ergodicity argument consists of two crucial ingredients: the \emph{d--contracting} property of the Markov semigroup associated with~\eqref{eqn:wave:original} and the notion of \emph{d--small} sets in the phase space; see Section~\ref{sec:geometric-ergodicity:proof}. In turn, the former relies on the so--called \emph{asymptotic strong--Feller} property, which is a large--time smoothing effect of the Markovian dynamics. The latter relies on an irreducibility property asserting that the solutions can always return to any neighborhood of the origin. In contrast with literature where unique ergodicity of~\eqref{eqn:wave:original} in dimensions $d\le 3$ was previously obtained for $\lambda\in[1,3)$ \cite{barbu2002stochastic,martirosyan2014exponential}, due to the singular limit $m\to 0$, the ergodicity result as well as the analysis in this note require the condition $\lambda\in[1,2)$. As a trade--off on the growth rate of the nonlinearities, the convergent speed is independent of the mass, which is very convenient for the purpose of studying the small mass regime. The uniform ergodicity result is rigorously given in Theorem \ref{thm:geometric-ergodicity:mass}, whose proof will be supplied in Section \ref{sec:geometric-ergodicity}. In particular, the condition $\lambda\in[1,2)$ will be employed in Proposition \ref{prop:asymptotic-Feller} establishing the asymptotic strong--Feller property. In turn, this will be invoked to prove Theorem \ref{thm:geometric-ergodicity:mass}.

Turning to~\eqref{lim:Wass}, the restriction $\lambda=1$ in dimensions $d=2,3$, was previously imposed in \cite{cerrai2020convergence} due to a lack of higher regularity of the invariant measure $\num$ as well as suitable bounds on $\sup_{t\in[0,T]}\|\um(t)\|_{L^\infty(\domain)}$ that is independent of $m$. In our work, we circumvent the former by proving that compared with the solutions, $\num$ actually supports in higher regular spaces (see Proposition~\ref{prop:regularity}). We also tackle the latter by performing a series of bootstrap arguments and ultimately obtain uniform estimates on $\sup_{t\in[0,T]}\|\um(t)\|_{L^\infty(\domain)}$. In particular, these delicate bounds employ the crucial condition $\lambda\in[1,2)$; see the proof of Lemma~\ref{lem:moment-bound:H^2:|Au|^2:sup_[0,T]:random-initial-cond} in Section~\ref{sec:moment-bound}. As a consequence, we demonstrate that in the regime of $m\to 0$, the limit~\eqref{lim:Wass} holds. In turn, together with the uniform exponential convergence of~\eqref{eqn:wave:original} toward $\num$, limit~\eqref{lim:Wass} will be invoked to establish~\eqref{lim:Wass:f}, i.e., the validity of the approximation for~\eqref{eqn:wave:original} by~\eqref{eqn:react-diff:original} on the infinite time horizon $[0,\infty)$. The strategy that we employ to prove~\eqref{lim:Wass:f} is drawn upon the framework recently developed in \cite{glatt2021mixing}.

The rest of the paper is organized as follows: in Section~\ref{sec:result}, we introduce all the functional settings as well as the main assumptions on the nonlinearities and noise structures. We also formulate our results in this section, including Theorem~\ref{thm:geometric-ergodicity:mass} on the uniform geometric ergodicity, and the main results concerning the small mass limit stated in Theorem~\ref{thm:nu^m->nu^0} and Theorem~\ref{thm:m->0:f}. In Section~\ref{sec:moment-bound}, we derive a priori moment bounds on the solutions of~\eqref{eqn:wave:original} that will be employed to prove the main results. We then discuss the asymptotic coupling and prove the geometric ergodicity in Section~\ref{sec:geometric-ergodicity}. In Section~\ref{sec:small-mass}, we provide the detailed proofs of the small mass limits making use of the previous sections. In Appendix~\ref{sec:stochastic-convolution}, we perform several estimates on the linear version of~\eqref{eqn:wave:original} ($\f\equiv 0$) that were employed to prove geometric ergodicity and moment bounds in higher regularity. In Appendix \ref{sec:react-diff}, we review previously established results on the reaction--diffusion equation~\eqref{eqn:react-diff:original}. In Appendix \ref{sec:auxiliary-result}, we collect  auxiliary lemmas on the nonlinearities and Wasserstein distances, that were invoked to prove the small mass limits.

\section{Assumptions and main results} \label{sec:result}

\subsection{Functional setting} \label{sec:functional-setting}
Letting $\domain$ be a smooth bounded domain in $\rbb^d$, we denote by $H$ the Hilbert space $L^2(\domain)$ endowed with the inner product $\la\cdot,\cdot\ra_H$ and the induced norm $\|\cdot\|_H$.

Let $A$ be the realization of $-\triangle$ in $H$ endowed with the Dirichlet boundary condition and the domain $\Dom(A)=H^1_0(\domain)\cap H^2(\domain)$. It is well-known that there exists an orthonormal basis $\{e_k\}_{k\ge 1}$ in $H$ that diagonalizes $A$, i.e.,  
\begin{equation}\label{eqn:Ae_k=alpha_k.e_k}
Ae_k=\alpha_k e_k,
\end{equation}
for a sequence of positive numbers $\alpha_1<\alpha_2<\dots$ diverging to infinity. For each $r\in\rbb$, we denote
\begin{equation}
H^r=\Dom(A^{r/2}),
\end{equation}
endowed with the inner product
\begin{align*}
\la u_1,u_2\ra_{H^r}=\la A^{r/2}u_1,A^{r/2}u_2\ra_H.
\end{align*}
In view of~\eqref{eqn:Ae_k=alpha_k.e_k}, the inner product in $H^r$ may be rewritten as \cite{cerrai2006smoluchowski,cerrai2020convergence}
\begin{align*}
\la u_1,u_2\ra_{H^r}=\sum_{k\ge 1}\alpha_k^{r}\la u_1,e_k\ra_H\la u_2,e_k\ra_H.
\end{align*}
The induced norm in $H^r$ then is given by
\begin{align*}
\|u\|^2_{H^r}=\sum_{k\ge 1}\alpha_k^{r}|\la u,e_k\ra_H|^2.
\end{align*}
For $n\ge 1$, we denote by $P_n$ the projection onto the span$\{e_1,\dots,e_n\}$, i.e.,
\begin{equation} \label{form:P_n.u}
P_nu=\sum_{k=1}^n \la u,e_k\ra_He_k.
\end{equation}
To construct a phase space where~\eqref{eqn:wave:original} evolves on, for each $\beta\in\rbb$, let $\Hcal^\beta$ be the product space given by
\begin{equation} \label{form:H^beta.x.H^(beta-1)}
\Hcal^\beta = H^\beta \times H^{\beta-1},
\end{equation}
endowed by the norm
\begin{align*}
\|(u,v)\|^2_{\Hcal^\beta}=\|u\|^2_{H^\beta}+\|v\|^2_{H^{\beta-1}}.
\end{align*}
The projection of $\Hcal^\beta$ on the marginal spaces is denoted by $\pi$, namely,
\begin{align*}
\pi_1(u,v)=u,\quad \pi_2(u,v)=v.
\end{align*}

Following \cite{barbu2002stochastic}, by setting $\vm(t)=\partial_t \um(t)$, we may recast~\eqref{eqn:wave:original} as the following system in the space $\Hcal^1=H^1\times H$
\begin{equation} \label{eqn:wave}
\begin{aligned}
\d \um(t)&=\vm(t)\d t,\\
m\,\d \vm(t)&=-A\um(t)\d t-\vm(t)\d t+\f(\um(t))\d t+ Q\d w(t),\\
(\um(0),\vm(0))&=(u_0,v_0)\in \Hcal^1.
\end{aligned}
\end{equation}
The generator associated with~\eqref{eqn:wave} is denoted by  \cite{cerrai2020convergence}
\begin{equation} \label{form:L^m}
\L^m g(u,v)= \la D_u g, v\ra_H+\tfrac{1}{m}\la D_v g,-Au-v(t)+\f(u(t))\ra_H+\tfrac{1}{2m^2} \Tr(D_{vv}gQQ^*),
\end{equation}
and is defined for all $g\in C^2(\Hcal^1;\rbb)$ such that $\Tr(D_{vv}gQQ^*)<\infty$. Here, we recall the definition $\Tr(M)=\sum_{k\ge 1}\la Me_k,e_k\ra_H$ for $M:H\to H$.

\subsection{Main assumptions} \label{sec:result:assumption}
In this subsection, we state the main assumptions on the nonlinearities and the noise structure what will be employed throughout the paper.

Concerning the nonlinearities $\f:\rbb\to\rbb$, we make the following standard assumptions on $\f$: \cite{cerrai2020convergence,glatt2022short}

\begin{assumption} \label{cond:phi:well-posed} $\f\in C^1$ satisfies $\f(0)=0$.

\begin{enumerate}[noitemsep,topsep=0pt,wide=\parindent, label=\arabic*.,ref=\theassumption.\arabic*]

\item There exist positive constants $a_1,a_2,a_3$ and $\lambda\in[1,2)$ such that for all $x\in\rbb$,
\begin{equation} \label{cond:phi:phi(x)=O(x^lambda)} 
|\f(x)|\le a_1(1+|x|^{\lambda}),
\end{equation}
and
\begin{equation} \label{cond:phi:x.phi(x)<-x^(lambda+1)}
x\f(x)\le -a_2|x|^{\lambda+1}+a_3.
\end{equation}

\item There exist positive constants $a_4$ and $a_\f$ such that the derivative $\f'$ satisfies 
\begin{equation}  \label{cond:phi:phi'=O(x^(lambda-1))}
|\f'(x)| \le a_4(|x|^{\lambda-1}+1),\quad x\in\rbb,
\end{equation}
and
\begin{equation} \label{cond:phi:sup.phi'<a_f}
\sup_{x\in\rbb}\f'(x)=:a_\f<\infty.
\end{equation}

\end{enumerate}

\end{assumption}

\begin{remark} \label{remark:Phi_1}
Under Assumption~\ref{cond:phi:well-posed}, there always exists a function $\Phi_1\ge 1$ such that for all $x\in\rbb$ \cite{cerrai2020convergence}
\begin{equation} \label{cond:Phi_1}
-\Phi_1'(x)=\f(x),\quad\text{and}\quad c_\f|x|^{\lambda+1}\le \Phi_1(x)\le C_\f(|x|^{\lambda+1}+1),
\end{equation}
for some positive constant $c_\f$ and $C_\f$. We will exploit~\eqref{cond:Phi_1} to perform moment bounds on equation~\eqref{eqn:wave} later in Section~\ref{sec:moment-bound}.
\end{remark}

With regard to the noise, we assume that $w(t)$ is a cylindrical Wiener process on $H$, whose decomposition is given by
$$w(t)=\sum_{k\ge 1}e_kB_k(t),$$
where $\{e_k\}_{k\ge 1}$ is the orthonormal basis of $H$ as in \eqref{eqn:Ae_k=alpha_k.e_k} and $\{B_k(t)\}_{k\ge 1}$ is a sequence of independent standard one--dimensional Brownian motions, each defined on the same stochastic basis $\mathcal{S}=(\Omega, \mathcal{F},\{\mathcal{F}_t\}_{t\ge 0},\P)$ \cite{karatzas2012brownian}. Concerning the linear operator $Q$, we impose the following assumption \cite{bonaccorsi2012asymptotic,cerrai2020convergence,da2014stochastic,glatt2017unique}:

\begin{assumption} \label{cond:Q}
1. The operator $Q:H\to H$ is a symmetric, non--negative, bounded linear map satisfying 
\begin{equation} \label{cond:Q:Tr(QA^3Q)}
\emph{Tr}(QA^3Q^*)<\infty.
\end{equation}

2. Let $\f$ be as in Assumption~\ref{cond:phi:well-posed} and $\nbar\in\nbb$ be the minimal index such that
\begin{equation} \label{cond:phi:ergodicity}
\nbar=\min\{n\in\nbb:\alpha_n>a_\f\},
\end{equation}
where $\alpha_{n}$ is the eigenvalue associated with $e_{n}$ as in~\eqref{eqn:Ae_k=alpha_k.e_k} and $a_\f=\sup_{x\in\rbb}\f'(x)$ as in \eqref{cond:phi:sup.phi'<a_f}. There exists a positive constant $a_Q$ such that
\begin{equation} \label{cond:Q:ergodicity}
\|Qu\|_H\ge a_Q\|P_{\nbar}u\|_H,\quad u\in H,
\end{equation}
where  $P_{\nbar}$ is the projection onto $\{e_1,\dots,e_{\nbar}\}$ as in~\eqref{form:P_n.u}.
\end{assumption}

\begin{remark} We note that condition \eqref{cond:Q:Tr(QA^3Q)} states that the noise's regularity is at least in $H^3$. This condition is stronger than \cite[Condition (2.4)]{cerrai2020convergence} where noise is assumed to belong to $H^\beta$ for some $\beta>0$. On the other hand, condition \eqref{cond:Q:ergodicity} in dimension $d=2,3$, is analogous to \cite[Condition (2.10)]{cerrai2020convergence} in dimension $d=1$. More importantly, \eqref{cond:Q:ergodicity} states that a sufficiently large but finite number of directions in the phase space are required to be stochastically forced.

\end{remark}

Under Assumption~\ref{cond:phi:well-posed} and Assumption~\ref{cond:Q}, it is a classical result that \eqref{eqn:wave} is well-posed. That is, fixing the stochastic basis $\mathcal{S}$, for each initial data $U_0=(u_0,v_0)\in\Hcal^1=H^1\times H$, equation~\eqref{eqn:wave} admits a unique weak solution $\Um(t;U_0)=(\um(t;U_0),\vm(t;U_0))$. The argument may be derived following standard methods, e.g., the Galerkin approximation \cite[Section 3.3]{crauel1997random}. See also \cite{carmona1988random,
da1996ergodicity,da2014stochastic,
ondrejat2004existence,ondrejat2010stochastic}.

As a consequence of the well--posedness, we can thus introduce the Markov transition probabilities of the solution $\Um(t)$ by
\begin{align*}
P_t^m(U_0,A):=\P(\Um(t;U_0)\in A),
\end{align*}
which are well--defined for $t\ge 0$, initial states $U_0\in\Hcal^1$ and Borel sets $A\subseteq \Hcal^1$. Letting $\B_b(\Hcal^1)$ denote the set of bounded Borel measurable functions $f:\Hcal^1 \rightarrow \rbb$, the associated Markov semigroup $P_t^m:\B_b(\Hcal^1)\to\B_b(\Hcal^1)$ is defined and denoted by
\begin{align}\label{form:P_t^m}
P_t^m f(U_0)=\E[f(\Um(t;U_0))], \quad f\in \B_b(\Hcal^1).
\end{align}
Let $\Pcal r(\Hcal^1)$ be the space of probability measures in $\Hcal^1$. The push--forward of $\nu\in\Pcal r(\Hcal^1)$ under the action of $P^m_t$ is denoted by $(P^m_t)^*\nu$ and defined as
\begin{align*}
(P^m_t)^*\nu (A)= \int_{\Hcal^1}P_t^m(U,A)\nu(\d U).
\end{align*}
The first marginal distribution of $\nu$ in $H^1$ is denoted by $\pi_1\nu$, namely,
\begin{align*}
\pi_1\nu(A) = \int_{\Hcal^1} \mathbf{1}_A(u)\nu(\d u,\d v),
\end{align*}
which is defined for all Borel sets $A\subseteq H^1$.

\subsection{Uniform geometric ergodicity} \label{sec:result:ergodicity}
We now turn to the topic of uniform geometric ergodicity of~\eqref{eqn:wave}. Recall that a probability measure $\num\in \Pcal r(\Hcal^1)$ is said to be {\it\textbf{invariant}} for the semigroup $P_t^m$ if for every $f\in \B_b(\Hcal^1)$
\begin{align*}
\int_{\Hcal^1}P_t^m f(U)\num(\d U)=\int_{\Hcal^1} f(U)\num(\d U).
\end{align*}
It is well--known that under Assumption~\ref{cond:phi:well-posed} and Assumption~\ref{cond:Q}, $P^m_t$ always admits an invariant probability measure $\num$, which is obtained via the classical Krylov--Bogoliubov tightness argument \cite{barbu2002stochastic} applied to a sequence of time--averaged measures. Alternatively, in \cite{crauel1997random}, the existence of random attractors is proved, thereby implying the existence of invariant probability measures. As mentioned in the introduction, in dimension $d=1$, if~\eqref{eqn:wave} is of gradient type, an explicit formula of $\num$ was given in \cite{cerrai2006smoluchowski}.

With regard to unique ergodicity and exponentially mixing of~\eqref{eqn:wave}, as mentioned in the introduction, we will draw upon the framework developed in \cite{hairer2006ergodicity,hairer2008spectral} and later popularized in \cite{butkovsky2020generalized,hairer2011asymptotic,
hairer2011theory,kulik2017ergodic,kulik2015generalized}, tailored to our settings. For the reader's convenience, we briefly review the theory below.

For a slight abuse of notation, recall that a function $d:\Hcal^1\times\Hcal^1\to [0,\infty)$ is called \emph{distance--like} if it is symmetric, lower semi--continuous and $d(U,\Ut)=0$ if and only if $U=\Ut$; see \cite[Definition 4.3]{hairer2011asymptotic}. Let $\W_d$ be the Wasserstein distance in $\Pcal r(\Hcal^1)$ associated with $d$ and given by
\begin{align} \label{form:W_d}
\W_d(\nu_1,\nu_2)=\inf \E\, d(X,Y),
\end{align}
where the infimum is taken over all bivariate random variables $(X,Y)$ such that $X\sim \nu_1$ and $Y\sim \nu_2$. In case, $d$ is a metric in $\Hcal^1$, by the dual Kantorovich Theorem, $\W_{d}$ is equivalently defined as \cite[Theorem 5.10]{villani2008optimal}
\begin{align} \label{form:W_d:dual-Kantorovich}
\W_{d}(\nu_1,\nu_2)=\sup_{[f]_{\text{Lip},d}\leq 1}\Big|\int_{\Hcal^1}f(U)\nu_1(\d U)-\int_{\Hcal^1}f(U)\nu_2(\d U)\Big|,
\end{align}
where
\begin{align} \label{form:Lipschitz}
[f]_{\text{Lip},d}=\sup_{U\neq \Ut}\frac{|f(U)-f(\tilde{U})|}{d(U,\tilde{U})}.
\end{align}
On the other hand, if $d$ is a distance--like function, then the following one--sided inequality holds
\begin{equation} \label{ineq:W_d(nu_1,nu_2):dual}
\W_{d}(\nu_1,\nu_2) \ge \sup_{[f]_{\text{Lip},d}\le 1}\Big|\int_{\Hcal^1} f(U)\nu_1(\d U)-\int_{\Hcal^1} f(U)\nu_2(\d U)\Big|.
\end{equation}
See \cite[Proposition A.3]{glatt2021mixing} for a further discussion of~\eqref{ineq:W_d(nu_1,nu_2):dual}. We will particularly invoke~\eqref{ineq:W_d(nu_1,nu_2):dual} to establish Theorem~\ref{thm:m->0:f} below.

To study geometric ergodicity of~\eqref{eqn:wave}, we introduce the function $V_m:\Hcal^1\to[0,\infty)$ defined as
\begin{align} \label{form:V_m}
V_m(u,v)= m\|u\|^2_{H^1}+m^2\|v\|^2_H+m\|u\|^{\lambda+1}_{L^{\lambda+1}}+\|u\|^2_H,
\end{align}
and the associated metric $\varrho^m_\beta:\Hcal^1\times\Hcal^1\to[0,\infty)$
\begin{align} \label{form:varrho^m_beta}
\varrho^m_\beta(U,\Ut)=\inf\int_0^1 e^{\beta V_m(\gamma(s))}\big(m\|\pi_1\gamma'(s)\|^2_{H^1}+m^2\|\pi_2\gamma'(s)\|^2_{H}+\|\pi_1\gamma'(s)\|^2_{H}\big)^{1/2}\d s,
\end{align}
where the infimum is taken over all paths $\gamma \in C^1([0,1];\Hcal^1)$ such that $\gamma(0)=U$ and $\gamma(1)=\Ut$. Following the framework of \cite{butkovsky2020generalized,glatt2021mixing,hairer2011asymptotic,
kulik2017ergodic,kulik2015generalized,nguyen2023ergodicity}, in our settings, for $N>0$, we consider the distance
\begin{align} \label{form:d^m_(N,beta)}
\dmnb(U,\Ut)=N\varrho^m_\beta (U,\Ut) \mi 1.
\end{align}
The actual convergent rate of~\eqref{eqn:wave} toward equilibrium is measured through the distance--like function $\dmtnb$ defined as
\begin{align} \label{form:d.tilde^m_(N,beta)}
\dmtnb(U,\Ut) =\sqrt{ \dmnb(U,\Ut)\big[1+e^{\beta V_m(U)}+e^{\beta V_m(\Ut)}   \big]}.
\end{align}

We now state the first main result giving the unique ergodicity and the uniform exponential convergent rate of \eqref{eqn:wave} with respect to the mass $m$.

\begin{theorem} \label{thm:geometric-ergodicity:mass}
Suppose Assumption \ref{cond:phi:well-posed} and Assumption \ref{cond:Q} hold. Then, for all $m$ sufficiently small,~\eqref{eqn:wave} admits a unique invariant probability measure $\num$. Furthermore, for all $N$ sufficiently large and $\beta$ sufficiently small, there exists a positive constant $T^*=T^*(N,\beta)$ independent of $m$ such that
\begin{align} \label{ineq:geometric-ergodicity:mass}
\limsup_{m\to 0}\sup_{\nu_1\neq \nu_2\in \Pcal r(\Hcal^1)}  \frac{ \W_{\dmtnb}( (P^m_{nT^*})^*\nu_1,(P^m_{nT^*})^*\nu_2) }{ \W_{\dmtnb}( \nu_1,\nu_2)  } \le e^{-c^* n},\quad n\in\nbb,
\end{align}
where $\W_{\dmtnb}$ is the Wasserstein distances associated with $\dmtnb$ as in~\eqref{form:d.tilde^m_(N,beta)}, and  $c^*$ is a positive constant independent of $m$, $n$, $\nu_1$ and $\nu_2$.

\end{theorem}

As mentioned in the introduction, the argument of~\eqref{ineq:geometric-ergodicity:mass} relies on two ingredients, namely, the $\dmnb-contracting$ property of $P^m_t$ and the $\dmnb-small$ sets. Then, combined with suitable energy estimates, we will be able to conclude the convergent rate~\eqref{ineq:geometric-ergodicity:mass} in term of $\dmtnb$. In Section~\ref{sec:geometric-ergodicity:proof}, we will discuss these terminologies in detail (see Proposition~\ref{prop:contracting-d-small}) and supply the proof of Theorem~\ref{thm:geometric-ergodicity:mass}, cf.~\eqref{ineq:geometric-ergodicity:mass}. In turn, the uniform exponential convergent rate~\eqref{ineq:geometric-ergodicity:mass} will be exploited to study the small mass limits, which we describe next.

\subsection{Small mass limits} \label{sec:result:small-mass}

Having established the geometric ergodicity for $\num$ in the previous section, we turn to the main topic of the paper concerning the small mass limits as $m\to 0$. We first recast equation~\eqref{eqn:react-diff:original} as
\begin{align}
\d u^0(t)&=-A u^0(t)\d t+\f(u^0(t))\d t+ Q\d w(t), \quad u^0(0)=u_0\in H\label{eqn:react-diff}.
\end{align}
It is well--known that under Assumption \ref{cond:phi:well-posed} and Assumption~\ref{cond:Q}, equation~\eqref{eqn:react-diff} admits a unique invariant probability measure $\nu^0$ \cite{cerrai2020convergence,hairer2011theory,glatt2022short}. To study the convergence of $\num$ toward $\nu^0$, we introduce the $H-$analogue version of $\dmtnb$, namely, the distance--like function $\dtnb:H\times H\to [0,\infty)$ given by
\begin{equation} \label{form:d.tilde_(N,beta)}
\dtnb(u,\ut)=\sqrt{(N\|u-\ut\|_H\mi 1)\big[1+ e^{\beta \|u\|^2_H}+ e^{\beta \|\ut\|^2_H}   \big]}.
\end{equation}
Our second main result is the following theorem giving the convergence of $\pi_1\num$ toward $\nu^0$ in $\W_{\dtnb}$, which extends \cite[Theorem 5.1]{cerrai2020convergence} in dimension $d=2,3$.

\begin{theorem} \label{thm:nu^m->nu^0} 
Suppose that Assumption \ref{cond:phi:well-posed} and Assumption \ref{cond:Q} hold. Let $\num$ and $\nu^0$ respectively be the unique invariant probability measure for~\eqref{eqn:wave} and~\eqref{eqn:react-diff}. Then, for all $N>0$ large and $\beta>0$ small enough,
\begin{align*}
 \W_{\dtnb}(\pi_1\num,\nu^0) \to 0,\quad\text{as }\,\,m\to 0.
\end{align*}
In the above, $\dtnb$ is the distance defined in~\eqref{form:d.tilde_(N,beta)}.

\end{theorem}

We note that the distance $\W_{\dtnb}$ as in Theorem~\ref{thm:nu^m->nu^0} is different from the Wasserstein distance in \cite[Theorem 5.1]{cerrai2020convergence}, which was first studied in \cite{hairer2008spectral}. Here, the distance $\W_{\dtnb}$ is motivated by the work of \cite{hairer2011asymptotic} and turns out to be more convenient to explore small mass limits in our settings. The proof of Theorem~\ref{thm:nu^m->nu^0} relies on two crucial properties: of which, the first is uniform moment bounds $\num$ in $\Hcal^2$ (see~\eqref{ineq:moment-bound:nu^m:H^1+H^2} below) as well as $\sup_{[0,T]}\|\um(t)\|_{L^\infty}$ (see Lemma~\ref{lem:moment-bound:H^2:|Au|^2:sup_[0,T]:random-initial-cond}). As mentioned in the introduction, these estimates were not available in \cite{cerrai2020convergence}, hence the restriction $\lambda=1$ therein. The second key step in the proof of Theorem~\ref{thm:nu^m->nu^0} is the well--known fact that~\eqref{eqn:react-diff} possesses geometric ergodicity with respect to $\W_{\dtnb}$. That is letting $P^0_t$ be the Markov semigroup associated with~\eqref{eqn:react-diff}, for all $N$ large and $\beta$ small enough, it holds that \cite[Theorem 8.1]{glatt2022short}
\begin{align} \label{ineq:react-diff:geom}
\W_{\dtnb}\big( (P^0_t)^*\nu_1,(P^0_t)^*\nu_2   \big)\le Ce^{-ct}\W_{\dtnbhalf}\big( \nu_1,\nu_2   \big),\quad t\ge \tilde{T}=\tilde{T}(N,\beta),\quad \nu_1,\nu_2\in \Pcal r(H).
\end{align}
See Theorem~\ref{thm:react-diff:geometric-ergodicity} in Appendix \ref{sec:react-diff} for the precise statement of the above result. It is important to point out that the appearance of $\W_{\dtnbhalf}$ rather than $\W_{\dtnb}$ on the right--hand side of~\eqref{ineq:react-diff:geom} stems from the fact that $\dtnb$ does not satisfy the usual triangle inequality. All of this will be clearer in the proof of Theorem~\ref{thm:nu^m->nu^0}, which is provided in Section \ref{sec:small-mass:nu^m->nu^0}.

As a consequence of Theorem~\ref{thm:nu^m->nu^0}, together with Theorem~\ref{thm:geometric-ergodicity:mass}, we establish the validity of the approximation of the solutions of~\eqref{eqn:wave} by \eqref{eqn:react-diff} on the infinite time horizon $[0,\infty)$.

\begin{theorem} \label{thm:m->0:f}
Suppose Assumption \ref{cond:phi:well-posed} and Assumption \ref{cond:Q} hold. Given $R>0$, let $\{U_0^m=(u_0^m,v_0^m)\}_{m\in(0,1)}$ be a sequence of deterministic initial conditions such that $\|U^m_0\|_{\Hcal^2}<R$ for all $m\in(0,1)$. Then, the followings hold:

1. For all $N$ sufficiently large and $\beta$ sufficiently small, 
\begin{align} \label{lim:m->0:delta_(u^m)}
\sup_{t\ge 0}\W_{\dtnb}\big( \pi_1(P^m_t)^*\delta_{U^m_0},(P^0_t)^*\delta_{u_0^m}  \big) \to 0,\quad\text{as }\,\, m\to 0.
\end{align}
In the above, $\dtnb$ is the distance defined in~\eqref{form:d.tilde_(N,beta)}. 

2. As a consequence, for every $f\in C_b(H)$ such that $[f]_{\emph{Lip},\dtnb}<\infty$,
\begin{align} \label{lim:m->0:f(u^m)}
\sup_{t\ge 0} |\E\, f(u^m(t))-\E\, f(u^0(t))| \to 0,\quad\text{as }\,\, m\to 0.
\end{align}

\end{theorem}

The method that we employ to prove Theorem~\ref{thm:m->0:f} draws upon the framework recently developed in \cite{glatt2021mixing}. The argument is summarized as follows: we first decompose $[0,\infty)=[0,T]\cup[T,\infty)$ for some suitably chosen $T$. Then, for $t\in[0,T]$, er demonstrate that $\um(t)$ can be well approximated by $\uo(t)$, the solution of~\eqref{eqn:react-diff}. This will be derived rigorously by making use of the strategy from \cite{cerrai2020convergence} tailored to our settings. On the other hand, to control the difference between $\um(t)$ and $\uo(t)$ for $t\in[T,\infty)$, we invoke the exponential convergent rates~\eqref{ineq:geometric-ergodicity:mass} and \eqref{ineq:react-diff:geom}. Altogether, since the analysis does not depend on $m$, we conclude \eqref{lim:m->0:delta_(u^m)} by sending $m\to 0$. In turn, this implies~\eqref{lim:m->0:f(u^m)} by invoking~\eqref{ineq:W_d(nu_1,nu_2):dual}. The proof of Theorem~\ref{thm:m->0:f} will be provided in Section~\ref{sec:small-mass:f}.

\section{A priori moment estimates} \label{sec:moment-bound}

Throughout the rest of the paper, $c$ and $C$ denote generic positive constants that may change from line to line. The main parameters that they depend on will appear between parenthesis, e.g., $c(T,q)$ is a function of $T$ and $q$.

In this section, we perform useful moment bounds on the solutions $(\um(t),\vm(t))$ of~\eqref{eqn:wave}. For this purpose, we introduce the function
\begin{equation} \label{form:Psi_1m}
\Psi_{1,m}(u,v)=m\|A^{1/2}u\|^2_H+m^2\|v\|^2_H+m\la u,v\ra_H+\tfrac{1}{2}\|u\|^2_H.
\end{equation}
We start the procedure by Lemma~\ref{lem:moment-bound:H}, stated and proven next, asserting moment bounds of $(\um(t),\vm(t))$ in $\Hcal^1$.

\begin{lemma} \label{lem:moment-bound:H}
Given $(u_0,v_0)\in \Hcal^1$, let $(u^m(t),v^m(t))$ be the solution of~\eqref{eqn:wave} with initial condition $(u_0,v_0)$. Then, the followings hold for all $t\ge 0$:

1. There exist positive constants $C_{1,t}$ independent of $m$, and $(u_0,v_0)$ such that
\begin{align} 
&\E\Big[\sup_{r\in[0,t]} \Psi_{1,m}(\um(r),\vm(r))+2m\|\Phi_1(\um(r))\|_{L^1}\Big]   \nt \\
&\le C_{1,t}\big(\Psi_{1,m}(u_0,v_0)+2m\|\Phi_1(u_0)\|_{L^1}+1\big),\label{ineq:moment-bound:H:sup_[0,t]}
\end{align}
where $\Phi_1$ and $\Psionem$ are the functions as in~\eqref{cond:Phi_1} and~\eqref{form:Psi_1m}, respectively.

2.  For all $n\ge 1$, there exist positive constants $C_{1,n},c_{1,n}$ independent of $m$, $t$ and $(u_0,v_0)$ such that
\begin{align}
&\E\big[ \Psi_{1,m}(\um(t),\vm(t))+2m \|\Phi_1(\um(t))\|_{L^1} \big]^n \nt \\
 &\le C_{1,n} e^{-c_{1,n}t} \big( \Psi_{1,m}(u_0,v_0)+2m\|\Phi_1(u_0)\|_{L^1}  \big)^n +C_{1,n}.\label{ineq:moment-bound:H:n-moment}
\end{align}

3. For all $\beta$ sufficiently small independent of $m$, 
\begin{align}
&\E \exp\Big\{ \beta\big( \Psi_{1,m}(\um(t),\vm(t))+2m\|\Phi_1(\um(t))\|_{L^1}\big) \Big\} \nt  \\
&\le C_{1,0} e^{-c_{1,0}t}\exp\{\beta\big(\Psi_{0,m}(u_0,v_0)+2m\|\Phi_1(u_0)\|_{L^1}\big)\big\}+C_{1,0}, \label{ineq:moment-bound:H:exponential}
\end{align}
holds for some positive constants $C_{1,0},c_{1,0}$ independent of $m$, $t$ and $(u_0,v_0)$. 

4. For all $\beta$ sufficiently small independent of $m$, $t$ and $(u_0,v_0)$, it holds that
\begin{align}
&\E \exp \Big\{  \tfrac{1}{2}\beta \int_0^t\|A^{1/2}\um(r)\|^2_H \emph{d} r  \Big\}  \nt \\
&\le 2 \exp \Big\{ \beta\big[ \Psionem(u_0,v_0)+2m\|\Phi_1(u_0)\|_{L^1}\big]  +\beta (a_2+a_3|\domain|)t\Big\}. \label{ineq:moment-bound:H:exponential:int_0^t}
\end{align}
In the above, $a_2,a_3$ are as in Assumption~\ref{cond:phi:well-posed}, and $|\domain|$ denotes the volume of $\domain$.

5. For all $\beta$ and $\gamma$ sufficiently small independent of $m$, there exists a positive constant $C_{\beta,\gamma}$ such that
\begin{align} 
&\E \exp\big\{\beta \big[ \Psionem(\um(t),\vm(t))+2m\|\Phi_1(\um(t))\|_{L^1}\big]\big\}  \nt  \\
&\le C_{\beta,\gamma}\exp \big\{\beta\big[ \Psionem(u_0,v_0)+2m\|\Phi_1(u_0)\|_{L^1}\big] e^{-\gamma t}\big\}.\label{ineq:moment-bound:H:exponential:super-Lyapunov}
\end{align}

\end{lemma}

\begin{remark}
We note that $V_m(u,v)$ as in~\eqref{form:V_m} is slightly different from $\Psionem(u,v)+2m\|\Phi_1(u)\|_{L^1}$, which is employed in Lemma~\ref{lem:moment-bound:H} and is more convenient for the purpose of computing It\^o's formula. By the choice of $\Phi_1$ as in~\eqref{cond:Phi_1}, it is not difficult to see that
\begin{align}
cV_m(u,v)&=c \big[m\|u\|^2_{H^1}+m^2\|v\|^2_{H}+\|u\|^2_{H}+m\|u\|^{\lambda+1}_{L^{\lambda+1}}\big] \nt \\
&\le \Psionem(u,v)+2m\|\Phi_1(u)\|_{L^1} \nt \\
&\le C (  m\|u\|^2_{H^1}+m^2\|v\|^2_{H}+\|u\|^2_{H}+m\|u\|^{\lambda+1}_{H^1} +1)=C(V_m(u,v)+1) ,\label{cond:V_m<Psi_1m<V_m}
\end{align}
for some positive constants $C,c$ independent of $m$ and $(u,v)$.
\end{remark}

\begin{proof}[Proof of Lemma~\ref{lem:moment-bound:H}]
1. Let $\Phi_1$ and $\Psionem$ be the functions as in~\eqref{cond:Phi_1} and~\eqref{form:Psi_1m}, respectively. In view of the generator $\L^m$ from~\eqref{form:L^m}, a routine calculation gives
\begin{align}
\L^m\big[\Psionem(u,v)+2m\|\Phi_1(u)\|_{L^1}\big]&=-\|A^{1/2}u\|^2_H -m\|v\|^2_H+\la \f(u),u\ra_H+\Tr(QQ^*).\label{eqn:Ito:L^m.Psi_1m}
\end{align}
We invoke~\eqref{cond:phi:x.phi(x)<-x^(lambda+1)} and~\eqref{cond:Phi_1} to see that
\begin{align*}
\la \f(u),u\ra_H\le -a_2\|u\|^{\lambda+1}_{L^{\lambda+1}}+a_3|\domain|\le -\tfrac{a_2}{C_\f}\|\Phi_1(u)\|_{L^1}+a_2+a_3|\domain|,
\end{align*}
where $a_2,a_3,\lambda$ are as in Assumption~\ref{cond:phi:well-posed}, $C_\f$ is from~\eqref{cond:Phi_1} and $|\domain|$ denotes the volume of $\domain$. It follows that
\begin{align}
&\L^m\big[\Psionem(u,v)+2m\|\Phi_1(u)\|_{L^1}\big]\nt  \\
&\le -\|A^{1/2}u\|^2_H -m\|v\|^2_H-\tfrac{a_2}{C_\f}\|\Phi_1(u)\|_{L^1}+a_2+a_3|\domain|+\Tr(QQ^*). \label{ineq:L^m.Psi_1m}
\end{align}
By It\^o's formula, we obtain
\begin{align}
&\d\big[ \Psionem(\um(t),\vm(t))+2m\|\Phi_1(\um(t))\|_{L^1}\big] \nt \\
&= \L^m\big[\Psionem(\um(t),\vm(t))+2m\|\Phi_1(\um(t))\|_{L^1}\big]\d t+\la 2m \vm(t)+\um(t),Q\d w(t)\ra_H \label{eqn:Ito:d.Psi_1m} \\
&\le -\|A^{1/2}\um(t)\|^2_H\d t -m\|\vm(t)\|^2_H\d t  -\tfrac{a_2}{C_\f}\|\Phi_1(\um(t))\|_{L^1}\d t \nt \\
&\qquad+\Big(a_2+a_3|\domain|+\Tr(QQ^*)\Big)\d t +\la 2m \vm(t)+\um(t),Q\d w(t)\ra_H.\nt
\end{align}
As a consequence, integrating the above estimate with respect to time yields
\begin{align*}
&\Psionem(\um(t),\vm(t))+2m\|\Phi_1(\um(t))\|_{L^1}\\
&\le \Psionem(u_0,v_0)+2m\|\Phi_1(u_0)\|_{L^1}+\Big(a_2+a_3|\domain|+\Tr(QQ^*)\Big)t\\
&\qquad+\int_0^t \la 2m \vm(r)+\um(r),Q\d w(r)\ra_H.
\end{align*}
Now, we invoke Burkholder's inequality to estimate the Martingale term on the above right-hand side
\begin{align*}
\E\sup_{r\in[0,t]}\Big|\int_0^r \la 2m \vm(s)+\um(s),Q\d w(s)\ra_H\Big|^2
& \le \E\int_0^t \|Q\big(2m\vm(r)+\um(r)\big)\|^2_H\d r\\
&\le \Tr(QQ^*)\int_0^t \E\|2m\vm(r)+\um(r)\|^2_H\d r\\
&\le c\int_0^t\E \Psionem(\um(r),\vm(r))+2m\|\Phi_1(\um(r))\|_{L^1}\d r.
\end{align*}
As a consequence, 
\begin{align*}
&\E\Big[\sup_{r\in[0,t]}\Psionem(\um(r),\vm(r))+2m\|\Phi_1(\um(r))\|_{L^1}\Big]\\
&\le \Psionem(u_0,v_0)+2m\|\Phi_1(u_0)\|_{L^1}+\Big(a_2+a_3|\domain|+\Tr(QQ^*)\Big)t\\
&\qquad+C\int_0^t\E\big[ \Psionem(\um(r),\vm(r))+2m\|\Phi_1(\um(r))\|_{L^1}\big]\d r+C.
\end{align*}
By Gronwall's inequality, this produces~\eqref{ineq:moment-bound:H:sup_[0,t]}, as claimed.

2. We proceed to prove~\eqref{ineq:moment-bound:H:n-moment} by induction on $n$. For the base case $n=1$, from~\eqref{eqn:Ito:d.Psi_1m}, we deduce the following bound in expectation
\begin{align}
&\tfrac{\d}{\d t}\E\big[ \Psionem(\um(t),\vm(t))+2m\|\Phi_1(\um(t))\|_{L^1}\big]   \nt \\
&\le  -\E\big[\|A^{1/2}\um(t)\|^2_H+m\|\vm(t)\|^2_H+\tfrac{a_2}{C_\f}\|\Phi_1(\um(t))\|_{L^1}\big]   \nt   \\
&\qquad+\Big(a_2+a_3|\domain|+\Tr(QQ^*)\Big). \label{ineq:d.E.Psi_1m}
\end{align}
Recall $\Psionem$ from~\eqref{form:Psi_1m}, we deduce ($m\le 1$)
\begin{align*}
&\E\Big[ \Psionem(\um(t),\vm(t))+2m\|\Phi_1(\um(t))\|_{L^1}\Big]\\
&\le Ce^{-ct}\Big[ \Psionem(u_0,v_0)+2m\|\Phi_1(u_0)\|_{L^1}\Big]+C.
\end{align*}
This proves~\eqref{ineq:moment-bound:H:n-moment} for $n=1$. 

Next, assuming~\eqref{ineq:moment-bound:H:n-moment} holds for $n-1$, we consider the case $n\ge 2$. The partial derivatives of $(\Psionem(u,v)+2m\|\Phi_1(u)\|_{L^1})^n$ along a direction $\xi\in \Hcal^1$ are given by
\begin{align*}
&\la D_u\big[\Psionem(u,v)+2m\|\Phi_1(u)\|_{L^1}\big]^n,\pi_1\xi\ra_H\\ &= n\big[\Psionem(u,v)+2m\|\Phi_1(u)\|_{L^1}\big]^{n-1}\big( 2m\la A^{1/2}u,A^{1/2}\pi_1 \xi   \ra_H+ m\la v,\pi_1\xi  \ra_H\\
&\qquad\qquad\qquad\qquad+ \la u,\pi_1\xi\ra_H -2m \la \f(u),\pi_1\xi\ra_H\big),\\
&\la D_v\big[\Psionem(u,v)+2m\|\Phi_1(u)\|_{L^1}\big]^n,\pi_2\xi\ra_H\\
&= n\big[\Psionem(u,v)+2m\|\Phi_1(u)\|_{L^1}\big]^{n-1}\la 2m^2 v+mu,\pi_2\xi\ra_H,
\end{align*}
and
\begin{align*}
& D_{vv}\big[\Psionem(u,v)+2m\|\Phi_1(u)\|_{L^1}\big]^n(\xi)\\
&=n\big[\Psionem(u,v)+2m\|\Phi_1(u)\|_{L^1}\big]^{n-1}2m^2\pi_2\xi\\
&\qquad +n(n-1) \big[\Psionem(u,v)+2m\|\Phi_1(u)\|_{L^1}\big]^{n-2}\la 2m^2 v+mu,\pi_2\xi\ra_H\cdot  (2m^2 v+mu).
\end{align*}
Applying $\L^m$ from~\eqref{form:L^m} to $\big[\Psionem(u,v)+2m\|\Phi_1(u)\|_{L^1}\big]^n$ gives
\begin{align}
&\L^m\big[\Psionem(u,v)+2m\|\Phi_1(u)\|_{L^1}\big]^n\notag \\
&=n\big[\Psionem(u,v)+2m\|\Phi_1(u)\|_{L^1}\big]^{n-1}\big(-\|A^{1/2}u\|^2_H -m\|v\|^2_H+\la \f(u),u\ra_H +\Tr(QQ^*) \big) \nt \\
&\qquad+\tfrac{1}{2} n(n-1) \big[\Psionem(u,v)+2m\|\Phi_1(u)\|_{L^1}\big]^{n-2}\|Q( 2m v+u)\|_H^2.\label{eqn:Ito:L^m.(Psi_1m)^n}
\end{align}
To estimate the last term on the above right--hand side, we recall that $\Tr(QQ^*)<\infty$ (see condition~\eqref{cond:Q:Tr(QA^3Q)}) implying
\begin{align*}
\|Q( 2m v+u)\|_H^2 \le \Tr(QQ^*)\|2mv+u\|^2_H&\le 8 \Tr(QQ^*)(m^2\|v\|_H^2+\tfrac{1}{2}\|u\|^2_H)\\
&\le c\big[\Psionem(u,v)+2m\|\Phi_1(u)\|_{L^1}\big],
\end{align*}
for some positive constant $c$ independent of $m$. It follows that 
\begin{align*}
&\tfrac{1}{2} n(n-1) \big[\Psionem(u,v)+2m\|\Phi_1(u)\|_{L^1}\big]^{n-2}\|Q( 2m v+u)\|_H^2\\
&\le c \big[\Psionem(u,v)+2m\|\Phi_1(u)\|_{L^1}\big]^{n-1}.
\end{align*}
To estimate the first term on the right--hand side of~\eqref{eqn:Ito:L^m.(Psi_1m)^n}, we employ the same bound from the base case $n=1$ to deduce (recalling $m\le 1$)
\begin{align*}
&n\big[\Psionem(u,v)+2m\|\Phi_1(u)\|_{L^1}\big]^{n-1}\big(-\|A^{1/2}u\|^2_H -m\|v\|^2_H+\la \f(u),u\ra_H  \big)\\
&\le  - c \big[\Psionem(u,v)+2m\|\Phi_1(u)\|_{L^1}\big]^{n}+C\big[\Psionem(u,v)+2m\|\Phi_1(u)\|_{L^1}\big]^{n-1} ,
\end{align*}
for some positive constants $c=c(n), C=C(n)$ independent of $m$. As a consequence, from~\eqref{eqn:Ito:L^m.(Psi_1m)^n}, we obtain
\begin{align}
&\L^m\big[\Psionem(u,v)+2m\|\Phi_1(u)\|_{L^1}\big]^n\notag \\
&\le -c \big[\Psionem(u,v)+2m\|\Phi_1(u)\|_{L^1}\big]^n+C\big[\Psionem(u,v)+2m\|\Phi_1(u)\|_{L^1}\big]^{n-1} \nt  \\
&\le -c \big[\Psionem(u,v)+2m\|\Phi_1(u)\|_{L^1}\big]^n +C. \label{ineq:L^m.(Psi_1m)^n}
\end{align}
In the last estimate above, we employed Young's inequality to subsume the lower order term into $-c\big[\Psionem(u,v)+2m\|\Phi_1(u)\|_{L^1}\big]^n$. 

Turning back to~\eqref{ineq:moment-bound:H:n-moment}, we invoke It\^o's formula to see that
\begin{align*}
&\d \big[\Psionem(\um(t),\vm(t))+2m\|\Phi_1(\um(t))\|_{L^1}\big]^n\\
&= \L^m \big[\Psionem(\um(t),\vm(t))+2m\|\Phi_1(\um(t))\|_{L^1}\big]^n\d t+ \d M_n(t),
\end{align*}
where the semi-Martingale term $M_n(t)$ is given by
\begin{align*}
M_n(t)=\int_0^t n\big[\Psionem(\um(r),\vm(r))+2m\|\Phi_1(\um(r))\|_{L^1}\big]^{n-1}\la 2m^2 \vm(r)+m\um(r),Q\d w(r)\ra_H.
\end{align*}
In view of~\eqref{ineq:L^m.(Psi_1m)^n}, we obtain the bound in expectation
\begin{align*}
&\tfrac{\d}{\d t}\E [\Psionem(\um(t),\vm(t))+2m\|\Phi_1(\um(t))\|_{L^1}\big]^n \\
&\le -c \E[\Psionem(\um(t),\vm(t))+2m\|\Phi_1(\um(t))\|_{L^1}\big]^n+C.
\end{align*}
In the above, we emphasize that $c,C$ do not depend on $m$. This establishes~\eqref{ineq:moment-bound:H:n-moment} for the general case $n\ge 2$.

3. The proof of~\eqref{ineq:moment-bound:H:exponential} follows along the lines of \cite[Section 9]{cerrai2020convergence} and is included here for the sake of completeness.

For $\beta>0$, denote 
\begin{align*}
G_1(u,v)= \exp\big\{ \beta\big(\Psionem(u,v)+2m\|\Phi_1(u)\|_{L^1} \big)  \big\}.
\end{align*}
In view of~\eqref{eqn:Ito:L^m.Psi_1m} and~\eqref{eqn:Ito:d.Psi_1m}, we see that
\begin{align*}
\L^mG_1(u,v)&=\beta G_1(u,v)\big(   -\|A^{1/2}u\|^2_H -m\|v\|^2_H+\la \f(u),u\ra_H+\Tr(QQ^*)    \big)\\
&\qquad+ \tfrac{1}{2}\beta^2G_1(u,v)\|Q( 2m v+u)\|_H^2.
\end{align*}
Similarly the part 2,
\begin{align*}
\|Q( 2m v+u)\|_H^2 
&\le \Tr(QQ^*) \|2mv+u\|^2_H\le 2\Tr(QQ^*)(4m^2\|v\|^2_H+\|u\|_H^2).
\end{align*}
Together with~\eqref{ineq:L^m.Psi_1m}, we infer the existence of $c,C$ independent of $m$ and $\beta$ such that
\begin{align*}
\L^m G_1(u,v)\le \beta G_1(u,v)\big[ -c\big(\Psionem(u,v)+2m\|\Phi_1(u)\|_{L^1} \big) +C\\
&\hspace{-4cm} +\beta\Tr(QQ^*)(4m^2\|v\|^2_H+\|u\|_H^2)\big]. 
\end{align*}
Recalling $\Psionem$ as in~\eqref{form:Psi_1m} and taking $\beta$ sufficiently small independent of $m$, we deduce
\begin{align*}
\L^m G_1(u,v) \le \beta G_1(u,v) \big[ -c\big(\Psionem(u,v)+2m\|\Phi_1(u)\|_{L^1} \big) +C    \big].
\end{align*}
We now invoke the elementary inequality
\begin{align*}
e^{\beta r}(cr-C)\ge \tilde{c} e^{\beta r} -\tilde{C}, \quad r\ge 0,
\end{align*}
and obtain
\begin{align*}
\L^m G_1(u,v) \le - c G_1(u,v)+C,
\end{align*}
for some positive constants $c=c(\beta), C=C(\beta)$ independent of $m$. Taking expectation from it\^o' formula, this also implies
\begin{align*}
\tfrac{\d}{\d t}\E G_1(\um(t),\vm(t)) \le -c G_1(\um(t),\vm(t))+C.
\end{align*}
This establishes~\eqref{ineq:moment-bound:H:exponential}, as claimed.

4. Concerning the exponential bound in~\eqref{ineq:moment-bound:H:exponential:int_0^t}, for $\beta>0$, recall from~\eqref{eqn:Ito:d.Psi_1m} that
\begin{align*}
& \beta\big[\Psionem(\um(t),\vm(t))+2m\|\Phi_1(\um(t))\|_{L^1}\big]-\beta\big[ \Psionem(u_0,v_0)+2m\|\Phi_1(u_0)\|_{L^1}\big]\\
&=\beta\int_0^t\L^m\big[\Psionem(\um(r),\vm(r))+2m\|\Phi_1(\um(r))\|_{L^1}\big]\d r+ M(t),
\end{align*}
where
\begin{align*}
M(t)=\beta\int_0^t\la 2m \vm(r)+\um(r),Q\d w(r)\ra_H,
\end{align*}
whose quadratic variation process is given by
\begin{align*}
\la M\ra (t) =\beta^2 \int_0^t\|Q(2m \vm(r)+\um(r))\|_H^2\d r.
\end{align*}
Note also that
\begin{align*}
\d \la M\ra (t) \le 2\beta^2 \Tr(QQ^*)(4m^2\|\vm(t)\|^2_H+\|\um(t)\|^2_H).
\end{align*}
On the other hand, in view of~\eqref{ineq:L^m.Psi_1m}, it holds that
\begin{align*}
&\beta \int_0^t \L^m\big[\Psionem(\um(r),\vm(r))+2m\|\Phi_1(\um(r))\|_{L^1}\big]\d r\\
&\le  - \beta \int_0^t \|A^{1/2}\um(r)\|^2_H+m\|\vm(r)\|_H^2 +\tfrac{a_2}{C_\f}\|\Phi_1(\um(r))\|_{L^1}   \d r+\beta (a_2+a_3|\domain|)t\\
&= - \beta \int_0^t \tfrac{1}{2}\|A^{1/2}\um(r)\|^2_H+m(1-2m)\|\vm(r)\|_H^2 +\tfrac{a_2}{C_\f}\|\Phi_1(\um(r))\|_{L^1}   \d r+\beta (a_2+a_3|\domain|)t\\
&\qquad -\tfrac{1}{2}\cdot \tfrac{1}{\beta\Tr(QQ^*)}\cdot \beta^2\Tr(QQ^*)\int_0^t \|A^{1/2}\um(r)\|^2_H+4m^2\|\vm(r)\|_H^2  \d r.
\end{align*}
Thanks to the fact that $\Tr(QQ^*)<\infty$, we use Sobolev embedding to further deduce
\begin{align*}
&\beta^2\Tr(QQ^*)\int_0^t \|A^{1/2}\um(r)\|^2_H+4m^2\|\vm(r)\|_H^2  \d r \\
&\ge  \beta^2\Tr(QQ^*)\int_0^t \alpha_1 \|\um(r)\|^2_H+4m^2\|\vm(r)\|_H^2  \d r \\
&\ge \tfrac{1}{2}\min\{\alpha_1,1\} \la M\ra (t).
\end{align*}
In the above, $\alpha_1>0$ is the first eigenvalue of $A$ as in~\eqref{eqn:Ae_k=alpha_k.e_k}. As a consequence, we obtain a.s.
\begin{align*}
& \beta\big[\Psionem(\um(t),\vm(t))+2m\|\Phi_1(\um(t))\|_{L^1}\big]-\beta\big[ \Psionem(u_0,v_0)+2m\|\Phi_1(u_0)\|_{L^1}\big]\\
&\le - \beta \int_0^t \tfrac{1}{2}\|A^{1/2}\um(r)\|^2_H+m(1-2m)\|\vm(r)\|_H^2 +\tfrac{a_2}{C_\f}\|\Phi_1(\um(r))\|_{L^1}   \d r\\ &\qquad +M(t)  -\tfrac{1}{2}\cdot \tfrac{\min\{\alpha_1,1\}}{2\beta\Tr(QQ^*)}\cdot \la M\ra(t)+\beta (a_2+a_3|\domain|)t.
\end{align*}
Now, applying the exponential Martingale inequality to $M(t)$ gives
\begin{align*}
\P\Big( \sup_{t\ge 0}\Big[ M(t)  -\tfrac{1}{2}\cdot \tfrac{\min\{\alpha_1,1\}}{2\beta\Tr(QQ^*)}\cdot \la M\ra(t)\Big] \ge R \Big)\le e^{-\frac{\min\{\alpha_1,1\}}{2\beta\Tr(QQ^*)}R},\quad R>0.
\end{align*}
In particular, taking $\beta$ sufficiently small, e.g., 
\begin{align*}
\beta<\frac{\min\{\alpha_1,1\}}{4\Tr(QQ^*)},
\end{align*}
we infer
\begin{align*}
\P\Big( \sup_{t\ge 0}\Big[ M(t)  -\tfrac{1}{2}\cdot \tfrac{\min\{\alpha_1,1\}}{2\beta\Tr(QQ^*)}\cdot \la M\ra(t)\Big] \ge R \Big)\le e^{-2R},\quad R>0,
\end{align*}
whence
\begin{align*}
\E \exp\Big\{ \sup_{t\ge 0}\Big[ M(t)  -\tfrac{1}{2}\cdot \tfrac{\min\{\alpha_1,1\}}{2\beta\Tr(QQ^*)}\cdot \la M\ra(t)\Big]\Big\}<2.
\end{align*}
Altogether, we arrive at the bounds
\begin{align*}
&\E \exp \Big\{\beta\big[\Psionem(\um(t),\vm(t))+2m\|\Phi_1(\um(t))\|_{L^1}\big]\\
&\qquad\qquad+   \beta \int_0^t \tfrac{1}{2}\|A^{1/2}\um(r)\|^2_H+m(1-2m)\|\vm(r)\|_H^2 +\tfrac{a_2}{C_\f}\|\Phi_1(\um(r))\|_{L^1}   \d r  \Big\} \\
&\le 2 \exp \Big\{ \beta\big[ \Psionem(u_0,v_0)+2m\|\Phi_1(u_0)\|_{L^1}\big]  +\beta (a_2+a_3|\domain|)t\Big\}.
\end{align*}
This produces~\eqref{ineq:moment-bound:H:exponential:int_0^t}, as claimed.

5. Turning to~\eqref{ineq:moment-bound:H:exponential:super-Lyapunov}, let $\gamma\in(0,1)$ and $\beta>0$ be given and be chosen later. Fixing $T>0$, from~\eqref{eqn:Ito:d.Psi_1m}, we have
\begin{align*}
&\d \Big( \beta e^{\gamma(t-T)} \big[ \Psionem(\um(t),\vm(t))+2m\|\Phi_1(\um(t))\|_{L^1}\big]  \Big)\\
&= \beta \gamma e^{\gamma(t-T)} \big[ \Psionem(\um(t),\vm(t))+2m\|\Phi_1(\um(t))\|_{L^1}\big]\d t \\
&\qquad +\beta e^{\gamma(t-T)}\L^m \big[ \Psionem(\um(t),\vm(t))+2m\|\Phi_1(\um(t))\|_{L^1}\big]\d t \\
&\qquad+\beta e^{\gamma(t-T)}\la 2m\vm(t)+\um(t),Q\d w(t)\ra_H.
\end{align*}
Reasoning as in part 4 above, we see that for $t\in[0,T]$
\begin{align*}
&\beta e^{\gamma(t-T)}\L^m \big[ \Psionem(\um(t),\vm(t))+2m\|\Phi_1(\um(t))\|_{L^1}\big]\\
&\le  - \beta  e^{\gamma(t-T)} \big(\tfrac{1}{2}\|A^{1/2}\um(t)\|^2_H+m(1-2m)\|\vm(t)\|_H^2 +\tfrac{a_2}{C_\f}\|\Phi_1(\um(t))\|_{L^1}\big)\\
&\qquad -\tfrac{1}{2}\cdot \tfrac{\min\{\alpha_1,1\}}{2\beta\Tr(QQ^*)}\cdot \beta^2  e^{2\gamma(t-T)}\|Q(2m \vm(t)+\um(t))\|_H^2  +\beta  e^{\gamma(t-T)}(a_2+a_3|\domain|).
\end{align*}
So, for $\gamma$ sufficiently small, we obtain
\begin{align*}
&\d \Big( \beta e^{\gamma(t-T)} \big[ \Psionem(\um(t),\vm(t))+2m\|\Phi_1(\um(t))\|_{L^1}\big]  \Big)\\
&\le \beta  e^{\gamma(t-T)}(a_2+a_3|\domain|)+ \beta e^{\gamma(t-T)}\la 2m\vm(t)+\um(t),Q\d w(t)\ra_H\\
&\qquad-\tfrac{1}{2}\cdot \tfrac{\min\{\alpha_1,1\}}{2\beta\Tr(QQ^*)}\cdot \beta^2  e^{2\gamma(t-T)}\|Q(2m \vm(t)+\um(t))\|_H^2.
\end{align*}
Integrating the above inequality with respect to $t\in[0,T]$ yields
\begin{align*}
&\beta \big[ \Psionem(\um(T),\vm(T))+2m\|\Phi_1(\um(T))\|_{L^1}\big] - \beta e^{-\gamma T} \big[ \Psionem(u_0,v_0)+2m\|\Phi_1(u_0)\|_{L^1}\big] \\
&\le \tfrac{\beta}{\gamma} +M_1(T)-\tfrac{1}{2}\cdot \tfrac{\min\{\alpha_1,1\}}{2\beta\Tr(QQ^*)}\la M_1\ra(T),
\end{align*}
where
\begin{align*}
M_1(t)=\int_0^t \beta e^{\gamma(t-T)}\la 2m\vm(r)+\um(r),Q\d w(r)\ra_H.
\end{align*}
Similarly to part 4, applying the exponential Martingale inequality to $M_1$ for $\beta<\frac{\min\{\alpha_1,1\}}{4\Tr(QQ^*)}$, we arrive at
\begin{align*}
&\E \exp \Big\{  \beta \big[ \Psionem(\um(T),\vm(T))+2m\|\Phi_1(\um(T))\|_{L^1}\big]  \Big\}\\
& \le 2e^{\beta/\gamma}\exp\Big\{ \beta e^{-\gamma T} \big[ \Psionem(u_0,v_0)+2m\|\Phi_1(u_0)\|_{L^1}\big]\Big\}.
\end{align*}
This establishes~\eqref{ineq:moment-bound:H:exponential:super-Lyapunov} with $C_{\beta,\gamma}=2e^{\beta/\gamma}$, thereby finishing the proof.
\end{proof}

As a consequence of Lemma~\ref{lem:moment-bound:H}, in Lemma~\ref{lem:moment-bound:H:random-initial-cond} below, we obtain moment bounds in $\Hcal^1$ with respect to random initial conditions. In particular, Lemma~\ref{lem:moment-bound:H:random-initial-cond} will be invoked to prove the small mass limits in Section~\ref{sec:small-mass}.
\begin{lemma} \label{lem:moment-bound:H:random-initial-cond} 
Let $(u_0,v_0)$ be random variable in $L^2(\Omega;\Hcal^1)$ and $(u^m(t),v^m(t))$ be the solution of~\eqref{eqn:wave} with initial condition $(u_0,v_0)$. Then, the following holds:

1. For all $\beta$ sufficiently small independent of $m$,
\begin{align}
&\E \exp\Big\{ \beta\big( \Psi_{1,m}(\um(t),\vm(t))+2m\|\Phi_1(\um(t))\|_{L^1}\big) \Big\} \nt  \\
&\le C_{1,0} e^{-c_{1,0}t}\E\exp\{\beta\big(\Psi_{0,m}(u_0,v_0)+2m\|\Phi_1(u_0)\|_{L^1}\big)\big\}+C_{1,0},\quad t\ge 0 \label{ineq:moment-bound:H:exponential:random-initial-cond}
\end{align}
where $C_{1,0}$ and $c_{1,0}$ are as in~\eqref{ineq:moment-bound:H:exponential}.

2. For all $\beta$ sufficiently small independent of $m$,
 \begin{align}
&\E \exp\Big\{ \sup_{t\in[0,T]}\tfrac{1}{2}\beta\big[ \Psi_{1,m}(\um(t),\vm(t))+2m\|\Phi_1(\um(t))\|_{L^1}\big] \Big\} \nt  \\
&\le C_{\beta,T}\E\exp\{\beta\big[\Psi_{0,m}(u_0,v_0)+2m\|\Phi_1(u_0)\|_{L^1}\big]\big\},\quad t\ge 0, \label{ineq:moment-bound:H:exponential:random-initial-cond:sup_[0,T]}
\end{align}
for some positive constant $C_{\beta,T}$ independent of $m$ and $(u_0,v_0)$.

\end{lemma}
\begin{proof}
1. The proof of~\eqref{ineq:moment-bound:H:exponential:random-initial-cond} is similar to that of~\eqref{ineq:moment-bound:H:exponential} in Lemma~\ref{lem:moment-bound:H}, part 3, and thus is omitted.

2. Concerning~\eqref{ineq:moment-bound:H:exponential:random-initial-cond:sup_[0,T]}, we employ the same argument of~\eqref{ineq:moment-bound:H:exponential:int_0^t} (see the proof of Lemma~\ref{lem:moment-bound:H}, part 4) to infer
\begin{align*}
&\E \exp \Big\{\sup_{t\in[0,T]}\Big(\beta\big[\Psionem(\um(t),\vm(t))+2m\|\Phi_1(\um(t))\|_{L^1}\big]\Big)\\
&\qquad\qquad -\beta\big[ \Psionem(u_0,v_0)+2m\|\Phi_1(u_0)\|_{L^1}\big]  -\beta (a_2+a_3|\domain|)T\Big\} \le 2 .
\end{align*}
By Holder's inequality, we obtain
\begin{align*}
&\Big|\E \exp \Big\{\sup_{t\in[0,T]}\Big(\tfrac{1}{2}\beta\big[\Psionem(\um(t),\vm(t))+2m\|\Phi_1(\um(t))\|_{L^1}\big]\Big)\Big\}\Big|^2\\
&\le 2 \exp\Big\{\beta (a_2+a_3|\domain|)T\Big\}\E\exp\Big\{\beta\big[ \Psionem(u_0,v_0)+2m\|\Phi_1(u_0)\|_{L^1}\big] \Big\}.
\end{align*}
Since the expectation on the above right--hand side is always greater than one, we establish~\eqref{ineq:moment-bound:H:exponential:random-initial-cond:sup_[0,T]}, as claimed.

\end{proof}

Next, we establish a $\Hcal^2-$analogue of Lemma~\ref{lem:moment-bound:H} provided the initial data $(u_0,v_0)\in\Hcal^2$. For this purpose, we introduce the function
\begin{equation} \label{form:Psi_2m}
\Psi_{2,m}(u,v)=m\|Au\|^2_H+m^2\|A^{1/2}v\|^2_H+m\la A^{1/2}u,A^{1/2}v\ra_H+\tfrac{1}{2}\|A^{1/2}u\|^2_H.
\end{equation}
The moment bounds in $\Hcal^2$ are given in the following lemma:

\begin{lemma} \label{lem:moment-bound:H^2}
Given $(u_0,v_0)\in \Hcal^2$, let $(u^m(t),v^m(t))$ be the solution of~\eqref{eqn:wave} with initial condition $(u_0,v_0)$. Then, the followings hold for all $t\ge 0$:

1. For all $n\ge 1$, there exist positive constants $C_{2,n},c_{2,n}, q_{n}$ independent of $m$, $t$ and $(u_0,v_0)$ such that
\begin{align}
&\E\, \Psitwom(\um(t),\vm(t))^n \nt \\
& \le C_{2,n} e^{-c_{2,n}t}  \Big(\Psi_{2,m}(u_0,v_0)^n + \big[\Psionem(u_0,v_0)+2m\|\Phi_1(u_0)\|_{L^1}\big]^{q_{2,n}} \Big) +C_{2,n},\label{ineq:moment-bound:H^2:n-moment}
\end{align}
where $\Psitwom$, $\Psionem$ and $\Phi_1$ are the functions as in~\eqref{form:Psi_2m}, \eqref{form:Psi_1m} and~\eqref{cond:Phi_1}, respectively.

2. There exists a positive constant $C_{2,t}$ independent of $m$ and $(u_0,v_0)$ such that
\begin{align} 
&\E\Big[\sup_{r\in[0,t]} \Psitwom(\um(r),\vm(r))\Big]   \nt \\
&\le C_{2,t}\Big(\Psitwom(u_0,v_0)+   \big(\Psi_{1,m}(u_0,v_0)+2m\|\Phi_1(u_0)\|_{L^1}\big)^{q_{2,1}}+1\Big),\label{ineq:moment-bound:H^2:sup_[0,t]}
\end{align}
where $q_{2,1}$ is the constant as in~\eqref{ineq:moment-bound:H^2:n-moment} with $n=1$.
\end{lemma}

\begin{proof} 1. Similarly to the proof of~\eqref{ineq:moment-bound:H:n-moment}, we proceed by induction on $n$ . For the base case $n=1$, we apply $\L^m$ to $\Psitwom$ to obtain
\begin{align}
\L^m \Psitwom(u,v) &=-\|Au\|_H^2 -m\|A^{1/2}v\|^2_H+\Tr(QAQ^*) \nonumber \\
&\qquad+m\la \f'(u)\grad u,\grad v\ra_H+\la \f'(u)\grad u,\grad u\ra_H. \label{eqn:Ito:L^m.Psi_2m}
\end{align}
Recall condition~\eqref{cond:Q:Tr(QA^3Q)}, we readily have $\Tr(QAQ^*)<\infty$. To estimate the last term on the right-hand side of~\eqref{eqn:Ito:L^m.Psi_2m}, we invoke condition~\eqref{cond:phi:sup.phi'<a_f} with Cauchy-Schwarz inequality to see that
\begin{align} \label{ineq:phi'.grad.u.grad.u}
\la \f'(u)\grad u,\grad u\ra_H \le a_\f \la \grad u,\grad u\ra_H
& = a_\f \la u, Au\ra_H \le a_\f^2\|u\|^2_H+\tfrac{1}{4}\|Au\|^2_H.
\end{align}
Concerning the nonlinear term $m\la \f'(u)\grad u,\grad v\ra_H$,
we invoke~\eqref{cond:phi:phi'=O(x^(lambda-1))} and Cauchy--Schwarz inequality to see that
\begin{align*}
m\la \f'(u)\grad u,\grad v\ra_H & \le c\,m\, \big(1+\|u\|_{L^\infty}^{\lambda-1}\big)\|A^{1/2}u\|_H\|A^{1/2}v\|_H\\
&\le c\,m(1+\|u\|_{L^\infty}^{2\lambda-2}\big)\|A^{1/2}u\|_H^2+\tfrac{1}{2}m\|A^{1/2}v\|_H^2\\
&= c\, m\|A^{1/2}u\|_H^2+\tfrac{1}{2}m\|A^{1/2}v\|_H^2+c\, m\|u\|_{L^\infty}^{2\lambda-2}\|A^{1/2}u\|_H^2\\
&\le \tfrac{1}{8}\|Au\|_H^2+\tfrac{1}{2}m\|A^{1/2}v\|_H^2+c\, m\|u\|_{L^\infty}^{2\lambda-2}\|A^{1/2}u\|_H^2,
\end{align*}
which holds for all $m$ sufficiently small. To further bound the last term on the above right-hand side, we employ Agmon's, Young's and interpolation inequalities and obtain (in dimension $d=2$ or $d=3$)
\begin{align*}
m\|u\|_{L^\infty}^{2\lambda-2}\|A^{1/2}u\|_H^2 &\le c\,m\|A^{1/2}u\|_H^{\lambda-1}\|Au\|_H^{\lambda-1}\|u\|_H\|Au\|_H\\
&= c\|u\|_H\cdot m\|A^{1/2}u\|^{\lambda-1}\cdot\|Au\|^{\lambda}\\
&\le C \|u\|^{\frac{4}{2-\lambda}}_H+C\big(m\|A^{1/2}u\|_H^{\lambda-1} \big)^{\frac{4}{2-\lambda}}+\tfrac{1}{8}\|Au\|_H^2\\
&\le C \|u\|^{\frac{4}{2-\lambda}}_H+C\big(m\|A^{1/2}u\|_H^2 \big)^{\frac{2(\lambda-1)}{2-\lambda}}+\tfrac{1}{8}\|Au\|_H^2.
\end{align*}
In the above, we emphasize that $c$ and $C$ do not depend on $m$. As a consequence, we arrive at
\begin{align}
&m\la \f'(u)\grad u,\grad v\ra_H\nt \\
 &\le \tfrac{1}{4}\|Au\|_H^2+\tfrac{1}{2}m\|A^{1/2}v\|_H^2+C \|u\|^{\frac{4}{2-\lambda}}_H+C\big(m\|A^{1/2}u\|_H^{2} \big)^{\frac{2(\lambda-1)}{2-\lambda}} \nt \\
 &\le \tfrac{1}{4}\|Au\|_H^2+\tfrac{1}{2}m\|A^{1/2}v\|_H^2+\|u\|^{q}_H+\big(m\|A^{1/2}u\|_H^{2} \big)^{q}+C, \label{ineq:phi'.grad.u.grad.v:lambda<2}
\end{align}
for some positive constants $q=q(\lambda)\in\nbb$ and $C=C(\lambda)$ sufficiently large independent of $m$. 

Next, we collect~\eqref{eqn:Ito:L^m.Psi_2m}--\eqref{ineq:phi'.grad.u.grad.u}--\eqref{ineq:phi'.grad.u.grad.v:lambda<2} to deduce the following estimate
\begin{align*} 
\L^m\Psitwom(u,v)\le -\tfrac{1}{2}\|Au\|^2_H-\tfrac{1}{2}m\|A^{1/2}v\|^2_H+\|u\|^q_H+(m\|A^{1/2}u\|^2)^q+C.
\end{align*}
Recall the function $\Psionem$ as in~\eqref{form:Psi_1m}, we see that
\begin{align*}
\|u\|^q_H+(m\|A^{1/2}u\|^2)^q \le C\Psionem(u,v)^{\tilde{q}}+C,
\end{align*}
for a possibly different exponent $\tilde{q}$ larger than $q$. Altogether, we arrive at
\begin{align} \label{ineq:L^m.Psi_2m:lambda<2}
\L^m\Psitwom(u,v)\le -\tfrac{1}{2}\|Au\|^2_H-\tfrac{1}{2}m\|A^{1/2}v\|^2_H+C\Psionem(u,v)^q+C.
\end{align}
In particular, this also implies the bound in expectation
\begin{align}
&\tfrac{\d}{\d t}\E \Psitwom(\um(t),\vm(t)) \nt \\
&\le -\tfrac{1}{2}\E\|A\um(t)\|^2_H-\tfrac{1}{2}m\|A^{1/2}\vm(t)\|^2_H+C\E\Psionem(\um(t),\vm(t))^q+C, \label{ineq:d.E.Psi_2m}
\end{align}
whence
\begin{align*}
\E\Psitwom(\um(t),\vm(t))&\le e^{-ct}\Psitwom(u_0,v_0)+C\\
&\qquad+C\int_0^t e^{-c(t-r)}\E\Psionem(\um(r),\vm(r))^q\d r.
\end{align*}
To bound the integral on the above right-hand side, we recall Lemma~\ref{lem:moment-bound:H}, cf.~\eqref{ineq:moment-bound:H:n-moment}, that
\begin{align*}
\E\Psionem(\um(r),\vm(r))^q \le C_{1,q}e^{-c_{1,q}\,r}\big[\Psionem(u_0,v_0)+2m\|\Phi_1(u_0)\|_{L^1}\big]^q+C_{1,q},
\end{align*}
where $-\Phi_1$ is the antiderivative of $\f$ as in~\eqref{cond:Phi_1}. It follows that
\begin{align*}
\E\Psitwom(\um(t),\vm(t))&\le e^{-ct}\Psitwom(u_0,v_0)+C\\
&\qquad+C\int_0^t e^{-c(t-r)}\Big[C_{1,q}e^{-c_{1,q}\,r}\big[\Psionem(u_0,v_0)+2m\|\Phi_1(u_0)\|_{L^1}\big]^q+C_{1,q}\Big]\d r\\
&\le  e^{-ct}\Psitwom(u_0,v_0)+C e^{-\tilde{c}t}\big[\Psionem(u_0,v_0)+2m\|\Phi_1(u_0)\|_{L^1}\big]^q+C\\
&\le Ce^{-ct}\Big(\Psitwom(u_0,v_0)+\big[\Psionem(u_0,v_0)+2m\|\Phi_1(u_0)\|_{L^1}\big]^q\Big)+C.
\end{align*}
This produces~\eqref{ineq:moment-bound:H^2:n-moment} for $n=1$.

Next, assume~\eqref{ineq:moment-bound:H^2:n-moment} holds for $n-1$ and consider the case $n\ge 2$. Analogous to~\eqref{eqn:Ito:L^m.(Psi_1m)^n}, applying $\L^m$ to $\Psitwom(u,v)^n$ gives
\begin{align}
\L^m\Psitwom(u,v)^n&=n\Psitwom(u,v)^{n-1}\L^m\Psitwom(u,v) \nt \\
&\qquad+\tfrac{1}{2} n(n-1) \Psitwom(u,v)^{n-2}\sum_{k\ge 1}|\la  2mv+u, Qe_k\ra_{H^1}|^2.\label{eqn:Ito:L^m.(Psi_2m)^n}
\end{align}
In light of~\eqref{ineq:L^m.Psi_2m:lambda<2}, 
we employ Young's inequality to see that
\begin{align*}
&n\Psitwom(u,v)^{n-1}\L^m\Psitwom(u,v) \\
&\le n\Psitwom(u,v)^{n-1}\big(  -\tfrac{1}{2}\|Au\|^2_H-\tfrac{1}{2}m\|A^{1/2}v\|^2_H+\Psionem(u,v)^q+C    \big)\\
&\le -c \Psitwom(u,v)^n + C \Psitwom(u,v)^{n-1}\Psionem(u,v)^q+C\Psitwom(u,v)^{n-1}\\
&\le -c \Psitwom(u,v)^n +C \Psionem(u,v)^{nq}+C.
\end{align*}
Concerning the last term on the right--hand side of \eqref{eqn:Ito:L^m.(Psi_2m)^n}, we invoke the fact that $\Tr(QAQ^*)<\infty$ to see that
\begin{align*}
\sum_{k\ge 1}|\la  2mv+u, Qe_k\ra_{H^1}|^2 \le \Tr(QAQ^*) \|2mv+u\|^2_{H^1}\le c \Psitwom(u,v).
\end{align*}
whence
\begin{align*}
\tfrac{1}{2} n(n-1)& \Psitwom(u,v)^{n-2}\sum_{k\ge 1}|\la  2mv+u, Qe_k\ra_{H^1}|^2\le C\Psitwom(u,v)^{n-1}.
\end{align*}
Combining the above with~\eqref{eqn:Ito:L^m.(Psi_2m)^n}, we deduce 
\begin{align*}
\L^m\Psitwom(u,v)^n& \le -c \Psitwom(u,v)^n+C\Psitwom(u,v)^{n-1} +C \Psionem(u,v)^{nq}+C\\
&\le -c \Psitwom(u,v)^n+C \Psionem(u,v)^{nq}+C.
\end{align*}
As a consequence, from~\eqref{ineq:moment-bound:H:n-moment}, we arrive at
\begin{align*}
&\E\Psitwom(\um(t),\vm(t))^n\\
&\le e^{-ct}\Psitwom(u_0,v_0)^n+C+C\int_0^t e^{-c(t-r)}\E\Psionem(\um(r),\vm(r))^{qn}\d r\\
&\le e^{-ct}\Psitwom(u_0,v_0)^n+C+C\int_0^t e^{-c(t-r)}\Big[C_{1,qn}e^{-c_{1,qn}\,r}\big[\Psionem(u_0,v_0)+2m\|\Phi_1(u)\|_{L^1}\big]^{qn}+C_{1,qn}\Big]\d r\\
&\le Ce^{-ct}\Big(\Psitwom(u_0,v_0)+ \big[\Psionem(u_0,v_0)+2m\|\Phi_1(u_0)\|_{L^1}\big]^{qn}\Big)+C,
\end{align*}
which holds for some positive constants $C,\,c$ and $q$ independent of $m$. This produces~\eqref{ineq:moment-bound:H^2:n-moment} for the general case $n\ge 2$, thereby completing part 1.

2. Turning to~\eqref{ineq:moment-bound:H^2:sup_[0,t]}, from~\eqref{ineq:L^m.Psi_2m:lambda<2} together with It\^o's formula, we have
\begin{align*}
&\d\Psitwom(\um(t),\vm(t))\\
&= \L^m\Psitwom(\um(t),\vm(t))\d t + \la 2m\vm(t)+\um(t),Q\d w(t)\ra_{H^1}\\
&\le \big[\Psionem(\um(t),\vm(t))+2m\|\Phi_1(\um(t))\|_{L^1}  \big]^{q_{2,1}}\d t+\la 2m\vm(t)+\um(t),Q\d w(t)\ra_{H^1}+C.
\end{align*}
Integrating both sides with respect to time gives
\begin{align*}
\Psitwom(\um(t),\vm(t))&\le \Psitwom(u_0,v_0)+\int_0^t\big[\Psionem(\um(r),\vm(r))+2m\|\Phi_1(\um(r))\|_{L^1}  \big]^{q_{2,1}}\d r\\
&\qquad+\int_0^t\la 2m\vm(r)+\um(r),Q\d w(r)\ra_{H^1}+Ct.
\end{align*}
By Burkholder's inequality,
\begin{align*}
\E\sup_{r\in[0,t]}\Big|\int_0^r\la 2m\vm(s)+\um(s),Q\d w(s)\ra_{H^1}\Big|^2&\le  \E\int_0^t \|Q(2m\vm(r)+\um(r))\|^2_{H^1}\d r\\
&\le c\E\int_0^t m^2\|\vm(r)\|^2_{H^1}+\|\um(r)\|^2_{H^1}\d r\\
&\le c\int_0^t \E\Psitwom(\um(r),\vm(r))\d r.
\end{align*}
In view of~\eqref{ineq:moment-bound:H:n-moment} applying to $q_{2,1}$, we see that
\begin{align*}
&\int_0^t\E\big[\Psionem(\um(r),\vm(r))+2m\|\Phi_1(\um(r))\|_{L^1}  \big]^{q_{2,1}}\d r \\
&\le C \big[\Psionem(u_0,v_0)+2m\|\Phi_1(u_0)\|_{L^1}  \big]^{q_{2,1}}+Ct.
\end{align*}
Altogether, we obtain
\begin{align*}
\E\sup_{r\in[0,t]} \Psitwom(\um(r),\vm(r)) &\le \Psitwom(u_0,v_0)+C \big[\Psionem(u_0,v_0)+2m\|\Phi_1(u_0)\|_{L^1}  \big]^{q_{2,1}}+Ct+C\\
&\qquad +C\int_0^t \E\Psitwom(\um(r),\vm(r))\d r.
\end{align*}
This establish~\eqref{ineq:moment-bound:H^2:sup_[0,t]} by virtue of Gronwall's inequality, thereby finishing the proof.

\end{proof}

Having obtained moment bound in $\Hcal^2$, we turn to the sup norm in $H^1$ for the process $\um(t)$ with respect to random initial conditions.

\begin{lemma}  \label{lem:moment-bound:H^2:Psi_2m:sup_[0,T]:random-initial-cond}
Let $(u_0,v_0)$ be random variable in $L^2(\Omega;\Hcal^2)$ and $(u^m(t),v^m(t))$ be the solution of~\eqref{eqn:wave} with initial condition $(u_0,v_0)$. Then, the following holds for all $p\ge 1$, $\beta>0$ sufficiently small and $T> 0$:
\begin{align} 
&\E\sup_{t\in[0,T]}\|\um(t)\|^{2p}_{H^1}    \nt \\
&\le C\Big(\E \Psitwom(u_0,v_0)^p+ \E\exp\Big\{\beta \big[\Psionem(u_0,v_0) +2m\|\Phi_1(u_0)\|_{L^1}\big]  \Big\}+1\Big),\label{ineq:moment-bound:|u|_(H^1):sup_[0,T]:random-initial-cond}
\end{align}
for some positive constant $C=C(T,p,\beta)$ independent of $m$ and $(u_0,v_0)$.
\end{lemma}
\begin{proof}
Let $\Gamma^m(t)$ be the stochastic convolution as in~\eqref{eqn:wave:linear}. Denoting 
\begin{align*}
\ubar(t)=\um(t)-\Gammau(t) \quad \text{and}\quad \vbar(t)=\vm(t)-\Gammav(t),
\end{align*}
observe that $(\ubar(t),\vbar(t))$ solves the equation
\begin{align}
\ddt \ubar(t) &=\vbar(t),\nt \\
m\,\ddt\vbar(t)&=-A\ubar(t)-\vbar(t)+\f(\um(t)),\label{eqn:wave:u^m-Gamma^m}\\
(\ubar(0),\vbar(0))&=(u_0,v_0).\nt
\end{align}
Recalling $\Psitwom$ as in~\eqref{form:Psi_2m}, a calculation from~\eqref{eqn:wave:u^m-Gamma^m} gives
\begin{align}
\ddt \Psitwom(\ubar(t),\vbar(t)) &= -\|\ubar(t)\|^2_{H^2}-m\|\vbar(t)\|^2_{H^1}+2m\la \f'(\um(t))\grad \um(t),\grad \vbar(t)\ra_{H} \nt \\
&\qquad +\la  \f'(\um(t))\grad \um(t),\grad \ubar(t)\ra_{H}. \label{eqn:d.Psi_2m:u^m-Gamma^m}
\end{align}
Concerning the last term on the right--hand side of~\eqref{eqn:d.Psi_2m:u^m-Gamma^m}, we invoke~\eqref{cond:phi:phi'=O(x^(lambda-1))} and~\eqref{cond:phi:sup.phi'<a_f} to estimate
\begin{align*}
&\la  \f'(\um(t))\grad \um(t),\grad \ubar(t)\ra_{H}\\
&= \la  \f'(\um(t))\grad \ubar(t),\grad \ubar(t)\ra_{H}+\la  \f'(\um(t))\grad \Gammau(t),\grad \ubar(t)\ra_{H}\\
&\le a_\f \|\ubar(t)\|^2_{H^1}+ a_4\la (|\um(t)|^{\lambda-1}+1)|\grad \Gammau(t)|,|\grad \ubar(t)|\ra_H\\
&\le c \|\ubar(t)\|^2_{H^1}+c\|\Gammau(t)\|^2_{H^1}+c\la |\um(t)|^{\lambda-1}|\grad \Gammau(t)|,|\grad \ubar(t)|\ra_H.
\end{align*}
To further bound the last term on the above right--hand side, we employ Holder's inequality and Sobolev embedding (recalling $\lambda<2$)
\begin{align*}
\la |\um(t)|^{\lambda-1}|\grad \Gammau(t)|,|\grad \ubar(t)|\ra_H &\le c \|\um(t)\|^{\lambda-1}_{H^1}\|\Gammau(t)\|_{H^2}\|\ubar(t)\|_{H^1}\\
&\le c \|\ubar(t)\|^\lambda_{H^1}\|\Gammau(t)\|_{H^2}+c\|\ubar(t)\|_{H^1}\|\Gammau(t)\|^\lambda_{H^2}\\
&\le c \|\ubar(t)\|^2_{H^1}+ c\|\Gammau(t)\|_{H^2}^{2\lambda}+c\|\Gammau(t)\|_{H^2}^{\frac{2}{2-\lambda}}.
\end{align*}
It follows that
\begin{align} 
&\la  \f'(\um(t))\grad \um(t),\grad \ubar(t)\ra_{H} \nt \\
&\le c \|\ubar(t)\|^2_{H^1}+c\|\Gammau(t)\|^2_{H^1}+c\|\Gammau(t)\|_{H^2}^{2\lambda}+c\|\Gammau(t)\|_{H^2}^{\frac{2}{2-\lambda}}.\label{ineq:<phi'.grad(u^m),grad(u^m-Gamma^m)>}
\end{align}
To estimate the cross term between $\um(t)$ and $\vbar(t)$ on the right--hand side of~\eqref{eqn:d.Psi_2m:u^m-Gamma^m}, we employ an argument similarly to~\eqref{ineq:phi'.grad.u.grad.v:lambda<2} and obtain
\begin{align*}
&m\la \f'(\um(t))\grad \um(t),\grad \vbar(t)\ra_H\\
 &\le  \tfrac{1}{4}\|\um(t)\|_{H^2}^2+\tfrac{1}{2}m\|\vbar(t)\|_{H^2}^2+\big(\|\um(t)\|^{2}_H+m\|\um(t)\|_{H^1}^{2} \big)^{q}+C\\
 &\le \tfrac{1}{2}\|\ubar(t)\|_{H^2}^2+\tfrac{1}{2}m\|\vbar(t)\|_{H^2}^2+C\big[\Psionem(\um(t),\vm(t))+2m\|\Phi_1(\um(t))\|_{L^1} \big]^{q}+\tfrac{1}{2}\|\Gammau(t)\|_{H^2}^2+C.
\end{align*}
We now collect everything together with the identity~\eqref{eqn:d.Psi_2m:u^m-Gamma^m} to deduce
\begin{align*}
\ddt \Psitwom(\ubar(t),\vbar(t)) & \le C \Psitwom(\ubar(t),\vbar(t)) +C\big[\Psionem(\um(t),\vm(t)) +2m\|\Phi_1(\um(t))\|_{L^1}\big]^{q}\\
&\qquad+C\|\Gammau(t)\|_{H^2}^{2}+ C\|\Gammau(t)\|_{H^2}^{2\lambda}+C\|\Gammau(t)\|_{H^2}^{\frac{2}{2-\lambda}}.
\end{align*}
As a consequence, for all $p\ge1$, we infer a positive constant $q=q(p,\lambda)$ large enough such that
\begin{align*}
\ddt \Psitwom(\ubar(t),\vbar(t))^p &\le C \Psitwom(\ubar(t),\vbar(t))^p+C\big[\Psionem(\um(t),\vm(t)) +2m\|\Phi_1(\um(t))\|_{L^1}\big]^{q}\\
&\qquad+C\|\Gammau(t)\|_{H^2}^{q}+C.
\end{align*}
By virtue of Gronwall's inequality, we arrive at the suppremum bound for $T>0$
\begin{align*}
&\E\sup_{t\in[0,T]} \Psitwom(\ubar(t),\vbar(t))^p \\
& \le C\E \Psitwom(u_0,v_0)^p + C\E\sup_{t\in[0,T]}\big[\Psionem(\um(t),\vm(t)) +2m\|\Phi_1(\um(t))\|_{L^1}\big]^{q}\\
&\qquad+C\E\sup_{t\in [0,T]}\|\Gammau(t)\|_{H^2}^{q}+C,
\end{align*}
that holds for some positive constant $C=C(T)$ independent of $m$ and $(u_0,v_0)$. It follows that
\begin{align*}
&\E\sup_{t\in[0,T]} \|\um(t)\|^{2p}_{H^1} \\
& \le C\E \Psitwom(u_0,v_0)^p + C\E\sup_{t\in[0,T]}\big[\Psionem(\um(t),\vm(t)) +2m\|\Phi_1(\um(t))\|_{L^1}\big]^{q}\\
&\qquad+C\E\sup_{t\in [0,T]}\|\Gammau(t)\|_{H^2}^{q}+C.
\end{align*}
In view of~\eqref{ineq:moment-bound:H:exponential:random-initial-cond:sup_[0,T]} and~\eqref{ineq:moment-bound:Gamma^m(t):H^2:sup_[0,T]}, we deduce for $\beta$ sufficiently small
\begin{align*}
&\E\sup_{t\in[0,T]} \|\um(t)\|^{2p}_{H^1} \\
&\le C\E \Psitwom(u_0,v_0)^p+ C\E\exp\Big\{\beta \big[\Psionem(u_0,v_0) +2m\|\Phi_1(u_0)\|_{L^1}\big]  \Big\}+C.
\end{align*}
This produces~\eqref{ineq:moment-bound:|u|_(H^1):sup_[0,T]:random-initial-cond}, thereby finishing the proof.

\end{proof}

As a consequence of Lemma~\ref{lem:moment-bound:H^2:Psi_2m:sup_[0,T]:random-initial-cond}, we assert a moment bound on $\um(t)$ in the sup norm in $H^2$. Ultimately, Lemma~\ref{lem:moment-bound:H^2:|Au|^2:sup_[0,T]:random-initial-cond} will be exploited to prove the small mass limits in Section~\ref{sec:small-mass}.

\begin{lemma}  \label{lem:moment-bound:H^2:|Au|^2:sup_[0,T]:random-initial-cond}
Let $(u_0,v_0)$ be random variable in $L^2(\Omega;\Hcal^2)$ and $(u^m(t),v^m(t))$ be the solution of~\eqref{eqn:wave} with initial condition $(u_0,v_0)$. Then, there exist positive constants $p>0$ such that the following holds for all $\beta$ sufficiently small and $T> 0$:
\begin{align} 
\E\sup_{t\in[0,T]}\|\um(t)\|^2_{H^2}&\le C\Big(\E \big[\|u_0\|^2_{H^2}+m\|v_0\|^2_{H^1}\big]+\E\Psitwom(u_0,v_0)^p \nt   \\
&\qquad+\E\exp\Big\{\beta \big[\Psionem(u_0,v_0) +2m\|\Phi_1(u_0)\|_{L^1}\big]  \Big\}+1\Big),\label{ineq:moment-bound:H^2:|Au|^2:sup_[0,T]:random-initial-cond}
\end{align}
for some positive constant $C=C(T,p,\beta)$ independent of $m$ and $(u_0,v_0)$.
\end{lemma}
\begin{proof}
Similarly to the proof of Lemma~\ref{lem:moment-bound:H^2:Psi_2m:sup_[0,T]:random-initial-cond}, we consider the stochastic convolution $\Gamma^m(t)$ as in~\eqref{eqn:wave:linear} and the difference
\begin{align*}
\ubar(t)=\um(t)-\Gammau(t) \quad \text{and}\quad \vbar(t)=\vm(t)-\Gammav(t).
\end{align*}
From~\eqref{eqn:wave:u^m-Gamma^m}, we employ~\eqref{cond:phi:phi'=O(x^(lambda-1))} making use of Cauchy--Schwarz inequality to see that
\begin{align*}
\ddt \big(\|\ubar(t)\|^2_{H^2}+m\|\vbar(t)\|^2_{H^1}\big)& = -2\|\vbar(t)\|^2_{H^1}+2\la \f'(\um(t))\grad \um(t),\grad \vbar(t)\ra_{H} \\
&\le \| \f'(\um(t))\grad \um(t)\|^2_H\\
&\le c\|\um(t)\|^{2(\lambda-1)}_{L^\infty} \|\um(t)\|^2_{H^1}+c\|\um(t)\|^2_{H^1}.
\end{align*}
Since $H^2\hookrightarrow L^{\infty}$ in dimension $d\le 3$, we have (recalling $\lambda<2$)
\begin{align*}
\|\um(t)\|^{2(\lambda-1)}_{L^\infty} \|\um(t)\|^2_{H^1} &\le c\|\um(t)\|^{2(\lambda-1)}_{H^2} \|\um(t)\|^2_{H^1}\\
&\le c\|\um(t)\|^2_{H^2}+c\|\um(t)\|^{\frac{2}{2-\lambda}}_{H^1}\\
&\le c\|\ubar(t)\|^2_{H^2}+c\|\Gammau(t)\|^2_{H^2} +c\|\um(t)\|^{\frac{2}{2-\lambda}}_{H^1}.
\end{align*}
It follows that
\begin{align*}
\ddt \big(\|\ubar(t)\|^2_{H^2}+m\|\vbar(t)\|^2_{H^1}\big)& \le c\|\ubar(t)\|^2_{H^2}+c\|\Gammau(t)\|^2_{H^2} +c\|\um(t)\|^{\frac{2}{2-\lambda}}_{H^1}+c\|\um(t)\|^2_{H^1},
\end{align*}
whence
\begin{align*}
\ddt \big(\|\ubar(t)\|^2_{H^2}+m\|\vbar(t)\|^2_{H^1}\big) &\le c \big(\|\ubar(t)\|^2_{H^2}+m\|\vbar(t)\|^2_{H^1}\big) +c\|\Gammau(t)\|^2_{H^2} +c\|\um(t)\|^{2p}_{H^1}+c,
\end{align*}
where $p$ and $c$ are positive constants independent of $m$ and $(u_0,v_0)$.
Integrating the above inequality with respect to time, we obtain
\begin{align*}
&\E\sup_{t\in[0,T]}\big[\|\ubar(t)\|^2_{H^2}+m\|\vbar(t)\|^2_{H^1}\big]\\
&\le  \E \big[\|u_0\|^2_{H^2}+m\|v_0\|^2_{H^1}\big]+C\,\E\sup_{t\in[0,T]}\|\Gammau(t)\|^2_{H^2}+C\,\E\sup_{t\in[0,T]}\|\um(t)\|^{2p}_{H^1}+C,
\end{align*}
implying
\begin{align*}
&\E\sup_{t\in[0,T]}\|\um(t)\|^2_{H^2}\\
&\le 2 \E \big[\|u_0\|^2_{H^2}+m\|v_0\|^2_{H^1}\big]+C\,\E\sup_{t\in[0,T]}\|\Gammau(t)\|^2_{H^2}+C\,\E\sup_{t\in[0,T]}\|\um(t)\|^{2p}_{H^1}+C.
\end{align*}
In view of~\eqref{ineq:moment-bound:|u|_(H^1):sup_[0,T]:random-initial-cond} and~\eqref{ineq:moment-bound:Gamma^m(t):H^2:sup_[0,T]}, we infer $C=C(p,T)$ independent of $m$ and $(u_0,v_0)$ such that
\begin{align*}
\E\sup_{t\in[0,T]}\|\ubar(t)\|^2_{H^2}&\le C\Big(\E \big[\|u_0\|^2_{H^2}+m\|v_0\|^2_{H^1}\big]+\E\Psitwom(u_0,v_0)^p\\
&\qquad+\E\exp\Big\{\beta \big[\Psionem(u_0,v_0) +2m\|\Phi_1(u_0)\|_{L^1}\big]  \Big\}+1\Big).
\end{align*}
This produces~\eqref{ineq:moment-bound:H^2:|Au|^2:sup_[0,T]:random-initial-cond}, thereby finishing the proof.

\end{proof}


\section{Uniform Geometric ergodicity of~\eqref{eqn:wave}} \label{sec:geometric-ergodicity}

In this section, we prove Theorem~\ref{thm:geometric-ergodicity:mass} by establishing the uniqueness of invariant measure $\num$ as well as the uniform exponential
convergent rate of the Markov semigroup $P^m_t$ towards $\num$. In Section~\ref{sec:geometric-ergodicity:irreducibility}, we assert the irreducibility condition stating that the solution $(\um(t),\vm(t))$ is always able to visit any small neighborhood of the origin. In Section~\ref{sec:geometric-ergodicity:asymptotic-Feller}, we prove asymptotic strong Feller property, which is a large time smoothing effect of $P^m_t$. In Section~\ref{sec:geometric-ergodicity:proof}, we invoke irreducibility and asymptotic strong Feller respectively to prove that bounded sets are $\dmnb-$small and that $P^m_t$ is $\dmnb-$contracting; see Proposition~\ref{prop:contracting-d-small} for the precise statements. Altogether, we will conclude Theorem~\ref{thm:geometric-ergodicity:mass} making use of these ingredients.

\subsection{Irreducibility} \label{sec:geometric-ergodicity:irreducibility}

We start with the irreducibility property:

\begin{proposition} \label{prop:irreducibility}
Given $(u_0,v_0)\in\Hcal^1$, let $(\um(t),\vm(t))$ be the solution of~\eqref{eqn:wave}. Then, for all $r,\,R>0$, there exists $T_1$ sufficiently large independent of $m$ such that for all $t>T_1$,
\begin{align}\label{ineq:irreducibility}
\inf_{(u_0,v_0)\in B^m_R}\P\Big(V_m(\um(t),\vm(t))<r \Big) \ge c>0,
\end{align}
for some positive constant $c=c(R,r,t)$ independent of $m$. In the above, $V_m$ is the function as in~\eqref{form:V_m} and
\begin{align*}
B^m_R=\{(u_0,v_0)\in\Hcal^1: V_m(u_0,v_0)<R\}.
\end{align*} 

\end{proposition}
Due the difficulty caused by the nonlinear structure, in order to prove Proposition~\ref{prop:irreducibility}, we first
modify~\eqref{eqn:wave} as follows: we introduce
\begin{align}
\d \umt(t)&=\vmt(t)\d t,  \nt \\
m\,\d \vmt(t)&=-A\umt(t)\d t-\vmt(t)\d t+\f(\umt(t))\d t+ Q\d w(t), \label{eqn:wave:irreducibility}\\
&\qquad -\alpha_{\nbar} P_{\nbar}(\umt(t)-\Gammau(t))\d t, \nt \\
(\umt(0),\vmt(0))&=(u_0,v_0)\in \Hcal^1,\nt
\end{align}
where $\Gamma^m(t)$ is the stochastic convolution solving~\eqref{eqn:wave:linear}, $\alpha_{\nbar}$ is the eigenvalue of $A$ as in Assumption~\ref{cond:Q}, part 2, and $P_{\nbar}$ is the projection on $\text{span}(e_1,\dots,e_{\nbar})$. We note that~\eqref{eqn:wave:irreducibility} only differs from~\eqref{eqn:wave} by the appearance of the term $ -\alpha_{\nbar} P_{\nbar}(\umt(t)-\Gammau(t))\d t$. Two of the main ingredients to prove Proposition~\ref{prop:irreducibility} are given in the following lemmas whose proofs are deferred to the end of this subsection.

\begin{lemma} \label{lem:irreducibility:shift}
Given $(u_0,v_0)\in\Hcal^1$, let $(\umt(t),\vmt(t))$ be the solution of~\eqref{eqn:wave:irreducibility}. Then, for all $r,\,R>0$, there exists $T_1$ sufficiently large independent of $m$ such that for all $t>T_1$,
\begin{align}\label{ineq:irreducibility:shift}
\inf_{V_m(u_0,v_0)<R}\P\Big(V_m(\umt(t),\vmt(t))<r \Big) \ge c>0,
\end{align}
for some positive constant $c=c(R,r,t)$ independent of $m$. In the above, $V_m$ is the function as in~\eqref{form:V_m}.
\end{lemma}

\begin{lemma} \label{lem:irreducibility:shift:sup_[0,t]}
Given $(u_0,v_0)\in\Hcal^1$, let $(\umt(t),\vmt(t))$ be the solution of~\eqref{eqn:wave:irreducibility}. Then, for all $T>0$, the following holds
\begin{align}\label{ineq:irreducibility:shift:sup_[0,t]}
\E\sup_{t\in[0,T]} \|\umt(t)\|^2_H\le   C(V_m(u_0,v_0)+1)e^{cT},
\end{align}
for some positive constants $c,\, C$ independent of $m$, $T$ and $(u_0,v_0)$. In the above, $V_m$ is the function as in~\eqref{form:V_m}.
\end{lemma}

We now give the proof of Proposition~\ref{prop:irreducibility}, whose argument follows along the lines of \cite[Proposition 4.2]{foldes2019large} tailored to our settings.
\begin{proof}[Proof of Proposition~\ref{prop:irreducibility}]
For $N>0$, let $(\umt_N,\vmt_N)$ solve the following truncated version of~\eqref{eqn:wave:irreducibility}
\begin{align}
\d \umt_N (t)&=\vmt_N (t)\d t,  \nt \\
m\,\d \vmt_N (t)&=-A\umt_N (t)\d t-\vmt_N (t)\d t+\f(\umt_N (t))\d t+ Q\d w(t), \label{eqn:wave:irreducibility:N}\\
&\qquad -\alpha_{\nbar}\theta_N(\|P_{\nbar}(\umt_N (t)-\Gammau(t))\|_H) P_{\nbar}(\umt_N (t)-\Gammau(t))\d t, \nt \\
(\umt_N (0),\vmt_N (0))&=(u_0,v_0)\in \Hcal^1,\nt
\end{align}
where $\theta_N:\rbb\to\rbb$ is a smooth truncating function such that
\begin{align} \label{form:theta_N}
\theta_N(x)=\begin{cases}
1, & |x|\le N,\\
\text{monotonicity},& N\le |x|\le N+1,\\
0,& |x|\ge N+1.
\end{cases}
\end{align}
Denoting $\tilde{\tau}_N$ to be the stopping time given by
\begin{align*}
\tilde{\tau}_N^m=\inf \big\{t\ge 0:\|P_{\nbar}(\umt(t)-\Gammau(t))\|_H>N   \big\},
\end{align*}
observe that $\P-$a.s. for $0\le t\le \taut_N^m$, $(\umt_N(t),\vmt_N(t))=(\umt(t),\vmt(t))$. Furthermore, we note that under this truncation, by Girsanov Theorem, the law induced by $(\umt_N(\cdot),\vmt_N(\cdot))$ on $C([0,t];\Hcal^1)$ is equivalent to that induced by $(\um(\cdot),\vm(\cdot))$. Indeed, thanks to the condition~\eqref{cond:Q} on the fact that $Q$ is invertible on $\text{span}\{e_1,\dots,e_{\nbar}\}$, it holds that
\begin{align*}
&\E \exp\Big\{\tfrac{1}{2}\int_0^t\alpha_{\nbar}^2\|Q^{-1}P_{\nbar}\big(\umt_N(s)-\Gammau(s)\big)\|^2_H\theta_N^2\big(\|P_{\nbar}\big(\utilde^m_N(s)-\Gammau(s)\big)\|_H\big)\d s\Big\}\\
&\le  \exp\Big\{\tfrac{1}{2}\alpha_{\nbar}^2a_Q^{-2}(N+1)^2t\Big\}.
\end{align*}
This verifies Novikov's condition, which in turn yields the equivalence in laws. We now introduce the change--of--measure function
$$U_N^m(t)=\exp\Big\{\int_0^t\la f^m_N(s),\d w(s)\ra_H-\tfrac{1}{2}\int_0^t\|f_N^m(s)\|_H^2\d s \Big\} ,$$
where
$$ f_N^m(t)=\alpha_{\nbar}Q^{-1}P_{\nbar}\big(\utilde^m_N(t)-\Gammau(t)\big)\theta_N\big(\|P_{\nbar}\big(\utilde^m_N(t)-\Gammau(t)\big)\|_H\big).$$
By Girsanov Theorem, we have the following identity
\begin{align*}
\P(V_m(\um(t),\vm(t))< r)=\E\Big[U_N^m(t)\cdot \mathbf{1}\big\{V_m(\umt_N(t),\vmt_N(t))<r\big\}\Big].
\end{align*}
Hence, on one hand, for given $C>0$, we have a chain of implications
\begin{align*}
\P(V_m(\um(t),\vm(t))< r)&=\E\Big[U_N^m(t)\cdot \mathbf{1}\big\{V_m(\umt_N(t),\vmt_N(t))< r\big\}\Big]\\
&\ge C\,\P\Big(U_N^m(t)>C, V_m(\umt_N(t),\vmt_N(t))< r\Big)\\
&\ge C\,\P\Big(U_N^m(t)>C, V_m(\umt_N(t),\vmt_N(t))< r,\tautil_{N}^m>t\Big)\\
&=C\,\P\Big(U_N^m(t)>C, V_m(\umt(t),\vmt(t))< r,\tautil_N^m>t\Big).
\end{align*}
On the other hand, it holds that
\begin{align*}
&\P\big(V_m(\umt(t),\vmt(t))< r\big)\\
&\le \P\Big(U_N^m(t)>C, V_m(\umt(t),\vmt(t))< r,\tautil_{N}^m>t\Big)+\P\big(U_N^m(t)\le C\big)+\P\big(\tautil_{N}^m\le t\big).
\end{align*}
We thus arrive at the estimate
\begin{align*}
\frac{1}{C}\P(V_m(\um(t),\vm(t))< r)\ge \P\big(V_m(\umt(t),\vmt(t))< r\big)-\P\big(U_N^m(t)\le C\big)-\P\big(\tautil_N^m\le t\big),
\end{align*}
whence
\begin{align}
&\frac{1}{C}\inf_{V_m(u_0,v_0)< R}\P(V_m(\um(t),\vm(t))< r) \nt\\
&\ge \inf_{V_m(u_0,v_0)< R}\P\big(V_m(\umt(t),\vmt(t))< r\big)-\sup_{V_m(u_0,v_0)< R}\P\big(U_N^m(t)\le C\big)\nt\\
&\qquad-\sup_{V_m(u_0,v_0)< R}\P\big(\tautil_{N}^m\ge t\big). \label{ineq:irreducibility:1/C}
\end{align}
In view of Lemma~\ref{lem:irreducibility:shift}, for $t\ge T_1$ where $T_1=T_1(r,R)$ is as in Lemma~\ref{lem:irreducibility:shift}, we readily have
$$\inf_{V_m(u_0,v_0))< R}\P\big(V_m(\umt(t),\vmt(t))< r\big)>c^*>0, $$
where $c^*=c^*(r,R,t)>0$ is a constant independent of $C$, and $N$. It therefore remains to show that, by taking $N$ to infinity and $C$ to zero, two suppremum terms on the right-hand side of~\eqref{ineq:irreducibility:1/C} can be arbitrarily small. In other words, we claim that there exist $C$ sufficiently small and $N$ sufficiently large both independent of $m$ such that
\begin{equation} \label{ineq:sup.U+sup.tau}
\sup_{V_m(u_0,v_0)< R}\P\big(U_N^m(t)\le C\big)+\sup_{V_m(u_0,v_0)< R}\P\big(\tautil_{N}^m\le t\big)\le \frac{c^*}{2}. 
\end{equation}
To bound the second term involving $\tautil_{N}^m$, we invoke Markov inequality together with~\eqref{ineq:irreducibility:shift:sup_[0,t]} and~\eqref{ineq:moment-bound:Gamma^m(t):H^2:sup_[0,T]} to see that
\begin{align*}
\P\big(\tautil_{N}^m\le t\big)&=\P\big(\sup_{0\le s\le t} \|\utilde^m(s)-\Gammau(s)\|_H\ge N  \big) \\
&\le \frac{1}{N^2}\E\sup_{0\le s\le t} \|\utilde^m-\Gammau(s)\|_H^2\\
&\le \frac{1}{N^2} (V_m(u_0,v_0)+1)e^{ct}\\
&\le \frac{1}{N^2} (R+1)e^{ct},
\end{align*}
where in the last implication above, $c=c(t)$ does not depend on $N$, $R$ and $m$. With regards to first term $\sup_{x\in B_R}\P\big(U_{\x}^N(t)\le C\big)$, we employ Burkholder's inequality to estimate
\begin{align*}
\P(U_N^m(t)\le C)&=\P\Big(\tfrac{1}{2}\int_0^t\|f_N^m(s)\|_H^2\d s-\int_0^t\la f_N^m(s),\d w(s)\ra_H\d s \ge \log(C^{-1}) \Big)\\
&\le \frac{1}{\log(C^{-1})}\Big(\tfrac{1}{2}\E\int_0^t\|f_N^m(s)\|_H^2\d s+\E\Big|\int_0^t\la f_N^m(s),\d w(s)\ra_H\d s \Big|\Big)\\
&\le \frac{2}{\log(C^{-1})}\Big(1+\E\int_0^t\|f_N^m(s)\|_H^2\d s\Big)\\
&\le \frac{2}{\log(C^{-1})}\big(1+\alpha_{\nbar}^2a_Q^{-2}(N+1)^2t\big).
\end{align*}
Hence, by first taking $N$ large and then shrinking $C$ further to zero, we infer two positive constants $N^*$ and $C^*$ independent of $m$ such that inequality~\eqref{ineq:sup.U+sup.tau} holds. We therefore immediately obtain the desired lower bound
\begin{align*}
\inf_{V_m(u_0,v_0)<R}\P(V_m(\um(t),\vm(t))< r)\ge \frac{C^*c^*}{2},
\end{align*}
uniformly in $m$. This finishes the proof.

\end{proof}

Turning to Lemma~\ref{lem:irreducibility:shift}, we prove that $(\umt(t),\vmt(t))$ satisfies the irreducibility condition~\eqref{ineq:irreducibility:shift} making use of the stochastic convolution $\Gamma^m(t)$ solving the linear system~\eqref{eqn:wave:linear}.

\begin{proof}[Proof of Lemma~\ref{lem:irreducibility:shift}]
Letting $\Gamma^m(t)$ be the solution of~\eqref{eqn:wave:linear}, denote
\begin{align*}
\ubar(t)=\umt(t)-\pi_1\Gamma^m(t),\qquad  \vbar(t)=\vmt(t)-\pi_2\Gamma^m(t).
\end{align*}
From~\eqref{eqn:wave:irreducibility} and~\eqref{eqn:wave:linear}, observe that $(\ubar(t),\vbar(t))$ solves the system
\begin{align*}
\tfrac{\d}{\d t} \ubar(t)&=\vbar(t),\\
m\ddt \vbar(t)&=-A\ubar(t)-\vbar(t)+\f(\umt(t))-\alpha_{\nbar}P_{\nbar}\ubar(t),\\
(\ubar(0),\vbar(0))&=(u_0,v_0).
\end{align*}
Denote 
\begin{align} \label{form:Phi_2}
\Phi_2(x)=-\int_0^x \f(r)\d r.
\end{align}
For a slight abuse of notation, we denote
\begin{align*}
\|\Phi_2(u)\|_{L^1} = \int_{\domain} \Phi_2(u(x))\d x.
\end{align*}
Recall $\Psionem$ as in~\eqref{form:Psi_1m}, a routine calculation gives
\begin{align}
&\ddt \big[\Psionem(\ubar(t),\vbar(t))+2m\|\Phi_2(\ubar(t))\|_{L^1}\big] \nt \\
&=-\|A^{1/2}\ubar(t)\|^2_H-m\|\vbar(t)\|^2_H+2m\la \f(\um(t))-\f(\ubar(t)),\vbar(t)\ra_H+\la \f(\um(t)),\ubar(t)\ra_H  \label{eqn:Ito:Psi_1m+Phi:irreducibility}\\
&\qquad -2m\alpha_{\nbar}\la P_{\nbar}\ubar(t),\vbar(t)\ra_H-\alpha_{\nbar}\|P_n\ubar(t)\|^2_H. \nt
\end{align}
Using Young inequality, we readily have
\begin{align*}
-2m\alpha_{\nbar}\la P_{\nbar}\ubar(t),\vbar(t)\ra_H \le \tfrac{1}{4}m\|\vbar(t)\|^2_H+4m\alpha_{\nbar}^2\|\ubar(t)\|^2.
\end{align*}
Also, by Sobolev embedding,  it holds that
\begin{align*}
&-\|A^{1/2}\ubar(t)\|^2_H-\alpha_{\nbar}\|P_n\ubar(t)\|^2_H\\
& \le -\varepsilon\|A^{1/2}\ubar(t)\|^2_H-(1-\varepsilon)\|(I-P_{\nbar})A^{1/2}\ubar(t)\|^2_H-\alpha_{\nbar}\|P_{\nbar}\ubar(t)\|^2_H\\
&\le  -\varepsilon\|A^{1/2}\ubar(t)\|^2_H-(1-\varepsilon)\alpha_{\nbar}\|\ubar(t)\|^2_H.
\end{align*}
With regard to $\la \f(\um(t)),\ubar(t)\ra_H$, note that
\begin{align*}
\la \f(\um(t)),\ubar(t)\ra_H = \la \f(\um(t))-\f(\ubar(t)),\ubar(t)\ra_H+\la \f(\ubar(t)),\ubar(t)\ra_H.
\end{align*}
In view of~\eqref{cond:phi:xphi(x)<x^2-x^(lambda+1)}, for all $\varepsilon$ sufficiently small independent of $m$, we readily have 
\begin{align*}
\la \f(\ubar(t)),\ubar(t)\ra_H\le   (a_\f+\varepsilon)\|\ubar(t)\|^2_H-\frac{\varepsilon}{(\lambda+1)\big(\frac{2a_3}{a_2}\big)^{\frac{\lambda-1}{\lambda+1}}}\|\ubar(t)\|^{\lambda+1}_{L^{\lambda+1}}.
\end{align*}
On the other hand, we invoke~\eqref{cond:phi:|phi(x)-phi(y)|} to estimate
\begin{align*}
\f(\um(t))-\f(\ubar(t)) \le c\big( |\Gammau(t)|^{\lambda}+|\Gammau(t)|\,|\ubar(t)|^{\lambda-1}+|\Gammau(t)|\big),
\end{align*}
whence
\begin{align*}
 \la \f(\um(t))-\f(\ubar(t)),\ubar(t)\ra_H&\le C\big(\|\Gammau(t)\|^{2\lambda}_{L^{2\lambda}}+\|\Gammau(t)\|^{\lambda+1}_{L^{\lambda+1}}+\|\Gammau(t)\|^2_H  \big)\\
 &\qquad + \varepsilon\|\ubar(t)\|^2_H+\frac{\varepsilon}{2(\lambda+1)\big(\frac{2a_3}{a_2}\big)^{\frac{\lambda-1}{\lambda+1}}}\|\ubar(t)\|^{\lambda+1}_{L^{\lambda+1}}. 
\end{align*}
As a consequence, we obtain the bound
\begin{align*}
\la \f(\um(t)),\ubar(t)\ra_H&\le C\big(\|\Gammau(t)\|^{2\lambda}_{L^{\infty}}+\|\Gammau(t)\|^{\lambda+1}_{L^{\infty}}+\|\Gammau(t)\|^2_{L^\infty} \big)\\
 &\qquad + (a_\f+2\varepsilon)\|\ubar(t)\|^2_H-\frac{\varepsilon}{2(\lambda+1)\big(\frac{2a_3}{a_2}\big)^{\frac{\lambda-1}{\lambda+1}}}\|\ubar(t)\|^{\lambda+1}_{L^{\lambda+1}}.
\end{align*}
In the above, we emphasize that $C=C(\varepsilon, a_\f)$ is a positive constant independent of $m$. 

Similarly, concerning                                                                                                                                                                                                                                                                                                                                                                                                                                                                                                                                                                                                                                                                                                                                                                                                                                                                                                                                                                                                                                                                                                                                                                                                                                                                                                                                                                                                                                                                                                                                                                                                                                                                                                                                     the nonlinear term involving $\vbar$ on the right--hand side of~\eqref{eqn:Ito:Psi_1m+Phi:irreducibility}, we note that
\begin{align*}
2m\la \f(\um(t))-\f(\ubar(t)),\vbar(t)\ra_H &\le c\,m \big(\|\Gammau(t)\|_{L^\infty}^{\lambda}+\|\Gammau(t)\|_{L^\infty}\big)\|\vbar(t)\|_H\\
& \qquad + c\,m \la |\Gammau(t)|\,|\ubar(t)|^{\lambda-1},|\vbar(t)|\ra_H.
\end{align*}
We invoke Young's inequality to further estimate
\begin{align*}
 &c\,m \big(\|\Gammau(t)\|_{L^\infty}^{\lambda}+\|\Gammau(t)\|_{L^\infty}\big)\|\vbar(t)\|_H\\
 &\le  c\,m \big(\|\Gammau(t)\|_{L^\infty}^{2\lambda}+\|\Gammau(t)\|_{L^\infty}^2\big) +\tfrac{1}{4}m\|\vbar(t)\|_H^2,
\end{align*}
and (recalling $\lambda<2$)
\begin{align*}
&c\,m \la |\Gammau(t)|\,|\ubar(t)|^{\lambda-1},|\vbar(t)|\ra_H \\
&\le c\,m \|\Gammau(t)\|_{L^\infty}^{\frac{2(\lambda+1)}{3-\lambda}} + m\|\ubar(t)\|^{\lambda+1}_{L^{\lambda+1}}+ \tfrac{1}{4}m\|\vbar(t)\|_H^2.
\end{align*}
It follows that
\begin{align*}
2m\la \f(\um(t))-\f(\ubar(t)),\vbar(t)\ra_H &\le  c\,m \Big(\|\Gammau(t)\|_{L^\infty}^{2\lambda}+\|\Gammau(t)\|_{L^\infty}^2+\|\Gammau(t)\|_{L^\infty}^{\frac{2(\lambda+1)}{3-\lambda}} \Big)\\
&\qquad+ m\|\ubar(t)\|^{\lambda+1}_{L^{\lambda+1}}+ \tfrac{1}{2}m\|\vbar(t)\|_H^2.
\end{align*}
Altogether, from~\eqref{eqn:Ito:Psi_1m+Phi:irreducibility}, we deduce the bound
\begin{align*}
&\ddt \big[\Psionem(\ubar(t),\vbar(t))+2m\|\Phi_2(\ubar(t))\|_{L^1}\big] \\
&\le -\varepsilon\|A^{1/2}\ubar(t)\|^2_H-\big[(1-\varepsilon)\alpha_{\nbar}-(a_\f+2\varepsilon)-4m\alpha_{\nbar}^2\big]\|\ubar(t)\|^2_H\\
&\qquad- \tfrac{1}{4}m\|\vbar(t)\|_H^2 -\Big( \frac{\varepsilon}{2(\lambda+1)\big(\frac{2a_3}{a_2}\big)^{\frac{\lambda-1}{\lambda+1}}}-m\Big)\|\ubar(t)\|^{\lambda+1}_{L^{\lambda+1}}\\
&\qquad+ C\Big(\|\Gammau(t)\|^{2\lambda}_{L^{\infty}}+\|\Gammau(t)\|^{\lambda+1}_{L^{\infty}}+\|\Gammau(t)\|^2_{L^\infty} +\|\Gammau(t)\|_{L^\infty}^{\frac{2(\lambda+1)}{3-\lambda}} \Big).
\end{align*}
Since $m$ is arbitrarily small, we may further deduce
\begin{align*}
&\ddt \big[\Psionem(\ubar(t),\vbar(t))+2m\|\Phi_2(\ubar(t))\|_{L^1}\big] \\
&\le -\varepsilon\|A^{1/2}\ubar(t)\|^2_H-\big[(1-\varepsilon)\alpha_{\nbar}-(a_\f+3\varepsilon)\big]\|\ubar(t)\|^2_H\\
&\qquad- \tfrac{1}{4}m\|\vbar(t)\|_H^2- c\|\ubar(t)\|^{\lambda+1}_{L^{\lambda+1}}\\
&\qquad+ C\Big(\|\Gammau(t)\|^{2\lambda}_{L^{\infty}}+\|\Gammau(t)\|^{\lambda+1}_{L^{\infty}}+\|\Gammau(t)\|^2_{L^\infty} +\|\Gammau(t)\|_{L^\infty}^{\frac{2(\lambda+1)}{3-\lambda}} \Big),
\end{align*}
for some positive constants $C$ and $c$ independent of $m$. In view of~\eqref{cond:Phi_2:b}, 
\begin{align*}
- c \|\ubar(t)\|^{\lambda+1}_{L^{\lambda+1}}&\le -c\,m \|\ubar(t)\|^{\lambda+1}_{L^{\lambda+1}} \\
& \le - c\,m \|\Phi_2(\ubar(t))\|_{L^1}+c\,m \|\ubar(t)\|^2_H.
\end{align*}
It follows that
\begin{align}
&\ddt \big[\Psionem(\ubar(t),\vbar(t))+2m\|\Phi_2(\ubar(t))\|_{L^1}\big] \nt \\
&\le -\varepsilon\|A^{1/2}\ubar(t)\|^2_H-\big[(1-\varepsilon)\alpha_{\nbar}-(a_\f+4\varepsilon)\big]\|\ubar(t)\|^2_H  \nt \\
&\qquad- \tfrac{1}{4}m\|\vbar(t)\|_H^2- c\,m \|\Phi_2(\ubar(t))\|_{L^1}\nt \\
&\qquad+ C\Big(\|\Gammau(t)\|^{2\lambda}_{L^{\infty}}+\|\Gammau(t)\|^{\lambda+1}_{L^{\infty}}+\|\Gammau(t)\|^2_{L^\infty} +\|\Gammau(t)\|_{L^\infty}^{\frac{2(\lambda+1)}{3-\lambda}} \Big). \label{ineq:Psi_1m+Phi:irreducibility}
\end{align}
In view of the choice $\alpha_{\nbar}$ as in~\eqref{cond:phi:ergodicity}, i.e., $a_\f<\alpha_{\nbar}$, we may pick $\varepsilon$ sufficiently small such that
\begin{align*}
(1-\varepsilon)\alpha_{\nbar}>a_\f+4\varepsilon.
\end{align*}
This together with~\eqref{ineq:Psi_1m+Phi:irreducibility} implies the estimate
\begin{align*}
&\Psionem(\ubar(T),\vbar(T))+2m\|\Phi_2(\ubar(T))\|_{L^1} &\\
&\le e^{-cT}\big[\Psionem(u_0,v_0)+2m\|\Phi_2(u_0)\|_{L^1}\big]\\
&\qquad +CT\sup_{t\in [0,T]}\Big(\|\Gammau(t)\|^{2\lambda}_{L^{\infty}}+\|\Gammau(t)\|^{\lambda+1}_{L^{\infty}}+\|\Gammau(t)\|^2_{L^\infty} +\|\Gammau(t)\|_{L^\infty}^{\frac{2(\lambda+1)}{3-\lambda}} \Big).
\end{align*}
Recalling $V_m$ from~\eqref{form:V_m}, we employ~\eqref{cond:Phi_2:a}--\eqref{cond:Phi_2:b} again to see that (for $m$ sufficiently small)
\begin{align*}
cV_m(u,v)\le \Psionem(u,v)+2m\|\Phi_2(u)\|_{L^1}\le C (V_m(u,v)+1),
\end{align*}
holds for some positive constants $c,C$ independent of $m$ and $(u,v)$. As a consequence
\begin{align*}
&V_m(\ubar(T),\vbar(T)) 
\le Ce^{-cT}(V_m(u_0,v_0)+1)\\
&\qquad +C\sup_{t\in [0,T]}\Big(\|\Gammau(t)\|^{2\lambda}_{L^{\infty}}+\|\Gammau(t)\|^{\lambda+1}_{L^{\infty}}+\|\Gammau(t)\|^2_{L^\infty} +\|\Gammau(t)\|_{L^\infty}^{\frac{2(\lambda+1)}{3-\lambda}} \Big).
\end{align*}
It follows that
\begin{align*}
&V_m(\umt(T),\vmt(T)) 
\le Ce^{-cT}(V_m(u_0,v_0)+1)+C\sup_{t\in[0,T]}\Big(\|\Gammau(t)\|^2_{H^1}+m^2\|\Gammav(t)\|^2_H\Big)\\
&\qquad +C\sup_{t\in [0,T]}\Big(\|\Gammau(t)\|^{2\lambda}_{L^{\infty}}+\|\Gammau(t)\|^{\lambda+1}_{L^{\infty}}+\|\Gammau(t)\|^2_{L^\infty} +\|\Gammau(t)\|_{L^\infty}^{\frac{2(\lambda+1)}{3-\lambda}} \Big).
\end{align*}
In the above, we emphasize that $C$ and $c$ do not depend on $m$, $T$ and $(u_0,v_0)$. In light of Lemma~\ref{lem:irreducibility:Gamma^m(t)} and Agmon's inequality (in dimension $d=2$ or $d=3$), for $r\in(0,1)$ small enough, it holds that
\begin{align} \label{ineq:irreducibility:shift:sup_Gamma}
&\P\Big(\sup_{t\in [0,T]}\Big[\|\Gammau(t)\|^2_{H^1}+m^2\|\Gammav(t)\|^2_H+\|\Gammau(t)\|^{2\lambda}_{L^{\infty}}  \nt \\
&\qquad\qquad+\|\Gammau(t)\|^{\lambda+1}_{L^{\infty}}+\|\Gammau(t)\|^2_{L^\infty} +\|\Gammau(t)\|_{L^\infty}^{\frac{2(\lambda+1)}{3-\lambda}} \Big]< \tfrac{r}{2CT}\Big)>0.
\end{align}
Turning back to~\eqref{ineq:irreducibility:shift}, since $V_m(u_0,v_0)<R$, we may choose $T_1$ sufficiently large such that
\begin{align*}
Ce^{-cT_1}(V_m(u_0,v_0)+1)\le Ce^{-cT_1}(R+1)<\tfrac{r}{2},
\end{align*}
which together with~\eqref{ineq:irreducibility:shift:sup_Gamma} implies for all $r$ sufficiently small
\begin{align*}
\P\Big(V_m(\umt(t),\vmt(t))<r\Big)\ge c>0.
\end{align*}
This produces~\eqref{ineq:irreducibility:shift} for all $r>0$, thereby finishing the proof.

\end{proof}

We finish this subsection by the proof of Lemma~\ref{lem:irreducibility:shift:sup_[0,t]}, which together with Lemma~\ref{lem:irreducibility:shift} concludes Proposition~\ref{prop:irreducibility}.

\begin{proof}[Proof of Lemma~\ref{lem:irreducibility:shift:sup_[0,t]}] Recalling $\Psionem$ and $\Phi_1$ as in~\eqref{form:Psi_1m} and \eqref{cond:Phi_1}, respectively, from~\eqref{eqn:wave:irreducibility}, we compute
\begin{align}
&\d\big[ \Psionem(\umt(t),\vmt(t))+2m\|\Phi_1(\umt(t))\|_{L^1}\big] \nt \\
&= \L^m\big[\Psionem(\umt(t),\vmt(t))+2m\|\Phi_1(\umt(t))\|_{L^1}\big]\d t+\la 2m \vmt(t)+\umt(t),Q\d w(t)\ra_H\nt \\
&\qquad -2m\alpha_{\nbar}\la P_{\nbar}(\umt(t)-\Gammau(t)),\vmt(t)\ra_H\d t-\alpha_{\nbar}\la P_{\nbar}(\umt(t)-\Gammau(t)),\umt(t)\ra_H\d t. \label{eqn:Ito:d.Psi_1m:irreducibility}
\end{align}
In the above, $\L^m$ is the generator as in~\eqref{form:L^m}. Similarly to~\eqref{ineq:L^m.Psi_1m}, we readily have
\begin{align*}
&\L^m\big[\Psionem(\umt(t),\vmt(t))+2m\|\Phi_1(\umt(t))\|_{L^1}\big]
\\
&\le -\|A^{1/2}\um(t)\|^2_H -m\|\vm(t)\|^2_H  -\tfrac{a_2}{C_\f}\|\Phi_1(\um(t))\|_{L^1} +a_2+a_3|\domain|+\Tr(QQ^*).
\end{align*}
Using Burkholder's inequality, 
\begin{align*}
\E\sup_{r\in[0,t]}\Big|\int_0^t \la 2m \vmt(r)+\umt(r),Q\d w(r)\ra_H\Big|^2
& \le \E\int_0^t \|Q\big(2m\vmt(r)+\umt(r)\big)\|^2_H\d r\\
&\le \Tr(QQ^*)\int_0^t \E\|2m\vmt(r)+\umt(r)\|^2_H\d r\\
&\le c\int_0^t\E\big[ \Psionem(\umt(r),\vmt(r))+2m\|\Phi_1(\umt(r))\|_{L^1}\big]\d r.
\end{align*}
With regard to the last two terms on the right-hand side of~\eqref{eqn:Ito:d.Psi_1m:irreducibility}, we invoke Young's inequality to see that (here $m\le 1$)
\begin{align*}
&-2m\alpha_{\nbar}\la P_{\nbar}(\umt(t)-\Gammau(t)),\vmt(t)\ra_H-\alpha_{\nbar}\la P_{\nbar}(\umt(t)-\Gammau(t)),\umt(t)\ra_H\\
&\le  m\|\vmt(t)\|^2_H+m\alpha_{\nbar}^2\| \umt(t)-\Gammau(t)\|^2_H+\alpha_{\nbar}\|\umt(t)\|^2_H+\alpha_{\nbar}\|\Gammau(t)\|^2_H\\
&\le m\|\vmt(t)\|^2_H +c\|\umt(t)\|^2_H+c\|\Gammau(t)\|^2_H.
\end{align*}
In the above, $c>0$ does not depend on $m$ and $t$.

Altogether, from~\eqref{eqn:Ito:d.Psi_1m:irreducibility}, we arrive at
\begin{align*}
&\E\Big[\sup_{t\in[0,T]} \Psionem(\umt(t),\vmt(t))+2m\|\Phi_1(\umt(t))\|_{L^1}\Big]\\
& \le \Psionem(u_0,v_0)+2m\|\Phi_1(u_0)\|_{L^1}+c\,\E\sup_{t\in[0,T]}\|\Gammau(t)\|^2_H+cT\\
&\qquad +\int_0^T\E\big[ \Psionem(\umt(r),\vmt(r))+2m\|\Phi_1(\umt(r))\|_{L^1}\big]\d r.
\end{align*}
In view of Lemma~\ref{lem:irreducibility:Gamma^m(t)}, part 2, cf.~\eqref{ineq:moment-bound:Gamma^m(t):H^2:sup_[0,T]}, we obtain
\begin{align*}
&\E\Big[\sup_{t\in[0,T]} \Psionem(\umt(t),\vmt(t))+2m\|\Phi_1(\umt(t))\|_{L^1}\Big]\\
& \le C(\Psionem(u_0,v_0)+2m\|\Phi_1(u_0)\|_{L^1}+1)e^{cT},
\end{align*}
for some positive constants $C,c$ independent of $m$, $T$ and $(u_0,v_0)$. Recalling $V_m$ as in~\eqref{form:V_m} is equivalent to $\Psionem(u,v)+2m\|\Phi_1(u)\|_{L^1}$, cf.~\eqref{cond:V_m<Psi_1m<V_m}, this produces~\eqref{ineq:irreducibility:shift:sup_[0,t]}, thereby finishing the proof.

\end{proof}

\subsection{Asymptotic strong Feller property} \label{sec:geometric-ergodicity:asymptotic-Feller}

Having obtained irreducibility, we turn to the asymptotic strong Feller property. More specifically, the result in this subsection is given below.

\begin{proposition} \label{prop:asymptotic-Feller}
Under the same hypothesis of Theorem \ref{thm:geometric-ergodicity:mass}, there exist positive constants $c,C$ such that for all positive constants $N$ large, $\beta,$ $\varepsilon$, $\gamma$ sufficiently small, the following holds for for all $U,\xi\in\Hcal^1$, $f\in C^1_b(\Hcal^1;\rbb)$ satisfying $[f]_{\emph{Lip},\dmnb}\le 1$, and $m$ sufficiently small
\begin{align} 
&\big|\la D P_t^m f(U),\xi\ra_{\Hcal^1}\big| \nt\\
&\le C\, \Big(N\, e^{-ct} \exp\big\{C\big(e^{-\gamma t}+\varepsilon\big)\beta V_m(U) \big\}+\|f\|_{L^\infty}\exp \big\{C\varepsilon\beta V_m(U)   \big\}\Big)   \label{ineq:asymptotic-Feller}  \\
&\qquad\times\sqrt{m\|\pi_1\xi\|^2_{H^1}+m^2 \|\pi_2\xi\|^2_{H}+\|\pi_1\xi\|^2_{H}}.\nt 
\end{align}
Here, $[f]_{\emph{Lip},\dmnb}$ is defined in~\eqref{form:Lipschitz}, $\dmnb$ is the distance defined in~\eqref{form:d^m_(N,beta)}, and $V_m$ is the function as in~\eqref{form:V_m}.
\end{proposition}

\begin{remark} We note that the estimate~\eqref{ineq:asymptotic-Feller} is slightly different from the usual approach \cite{hairer2006ergodicity,hairer2008spectral,hairer2011theory,
hairer2011asymptotic} where it is sufficient to derive a bound on $\|DP^m_tf(U)\|_{\Hcal^1}=\sup_{\|\xi\|_{\Hcal^1}=1}\big|\la D P_t^m f(U),\xi\ra_{\Hcal^1}\big|$. In our settings, since we will employ~\eqref{ineq:asymptotic-Feller} to obtain the contracting of $P^m_t$ with respect to $\dmnb$, which is related to $\varrho^m_\beta$ defined in~\eqref{form:varrho^m_beta}, it is crucial to keep track of the dependence on $m$ and $\xi$, hence the right--hand side of~\eqref{ineq:asymptotic-Feller}. See the proof of Proposition~\ref{prop:contracting-d-small}, part 1, for a further discussion of this point.

\end{remark}

Before diving into the proof of Proposition~\ref{prop:asymptotic-Feller}, for the reader's convenience, we will briefly sketch the main approach that was developed in \cite{hairer2006ergodicity,hairer2008spectral,hairer2011theory}. For any  $\xi$ in $\Hcal^1$, we denote by $J^m_{0,t}\xi$, the derivative of $(\um(t),\vm(t))$ with respect to the initial condition $U_0\in\Hcal^1$, along the direction $\xi$. We observe that $J^m_{0,t}\xi$ satisfies the following random equation
\begin{align}
\ddt \Jmu\xi & = \Jmv\xi,  \nt \\
m\,\ddt \Jmv\xi &=-A\,\Jmu\xi-\Jmv\xi+\f'(\um(t))\Jmu\xi, \label{eqn:J}  \\ 
J^m_{0,0}\xi&=\xi.\nt 
\end{align}
Moreover, for any adapted process $\zeta\in L^2([0,t];H)$, we introduce $A_{0,t}^m  \zeta(t)$ the Malliavin derivative of $(\um(t),\vm(t))$ with respect to $w(t)$ along the path $\zeta(t)$. We note that $A_{0,t}^m\zeta$ is the solution of
\begin{align} 
\ddt \Amu\zeta & =  \Amv\zeta,  \nt \\
m\,\ddt \Amv\zeta &=-A\,\Amu\zeta-\Amv\zeta+\f'(\um(t))\Amu\zeta +Q\zeta(t), \label{eqn:A:Malliavin} \\
A^m_{0,0}\zeta&=0.\nt
\end{align}
Denote $\rho^m (t)=J_{0,t}^m\xi-A_{0,t}^m\zeta  $. Then $\rho^m(t)$ solves
\begin{align}
\ddt \rhomu(t) & =  \rhomv(t),  \nt \\
m\,\ddt \rhomv(t) &=-A\, \rhomu(t)-\rhomv(t)+\f'(\um(t))\rhomu(t) -Q\zeta(t),  \label{eqn:rho-JA} \\
\rho^m(0)&=\xi.\nt
\end{align}
We observe that for $f\in C^1_b(\Hcal^1)$ and $U,\xi\in\Hcal^1$,
\begin{align}
\la D P^m_t  f(U),\xi\ra_{\Hcal^1}&= \E\la D f(\um(t),\vm(t) ),J_{0,t}^m \xi\ra_{\Hcal^1}    \nt\\
&=\E\la Df(\um(t),\vm(t)),\rho^m(t)\ra_{\Hcal^1}+\E\la Df(\um(t),\vm(t)),A_{0,t}^m\zeta\ra_{\Hcal^1}  \nt\\
&=\E\la Df(\um (t),\vm(t)),\rho^m(t)\ra_{\Hcal^1}+\E f(\um(t),\vm(t))\int_0^t\la \zeta(s),Q\d w(s)\ra_H,\label{eqn:Malliavin-by-part}
\end{align}
where the last implication follows from the Malliavin integration by part and the integral is understood in Skorohod sense, should $\zeta(t)$ be non adapted. It is now a control problem to find $\zeta(\cdot)\in L^2([0,\infty);H)$ such that
\[\lim_{t\to\infty}\E\la Df(\um (t),\vm(t)),\rho^m(t)\ra_{\Hcal^1}=0,\quad\text{and}\quad \E\Big|\int_0^\infty\close\langle \zeta(s),Q\d w(s)\rangle_H\Big|^2<\infty.\]
Together with~\eqref{eqn:Malliavin-by-part}, the choice of $\zeta(t)$ will guarantee that the Markov semigroup $P_t^m $ satisfies the asymptotic strong Feller property as stated in~\eqref{ineq:asymptotic-Feller}. In the proof of Proposition~\ref{prop:asymptotic-Feller} below, we will construct such functions to meet our requirement.

\begin{proof}[Proof of Proposition~\ref{prop:asymptotic-Feller}]
Let $f\in C^1(\Hcal^1;\rbb)$ satisfy $[f]_{\text{Lip}_{\dmnb}}\le 1$. Recalling~\eqref{form:Lipschitz}, note that
\begin{align*}
[f]_{\text{Lip},\dmnb}=\sup_{U\neq \Ut}\frac{f(U)-f(\Ut)}{\dmnb(U,\tilde{U})}=\sup_{U\neq \Ut}\frac{\big(f(U)-f(0)\big)-\big(f(\tilde{U})-f(0)\big)}{\dmnb(U,\tilde{U})}.
\end{align*}
Thus, we may assume without loss of generality that $f(0)=0$. In particular, since $\dmnb\le 1$, cf.~\eqref{form:d^m_(N,beta)}, for all $U\in\Hcal^1$
\begin{align*}
|f(U)|=|f(U)-f(0)|\le \dmnb(U,0)\le 1.
\end{align*}
Furthermore, by choosing the path 
\begin{equation} \label{form:gamma_*}
\gamma_*(s)=s\cdot U+(1-s)\cdot\tilde{U},
\end{equation}
we see that
\begin{align*}
|f(U)-f(\Ut)|\le \dmnb(U,\Ut)&\le N\varrho^m_\beta(U,\Ut)\\
&\le N\int_0^1 e^{\beta V_m(\gamma_*(s))}\d s \cdot \big(m\|u-\ut\|^2_{H^1}+m^2\|v-\vt\|^2_{H}+\|u-\ut\|^2_{H}\big)^{1/2}.
\end{align*}
It follows that
\begin{align} \label{ineq:<df(U),Utilde>}
\la D f(U),\Ut\ra_{\Hcal^1}\le 2 N e^{\beta V_m(U)}  \big(m\|\ut\|^2_{H^1}+m^2\|\vt\|^2_{H}+\|\ut\|^2_{H}\big)^{1/2}.
\end{align}

Next, in view of~\eqref{eqn:J}--\eqref{eqn:A:Malliavin}--\eqref{eqn:rho-JA}, we choose $\zeta$ as
\begin{align} \label{form:zeta(t)}
\zeta(t)=\alpha_{\nbar} Q^{-1}P_{\nbar} \rhomu(t), 
\end{align}
where $\alpha_{\nbar}$ is the eigenvalue of $A$ as in~\eqref{cond:phi:ergodicity}, and $P_{\nbar}$ is the projection defined in~\eqref{form:P_n.u}. We note that this choice is possible thanks to the condition~\eqref{cond:Q:ergodicity}, that is $Q$ is invertible on span$\{e_1,\dots,e_{\nbar}\}$. With this choice of $\zeta$, observe that~\eqref{eqn:rho-JA} is reduced to
\begin{align}
\ddt \rhomu(t) & =  \rhomv(t),  \nt \\
m\,\ddt \rhomv(t) &=-A\, \rhomu(t)-\rhomv(t)+\f'(\um(t))\rhomu(t) -\alpha_{\nbar} P_{\nbar} \rhomu(t),  \label{eqn:rho-JA:eta} \\
\rho^m(0)&=\xi.\nt
\end{align}
To estimate $\rho^m(t)$, recalling $\Psionem$ defined in~\eqref{form:Psi_1m}, we employ~\eqref{eqn:rho-JA:eta} and compute 
\begin{align}
\ddt \Psionem(\rho^m(t))
&=-\|\rhomu(t)\|^2_{H^1}-m\|\rhomv(t)\|^2_H \nt \\
&\qquad+ 2m\la \f'(\um(t))\rhomu(t),\rhomv(t)\ra_H-2m\alpha_{\nbar}\la P_{\nbar}\rhomu(t),\rhomv(t)\ra_H   \nt \\
&\qquad +\la \f'(\um(t)\rhomu(t),\rhomu(t)\ra_H-\alpha_{\nbar}\| P_{\nbar}\rhomu(t)\|^2_H. \label{eqn:Ito:Psi_1m:rho^m}
\end{align}
In view of $\alpha_{\nbar}$ as in~\eqref{cond:phi:ergodicity}, we pick
\begin{equation*}
\ebar = \frac{\alpha_{\nbar}-a_\f}{2\alpha_{\nbar}}\in (0,1/2),
\end{equation*}
where $a_\f=\sup_{x\in\rbb}\f'(x)$. We invoke~\eqref{cond:phi:sup.phi'<a_f} together with Sobolev embedding
\begin{align*}
&-\|A^{1/2}\rhomu(t)\|^2_H+\la \f'(\um(t))\rhomu(t) ,\rhomu(t)\ra_H-\alpha_{\nbar}\| P_{\nbar}\rhomu(t)\|^2_H\\
&\le -\ebar\|A^{1/2}\rhomu(t)\|^2_H-(1-\ebar)\|(I-P_{\nbar})A^{1/2}\rhomu(t)\|_H+a_\f\|\rhomu(t)\|_H^2-\alpha_{\nbar}\| P_{\nbar}\rhomu(t)\|^2_H\\
&\le -\ebar\|A^{1/2}\rhomu(t)\|^2_H-(1-\ebar)\alpha_{\nbar}\|(I-P_{\nbar})\rhomu(t)\|_H+a_\f\|\rhomu(t)\|_H^2-\alpha_{\nbar}\| P_{\nbar}\rhomu(t)\|^2_H\\
&=  -\ebar\|A^{1/2}\rhomu(t)\|^2_H - \tfrac{1}{2}(\alpha_{\nbar}-a_\f)\|\rhomu(t)\|^2_H.
\end{align*}
To estimate the cross term $\la P_{\nbar}\rhomu(t),\rhomv(t)\ra_H$ on the right--hand side of~\eqref{eqn:rho-JA:eta}, we employ Cauchy--Schwarz inequality and obtain
\begin{align*}
-2m\alpha_{\nbar}\la P_{\nbar}\rhomu(t),\rhomv(t)\ra_H \le \tfrac{1}{4}m\|\rhomv(t)\|^2_H+4m\alpha_{\nbar}^2\|\rhomu(t)\|^2_H.
\end{align*}
Concerning the cross term $\la \f'(\um(t))\rhomu(t),\rhomv(t)\ra_H$, for small $\varepsilon\in(0,1)$ to be chosen later, we employ condition~\eqref{cond:phi:phi'=O(x^(lambda-1))} making use of the embedding $H^1 \subset L^6$ and Holder's inequality to estimate as follows:
\begin{align*}
& 2m\la \f'(\um(t))\rhomu(t),\rhomv(t)\ra_H\\
  & \le  c\, m \la |\rhomu(t)|\cdot (|\um(t)|^{\lambda-1}+1 ),|\rhomv(t)|\ra_H\\
 &\le c\, m\|\rhomu(t)\|_{H^1}\|\rhomv(t)\|_H\big( \|\um(t)\|^{\lambda-1}_{H^1}+1    \big)\\
&\le  m \|\rhomu(t)\|_{H^1}^2 \cdot \varepsilon\beta\|\um(t)\|^{2}_{H^1} + \tfrac{1}{4}m\|\rhomv(t)\|^2_H+C_{\beta,\varepsilon} m\|\rhomu(t)\|^2_{H^1}.
\end{align*}
We emphasize that $C_{\beta,\varepsilon}$ does not depend on $m$. Collecting the above estimates together with~\eqref{eqn:Ito:Psi_1m:rho^m} and take $m$ sufficiently small, we obtain 
\begin{align}
&\ddt \Psionem(\rho^m(t))   \nt \\
 &\le -\ebar\|A^{1/2}\rhomu(t)\|^2_H - \tfrac{1}{4}(\alpha_{\nbar}-a_\f)\|\rhomu(t)\|^2_H -\tfrac{1}{2}m\|\rhomv(t)\|^2_H \nt  \\
&\qquad+  m \|\rhomu(t)\|_{H^1}^2 \cdot \varepsilon\beta\|\um(t)\|^{2}_{H^1}. \label{ineq:d.Psi_1m:rho^m}
\end{align}
It follows that 
\begin{align}
\ddt \Psionem(\rho^m(t))  \le  \big(-c + \varepsilon\beta\|\um(t)\|^{2}_{H^1}  \big)\Psionem(\rho^m(t)), \label{ineq:d.Psi_1m:rho^m:a}
\end{align}
holds for some positive constant $c$ independent of $m$, $\beta$ and $\varepsilon$. Recalling $\rho^m(0)=\xi $, we deduce
\begin{align*}
\Psionem(\rho^m(t))  \le e^{-c t+\varepsilon\beta\int_0^t \|\um(r)\|^{2}_{H^1}\d r }\Psionem(\xi),
\end{align*}
whence
\begin{align*}
\E\Psionem(\rho^m(t))  \le e^{-c t}\Psionem(\xi) \E e^{\varepsilon\beta\int_0^t \|\um(r)\|^{2}_{H^1}\d r}.
\end{align*}
In light of~\eqref{ineq:moment-bound:H:exponential:int_0^t}, for $\beta$ sufficiently small, we deduce
\begin{align} \label{ineq:Psi_1m(rho^m(t))}
&\E\Psionem(\rho^m(t))   \nt  \\
& \le 2\,e^{-c t}\Psionem(\xi) \exp\big\{ 2\varepsilon \beta\big[ \Psionem(u,v)+2m\|\Phi_1(u)\|_{L^1}\big] \big\}.
\end{align}
holds for some positive constant $c =c(\alpha_{\nbar})$ independent of $\varepsilon,\beta,m,t$ and $\xi$. Recalling $V_m(u,v)$ is equivalent to $\Psionem(u,v)+\|\Phi_1(u)\|_{L^1}$, cf.~\eqref{cond:V_m<Psi_1m<V_m}, we combine~\eqref{ineq:Psi_1m(rho^m(t))} with~\eqref{ineq:moment-bound:H:exponential:super-Lyapunov} and~\eqref{ineq:<df(U),Utilde>} to infer
\begin{align*}
&\E\la Df(\um (t),\vm(t)),\rho^m(t)\ra_{\Hcal^1}\\
& \le 2 N \E\Big[ e^{\beta V_m(\um(t),\vm(t))}  \big(m\|\rhomu(t)\|^2_{H^1}+m^2\|\rhomv(t)\|^2_{H}+\|\rhomu(t)\|^2_{H}\big)^{1/2} \Big]\\
&\le C\,N \big(\E \exp\big\{ C\beta [\Psionem(\um(t),\vm(t))+2m\|\Phi_1(\um(t))\|_{L^1}]   \big\}\big)^{1/2}\big( \E\Psionem(\rho^m(t)) \big)^{1/2}\\
&\le C \, N e^{-ct} \exp\big\{ \big(C\,e^{-\gamma t}+\varepsilon\big)\beta[\Psionem(u,v)+2m\|\Phi_1(u)\|_{L^1}]  \big\}  \sqrt{\Psionem(\xi)}\\
&\le C\,  N e^{-ct }\exp\big\{ C\big(e^{-\gamma t}+\varepsilon\big)\beta V_m(U)\big\}\sqrt{\Psionem(\xi)}.
\end{align*}
In the above, $\gamma$ is the constant as in~\eqref{ineq:moment-bound:H:exponential:super-Lyapunov}. Since $\Psionem(\xi) $ is dominated by $m\|\pi_1\xi\|^2_{H^1}+m^2 \|\pi_2\xi\|^2_{H}+\|\pi_1\xi\|^2_{H}$, we deduce further 
\begin{align}
&\E\la Df(\um (t),\vm(t)),\rho^m(t)\ra_{\Hcal^1} \nt \\
& \le C\, N\, e^{-ct} \exp\big\{C\big(e^{-\gamma t}+\varepsilon\big)\beta V_m(U) \big\}\sqrt{m\|\pi_1\xi\|^2_{H^1}+m^2 \|\pi_2\xi\|^2_{H}+\|\pi_1\xi\|^2_{H}}.\label{ineq:E<Df,rho^m>}
\end{align}
We emphasize that $C,c$ are independent of $N,\gamma,\beta,\varepsilon,
 m$, $U$ and $\xi$.

Next, to estimate $\zeta(t)$ defined in~\eqref{form:zeta(t)}, we invoke~\eqref{ineq:Psi_1m(rho^m(t))} and It\^o's isometry to estimate
\begin{align}
\E\Big|\int_0^\infty\close \la \zeta(t),Q\d w(t)\ra_H\Big|^2 &= \int_0^\infty \close\E\|Q\zeta(t)\|^2_H\d t    \nt \\
&\le \alpha_{\nbar}^2 \int_0^\infty\close \E\|\rhomu(t)\|^2_H\d t  \nt \\
&\le C \big(m\|\pi_1\xi\|^2_{H^1}+m^2 \|\pi_2\xi\|^2_{H}+\|\pi_1\xi\|^2_{H}\big)\exp \big\{C\varepsilon\beta V_m(U)   \big\}.\label{ineq:int_0^infty.zeta}
\end{align}

Finally, we collect~\eqref{eqn:Malliavin-by-part},~\eqref{ineq:E<Df,rho^m>} and~\eqref{ineq:int_0^infty.zeta} to obtain the bound
\begin{align*}
\big|\la D P^t_m f(U),\xi\ra_{\Hcal^1}\big|
&\le C\Big( N\, e^{-ct} \exp\big\{C\big(e^{-\gamma t}+\varepsilon\big)\beta V_m(U) \big\}+\|f\|_{L^\infty}\exp \big\{C\varepsilon\beta V_m(U)   \big\}\Big)\\
&\qquad\times\sqrt{m\|\pi_1\xi\|^2_{H^1}+m^2 \|\pi_2\xi\|^2_{H}+\|\pi_1\xi\|^2_{H}}.
\end{align*}
This produces~\eqref{ineq:asymptotic-Feller}, as claimed.

\end{proof}

\subsection{Proof of Theorem \ref{thm:geometric-ergodicity:mass}  } \label{sec:geometric-ergodicity:proof}

In this section, we establish the uniform exponential convergent rate of $P^m_t$ toward $\num$. Two of the main ingredients, namely, $\dmnb-$contracting property and $\dmnb-$small sets, are given in the following proposition: \cite{butkovsky2020generalized,hairer2011asymptotic,
kulik2017ergodic,kulik2015generalized}

\begin{proposition} \label{prop:contracting-d-small} Under the same hypothesis of Theorem \ref{thm:geometric-ergodicity:mass}. Then, for all $N$ sufficiently large, $\beta$ sufficiently small, the followings hold for all $m$ sufficiently small:

1. The distance $\dmnb$ as in~\eqref{form:d^m_(N,beta)} is contracting for $P_t$.
That is, there exists $t_1=t_1(N,R)>0$ independent of $m$ such that
\begin{align}\label{ineq:contracting-d-small:contracting}
\W_{\dmnb}(P_t^m(U_0,\cdot),P_t^m(\Ut_0,\cdot))\le \tfrac{1}{2} \dmnb(U_0,\Ut_0),\quad t\ge t_1,
\end{align}
whenever $\dmnb(U,\Ut)<1$.

2.  For all $R$, there exists $t_2=t_2(N,\beta,R)>0$ independent of $m$ such that the set $B_R^m=\{U:V_m(U)\le R\}$, where $V_m$ is defined in~\eqref{form:V_m}, is $\dmnb$--small.
That is 
\begin{align}\label{ineq:contracting-d-small:d-small}
\sup_{U,\Ut\in B_R^m}\W_{\dmnb}(P_t^m(U,\cdot),P_t^m(\Ut,\cdot))\le 1-\gamma_1,\quad t\ge t_2,
\end{align}
for some $\gamma_1=\gamma_1(t,N,\beta,R)$ independent of $m$.

\end{proposition}

Assuming the result of Proposition~\ref{prop:contracting-d-small}, we now conclude Theorem \ref{thm:geometric-ergodicity:mass}, whose proof follows along the lines of \cite[Theorem 4.8]{hairer2011asymptotic}. For the sake of completeness, we sketch the main ideas without going into detail. See also \cite{glatt2022short,glatt2021mixing}.

\begin{proof}[Sketch of the proof of Theorem~\ref{thm:geometric-ergodicity:mass}]
Fix $N$ large and $\beta$ small enough such that Proposition~\ref{prop:contracting-d-small} holds. We first claim that there exist positive constants $t_3$ and $C$ independent of $m$ and $(u_0,v_0)$ such that
\begin{align} \label{ineq:moment-bound:H:exponential:super-Lyapunov:a}
\E\,e^{\beta V_m(\um(t),\vm(t))} \le \tfrac{1}{8} e^{\frac{1}{4}\beta V_m(u_0,v_0)}+C,\quad t\ge t_3.
\end{align}
To see this, fixing $\varepsilon>0$ to be chosen later, we invoke~\eqref{cond:V_m<Psi_1m<V_m} and Holder's inequality to estimate
\begin{align*}
&\E\,e^{\beta V_m(\um(t),\vm(t))}\\
 & \le \E\,e^{c\beta[ \Psionem(\um(t),\vm(t))+\|\Phi_1(\um(t))\|_{L^1}]  } \\
&\le \Big|\E\,e^{\varepsilon\beta[ \Psionem(\um(t),\vm(t))+\|\Phi_1(\um(t))\|_{L^1}]  }  \Big|^{1/2}\Big|\E\,e^{(2c-\varepsilon)\beta[ \Psionem(\um(t),\vm(t))+\|\Phi_1(\um(t))\|_{L^1}]  }\Big|^{1/2}.
\end{align*}
From~\eqref{ineq:moment-bound:H:exponential}, observe that
\begin{align*}
\E\,e^{\varepsilon\beta[ \Psionem(\um(t),\vm(t))+\|\Phi_1(\um(t))\|_{L^1}]  } 
&\le Ce^{-ct}e^{\varepsilon\beta[ \Psionem(u_0,v_0)+\|\Phi_1(u_0)\|_{L^1}]  } +C\\
&\le Ce^{-ct}e^{C\varepsilon \beta V_m(u_0,v_0)} +C,
\end{align*}
where in the last implication above, we once again employed~\eqref{cond:V_m<Psi_1m<V_m}. Likewise, for $\gamma>0$ small enough,~\eqref{ineq:moment-bound:H:exponential:super-Lyapunov} implies
\begin{align*}
\E\,e^{(2c-\varepsilon)\beta[ \Psionem(\um(t),\vm(t))+\|\Phi_1(\um(t))\|_{L^1}]  } 
&\le Ce^{(2c-\varepsilon)e^{-\gamma t}\beta[ \Psionem(u_0,v_0)+\|\Phi_1(u_0)\|_{L^1}]  } \\
&\le Ce^{C(2c-\varepsilon)e^{-\gamma t} \beta V_m(u_0,v_0)}\\
&\le \tfrac{1}{16} e^{2C(2c-\varepsilon)e^{-\gamma t} \beta V_m(u_0,v_0)}+16C^2.
\end{align*}
It follows that
\begin{align*}
&\E\,e^{\beta V_m(\um(t),\vm(t))}\\
&\le Ce^{-ct}e^{C\varepsilon \beta V_m(u_0,v_0)}+ \tfrac{1}{16}e^{2C(2c-\varepsilon)e^{-\gamma t} \beta V_m(u_0,v_0)} +16C^2+C.
\end{align*}
By taking $\varepsilon$ sufficiently small and $t$ large enough, e.g.,
\begin{align*}
Ce^{-ct}<\tfrac{1}{16},\quad  C\varepsilon<\tfrac{1}{4},\quad 2C(2c-\varepsilon)e^{-\gamma t}<\tfrac{1}{4},
\end{align*}
we obtain
\begin{align*}
\E\,e^{\beta V_m(\um(t),\vm(t))} &\le \tfrac{1}{8} e^{\frac{1}{4}\beta V_m(u_0,v_0)}+16C^2+C.
\end{align*}
That is, we infer the existence of a positive constant $t_3$ such that~\eqref{ineq:moment-bound:H:exponential:super-Lyapunov:a} holds.

Next, for $\lambda>0$, we introduce the function--like distance
\begin{align*}
\dmtnbl(U,\Ut)=\sqrt{\dmtnb(U,\Ut)\big[1+ \lambda e^{\beta V_m(U)} +\lambda e^{\beta V_m(\Ut) }\big]}, \quad U,\Ut\in\Hcal^1.
\end{align*}
Since~\eqref{ineq:contracting-d-small:contracting},~\eqref{ineq:contracting-d-small:d-small} and~\eqref{ineq:moment-bound:H:exponential:super-Lyapunov:a} hold uniformly with respect to $m$, following the proof of \cite[Theorem 4.8]{hairer2011asymptotic}, there exists a time constant $T^*=T^*(N,\beta,\lambda)$ independent of $m$ such that
\begin{align*}
\W_{\dmtnbl}\big( (P^m_{T^*})^*\nu_1, (P^m_{T^*})^*\nu_2\big) \le \alpha_*\W_{\dmtnbl}\big( \nu_1, \nu_2\big),\quad  \nu_1,\nu_2\in \Pcal r(\Hcal^1),
\end{align*}
holds for some $\alpha^*=\alpha^*(T^*,N,\beta,\lambda)\in (0,1)$ independent of $m$. By Markov property, we obtain
\begin{align} \label{ineq:geometric-ergodicity:mass:lambda}
\W_{\dmtnbl}\big( (P^m_{nT^*})^*\nu_1, (P^m_{nT^*})^*\nu_2\big) \le \alpha_*^n\W_{\dmtnbl}\big( \nu_1, \nu_2\big),\quad  \nu_1,\nu_2\in \Pcal r(\Hcal^1).
\end{align}
Note that $\W_{\dmtnbl}$ is actually equivalent to $\W_{\dmtnb}$ thanks to the fact that
\begin{align*}
\tfrac{1}{\max\{\lambda,1\} }  \dmtnbl(U,\Ut) \le \dmtnb(U,\Ut) \le \tfrac{1}{ \min\{\lambda,1\} } \dmtnbl(U,\Ut).
\end{align*}
This together with~\eqref{ineq:geometric-ergodicity:mass:lambda} produces the exponential decaying rate~\eqref{ineq:geometric-ergodicity:mass}, as claimed.

\end{proof}

We now give the proof of Proposition~\ref{prop:contracting-d-small} whose argument is standard and can be found in literarture \cite{butkovsky2020generalized,hairer2011asymptotic,
kulik2017ergodic,kulik2015generalized}.

\begin{proof}[Proof of Proposition~\ref{prop:contracting-d-small}]

1. Fix $\beta$ sufficiently small as in Lemma~\ref{lem:moment-bound:H} and Proposition~\ref{prop:asymptotic-Feller} and let $N>0$ be given and be chosen later. Consider $U_0,\,\Ut_0\in\Hcal^1$  such that $\dmnb(U_0,\Ut_0)<1$. By the definition~\eqref{form:d^m_(N,beta)} of $\dmnb$, observe that
\begin{align*}
\dmnb(U_0,\Ut_0) = N\varrho^m_\beta(U_0,\Ut_0),
\end{align*}
where $\varrho^m_\beta$ is the metric defined in~\eqref{form:varrho^m_beta}. Since $\dmnb$ is a metric \cite{hairer2008spectral,hairer2011asymptotic}, in view of the dual formula~\eqref{form:W_d:dual-Kantorovich}, the desired bound~\eqref{ineq:contracting-d-small:contracting} is equivalent to
\begin{align} \label{ineq:contracting-d-small:contracting:P^m_t.f(U)}
|P^m_tf(U_0)-P^m_tf(\Ut_0)| \le \tfrac{1}{2}N\varrho^m_\beta(U_0,\Ut_0),
\end{align}
which holds for all $f\in C^1_b(\Hcal^1)$ satisfying $[f]_{\text{Lip},\dmnb}\le 1$ (see~\eqref{form:Lipschitz}). Moreover, from~\eqref{form:Lipschitz}, we note that
\begin{align*}
[f]_{\text{Lip},\dmnb} =\sup_{U\neq \Ut}\frac{f(U)-f(\Ut)}{\dmnb(U,\Ut)}=\sup_{U\neq \Ut}\frac{f(U)-f(0)-(f(\Ut)-f(0))}{\dmnb(U,\Ut)},
\end{align*}
it suffices to prove~\eqref{ineq:contracting-d-small:contracting:P^m_t.f(U)} for those functions $f$'s with $f(0)=0$. In particular, this implies $\|f\|_{L^\infty}\le 1$ since
\begin{align*}
|f(U)|=|f(U)-f(0)|\le \dmnb(U,0)\le 1,\quad U\in\Hcal^1.
\end{align*}

Next, we recall from Proposition~\ref{prop:asymptotic-Feller} to see that for all $U,\xi\in\Hcal^1$,
\begin{align*}
\big|\la D P_t^m f(U),\xi\ra_{\Hcal^1}\big|
&\le C\Big( N\, e^{-ct} \exp\big\{C\big(e^{-\gamma t}+\varepsilon\big)\beta V_m(U) \big\}+\|f\|_{L^\infty}\exp \big\{C\varepsilon\beta V_m(U)   \big\}\Big)\\
&\qquad\times\sqrt{m\|\pi_1\xi\|^2_{H^1}+m^2 \|\pi_2\xi\|^2_{H}+\|\pi_1\xi\|^2_{H}}, 
\end{align*}
holds for all $\varepsilon$ and $\gamma$ sufficiently small. Here $C$ and $c$ are independent of $\varepsilon$, $\gamma$, $N$ and $\beta$. By choosing $\varepsilon$ appropriately and $t$ large enough, e.g.,
\begin{align*}
 C\big(e^{-\gamma t}+\varepsilon\big)\le 1,
\end{align*}
we observe that (recalling $\|f\|_{L^\infty}\le 1$)
\begin{align*}
&\big|\la D P_t^m f(U),\xi\ra_{\Hcal^1}\big|\\
&\le C\Big( N\, e^{-ct} \exp\big\{\beta V_m(U) \big\}+\exp \big\{\beta V_m(U)   \big\}\Big)\sqrt{m\|\pi_1\xi\|^2_{H^1}+m^2 \|\pi_2\xi\|^2_{H}+\|\pi_1\xi\|^2_{H}}\\
&\le  N\, \exp\big\{\beta V_m(U) \big\}\big( Ce^{-ct}+\tfrac{C}{N} \big)\sqrt{m\|\pi_1\xi\|^2_{H^1}+m^2 \|\pi_2\xi\|^2_{H}+\|\pi_1\xi\|^2_{H}}.
\end{align*}
We may now choose $t_1$ and $N$ sufficiently large satisfying
\begin{align*}
Ce^{-ct_1}+\tfrac{C}{N} <\tfrac{1}{2},
\end{align*}
and obtain
\begin{align*}
\big|\la D P_t^m f(U),\xi\ra_{\Hcal^1}\big|&\le \tfrac{1}{2} N\, e^{\beta V_m(U) }\sqrt{m\|\pi_1\xi\|^2_{H^1}+m^2 \|\pi_2\xi\|^2_{H}+\|\pi_1\xi\|^2_{H}},\quad t\ge t_1.
\end{align*}

Turning to~\eqref{ineq:contracting-d-small:contracting:P^m_t.f(U)}, for any differentiable path $\gamma:[0,1]\to\Hcal^1$ connecting $U_0,\Ut_0$, we invoke the above estimate to infer
\begin{align*}
|P^m_tf(U_0)-P^m_tf(\Ut_0)| &= \int_0^1 \la D P^m_tf(\gamma(s)),\gamma'(s)\ra_{\Hcal^1}\d s \\
&\le \tfrac{1}{2} N \int_0^ 1  e^{\beta V_m(\gamma(s)) }\sqrt{m\|\pi_1\gamma'(s)\|^2_{H^1}+m^2 \|\pi_2\gamma'(s)\|^2_{H}+\|\pi_1\gamma'(s)\|^2_{H}}   \d s.
\end{align*}
Since the above estimate holds for arbitrarily such path $\gamma$, in view of the expression~\eqref{form:varrho^m_beta} for $\varrho^m_\beta$, we establish~\eqref{ineq:contracting-d-small:contracting:P^m_t.f(U)}. In turn, this produces~\eqref{ineq:contracting-d-small:contracting} and completes part 1.

2. With regard to~\eqref{ineq:contracting-d-small:d-small}, we first note that for $s\in[0,1]$, $\lambda\in[1,2)$ and $V_m$ as in~\eqref{form:V_m},
\begin{align*}
V_m(sU+(1-s)\Ut) \le 4V_m(U)+4V_m(\Ut).
\end{align*}
So, applying~\eqref{form:varrho^m_beta} to the path $\gamma^*(s)=sU+(1-s)\Ut$, we have
\begin{align}
\varrho^m_\beta(U,\Ut)&\le e^{4\beta V_m(U)+4\beta V_m(\Ut)}\sqrt{m\|u-\ut\|^2_{H^1}+m^2 \|v-\vt\|^2_{H}+\|u-\ut\|^2_{H}} \nt \\
&\le \big(e^{8\beta V_m(U)}+e^{8\beta V_m(\Ut)}\big)\sqrt{m\|u-\ut\|^2_{H^1}+m^2 \|v-\vt\|^2_{H}+\|u-\ut\|^2_{H}}.\label{ineq:varrho^m_beta<|U-U.tilde|}
\end{align}
The proof of~\eqref{ineq:contracting-d-small:d-small} now follows along the lines of~\cite[Lemma 5.3]{hairer2011asymptotic} and \cite[Theorem 2.4]{butkovsky2020generalized} tailored to our setting.
Let $U(t)$ and $\Ut(t)$ be the solutions of~\eqref{eqn:wave} with initial conditions $U_0$ and $\Ut_0$, respectively. Here, $U_0$ and $\Ut_0$ both belong to $B_R^m$. Let $(X,Y)$ be a coupling of $\big(P^m_t(U_0,\cdot),P^m_t(\Ut_0,\cdot)\big)$ such that $X$ and $Y$ are independent. Let $r>0$ be given and to be chosen later, by the definition~\eqref{form:W_d}, we invoke~\eqref{ineq:varrho^m_beta<|U-U.tilde|} to infer
\begin{align*}
&\W_{\dmnb}\!\big(P^m_t(U_0,\cdot),P^m_t(\Ut_0,\cdot)\big)\\
& \le \E \, \dmnb (X,Y) = \E \big[1\mi N\varrho^m_\beta(X,Y)\big]\\
&\le  N \sup_{U,\Ut\in B_r^m}\big(e^{8\beta V_m(U)}+e^{8\beta V_m(\Ut)}\big)\sqrt{m\|u-\ut\|^2_{H^1}+m^2 \|v-\vt\|^2_{H}+\|u-\ut\|^2_{H}}\\
&\qquad\times \P\big(\{X\in B_r^m \}\cap\{Y\in B_r^m\}\big) +\P\big(\{X\notin B_r^m \}\cup\{Y\notin B_r^m\}\big).
\end{align*}
In the above, $B^m_r=\{U:V_m(U)\le r\}$. Since $V_m(U)$ dominates $m\|u\|^2_{H^1}+m^2 \|v\|^2_{H}+\|u\|^2_{H}$, we pick $r=r(N)$ sufficiently small such that
\begin{align*}
 & N \sup_{U,\Ut\in B_r^m}\big(e^{8\beta V_m(U)}+e^{8\beta V_m(\Ut)}\big)\sqrt{m\|u-\ut\|^2_{H^1}+m^2 \|v-\vt\|^2_{H}+\|u-\ut\|^2_{H}} \\
 &\le 4 N e^{8\beta r} r<\tfrac{1}{2}.
\end{align*}
As a consequence, we obtain
\begin{align*}
\W_{\dmnb}\!\big(P_t^m(U_0,\cdot),P_t^m(\Ut_0,\cdot)\big)&\le 1-\tfrac{1}{2}\P\big(\{X\in B_r^m\}\cap\{Y\in B_r^m\}\big)\\
&= 1-\tfrac{1}{2}\P\big(X\in B_r^m\big)\P\big(Y\in B_r^m\big).
\end{align*}
In the above, we simply employed the choice of $X$ and $Y$ being independent. By virtue of Proposition~\ref{prop:irreducibility}, for $t\ge T_1$ where $T_1$ is the time constant in~\eqref{ineq:irreducibility}, it holds that 
\begin{align*}
\inf_{U_0\in B_R^m}\P(V_m(U(t))\le r)\ge c(R,r,t),
\end{align*}
for some positive constant $c(R,r,t)$ independent of $U_0$. It follows that 
\begin{align*}
\sup_{U_0,\Ut_0\in B^m_R}\W_{\dmnb}\big(P_t^m(U_0,\cdot),P_t^m(\Ut_0,\cdot)\big)\le 1-\tfrac{1}{2}c(R,r,t)^2.
\end{align*}
This establishes~\eqref{ineq:contracting-d-small:d-small}, thereby finishing the proof.

\end{proof}



\section{Small mass limit} \label{sec:small-mass}

\subsection{Proof of Theorem~\ref{thm:nu^m->nu^0}} \label{sec:small-mass:nu^m->nu^0}

In this subsection, we establish Theorem~\ref{thm:nu^m->nu^0}. For the reader’s convenience, we summarize the idea of the proof of Theorem~\ref{thm:nu^m->nu^0}. The argument
essentially consists of four steps as follows.

\emph{Step 1}: we first show that under suitable moment bounds on the random initial conditions, cf \eqref{cond:U_0} below, the difference $\|\um(t)-\uo(t)\|_H$ converges to zero as $m\to 0$. This result appears in Proposition~\ref{prop:m->0:|u^m-u^0|}.

\emph{Step 2}: as a consequence, we obtain the convergence in $\dtnb(\um(t),\uo(t))$ as $m\to 0$, thanks to the fact that $\|\um(t)-\uo(t)\|_H$ dominates $\dtnb(\um(t),\uo(t))$. This is discussed in details in Corollary~\ref{cor:m->0:dtnb(u^m,u^0)}.

\emph{Step 3}: we show that the invariant measure $\num$ satisfies the condition~\eqref{cond:U_0}, allowing for the choice of $\num$ as an eligible random initial condition. This is established in Proposition~\ref{prop:regularity}.

\emph{Step 4}: we prove Theorem~\ref{thm:nu^m->nu^0} by showing that $\W_{\dtnb}(\pi_1\num,\nu^0)$ is dominated by $\dtnb(\um(t),\uo(t))$  where $t$ is chosen to be sufficiently large and $\um(0)=\uo(0)$ are both distributed as $\pi_1\num$.

For the sake of clarity, the proofs of Proposition~\ref{prop:m->0:|u^m-u^0|}, Corollary \ref{cor:m->0:dtnb(u^m,u^0)} and Proposition~\ref{prop:regularity} are deferred to the end of this subsection. We start by stating Proposition~\ref{prop:m->0:|u^m-u^0|} asserting the convergence of $\um(t)$ toward $\uo(t)$ in $H$.

\begin{proposition} \label{prop:m->0:|u^m-u^0|}
Under the same hypothesis of Theorem~\ref{thm:geometric-ergodicity:mass}, let $U_0^m=(u_0^m,v_0^m)\in \Hcal^1$ be a random variable satisfying for all $p\ge 1$, and for all $\beta>0$ sufficiently small
\begin{align}
\limsup_{m\to 0}\,&\Big(\E \big[\|u_0^m\|^2_{H^2}+m\|v_0^m\|^2_{H^1}\big]+\E\Psitwom(u_0^m,v_0^m)^p  \nt \\
&\qquad +\E\exp\big\{\beta \big[\Psionem(u_0^m,v_0^m) +2m\|\Phi_1(u_0^m)\|_{L^1}\big]\big\}\Big)<\infty, \label{cond:U_0}
\end{align}
where $\Phi_1$, $\Psionem$ and $\Psitwom$ are defined in~\eqref{cond:Phi_1}, \eqref{form:Psi_1m} and \eqref{form:Psi_2m}, respectively. Suppose that $(\um(t),\vm(t))$ and $\uo(t)$ are respectively the solutions of~\eqref{eqn:wave} and~\eqref{eqn:react-diff} with initial conditions $U_0^m$ and $u_0^m$. Then, for all $T>0$,
\begin{align} \label{lim:m->0:|u^m-u^0|}
\sup_{t\in[0,T]}\E \|\um(t)-\uo(t)\|_H\to 0,\quad{as  }\,\,m\to 0. 
\end{align}
\end{proposition}

As a consequence, we obtain the convergence of $\um(t)$ toward $\uo(t)$ in $\dtnb$.

\begin{corollary} \label{cor:m->0:dtnb(u^m,u^0)}
Under the same hypothesis of Theorem~\ref{thm:geometric-ergodicity:mass}, let $U_0^m=(u_0^m,v_0^m)\in \Hcal^1$ be a random variable satisfying condition~\eqref{cond:U_0}. Suppose that $(\um(t),\vm(t))$ and $\uo(t)$ are respectively the solutions of~\eqref{eqn:wave} and~\eqref{eqn:react-diff} with initial conditions $U_0^m$ and $u_0^m$. Then, for all $T,N>0$, and $\beta>0$ sufficiently small, 
\begin{align} \label{lim:m->0:dtnb(u^m,u^0)}
\sup_{t\in[0,T]}\E\, \dtnb(\um(t),\uo(t))\to 0,\quad{as  }\,\,m\to 0. 
\end{align}
In the above, $\dtnb$ is the distance in $H$ defined in~\eqref{form:d.tilde_(N,beta)}.
\end{corollary}

Next, we state the following result asserting that $\num$ satisfies the condition~\eqref{cond:U_0}.

\begin{proposition} \label{prop:regularity}
Let $\num$ be the unique invariant probability measure of~\eqref{eqn:wave} as in Theorem~\ref{thm:geometric-ergodicity:mass}. Then, for all $p\ge 1$ and $\beta>0$ sufficiently small,
 \begin{align} 
& \limsup_{m\to 0}\int_{\Hcal^1} \|u\|^2_{H^2}+m\|v\|^2_{H^1}+\Psitwom(u,v)^p \nt \\
 &\qquad\qquad+\exp\big\{\beta \big[\Psionem(u,v) +2m\|\Phi_1(u)\|_{L^1}\big]  \big\}\num(\emph{d} u,\emph{d}v)< \infty,\label{ineq:moment-bound:nu^m:H^1+H^2}
 \end{align}
where $\Phi_1$, $\Psionem$ and $\Psitwom$ are defined in~\eqref{cond:Phi_1},~\eqref{form:Psi_1m} and \eqref{form:Psi_2m}, respectively.
\end{proposition}

Assuming the above results, we are now in a position to conclude Theorem~\ref{thm:nu^m->nu^0} whose proof is standard and can be found in many previous works, e.g., \cite{cerrai2020convergence,foldes2017asymptotic,foldes2015ergodic,
foldes2016ergodicity,foldes2019large}.
\begin{proof}[Proof of Theorem~\ref{thm:nu^m->nu^0}]
Let $\W_{\dtnb}$ be the Wasserstein distance as in Theorem~\ref{thm:react-diff:geometric-ergodicity}. By the invariance of $\num$ and $\nu^0$, we invoke the generalized triangle inequality~\eqref{ineq:W_dtnbhalf<W_dtnb} to see that
\begin{align*}
\W_{\dtnb}\big( \pi_1\num,\nu^0    \big) &= \W_{\dtnb}\big(\pi_1(P^m_t)^*\num, (P^0_t)^*\nu^0 \big)  \\
&  \le C\Big[ \W_{\dtnbtwo}\big(\pi_1(P^m_t)^*\num, (P^0_t)^*(\pi_1\num) \big)  +\W_{\dtnbtwo}\big((P^0_t)^*(\pi_1\num), (P^0_t)^*\nu^0 \big)   \Big].
\end{align*}
In view of Theorem~\ref{thm:react-diff:geometric-ergodicity}, for $t$ sufficiently large independent of $\num$ and $\nu^0$,
\begin{align*}
C\W_{\dtnbtwo}\big((P^0_t)^*(\pi_1\num), (P^0_t)^*\nu^0 \big)  &  \le Ce^{-ct} \W_{\dtnb}\big(\pi_1\num, \nu^0 \big) \\
&\le \tfrac{1}{2}   \W_{\dtnb}\big(\pi_1\num, \nu^0 \big).   
\end{align*}
It follows that
\begin{align*}
\W_{\dtnb}\big( \pi_1\num,\nu^0    \big)\le  C \W_{\dtnbtwo}\big(\pi_1(P^m_t)^*\num, (P^0_t)^*(\pi_1\num) \big).
\end{align*}
From the definition~\eqref{form:W_d} of $\W_{\dtnb}$, we note that
\begin{align*}
\W_{\dtnbtwo}\big(\pi_1(P^m_t)^*\num, (P^0_t)^*(\pi_1\num) \big) \le \E\,\dtnbtwo(\um(t),\uo(t)) , 
\end{align*}
where $\um(0)=(u_0^m,v^m_0)$, $\uo(0)=u^m_0$ and $(u_0^m,v_0^m)\sim\num$. Note that
\begin{align*}
&\E \big[\|u_0^m\|^2_{H^2}+m\|v_0^m\|^2_{H^1}\big]+\E\Psitwom(u_0^m,v_0^m)^p +\E\exp\big\{\beta \big[\Psionem(u_0^m,v_0^m) +2m\|\Phi_1(u_0^m)\|_{L^1}\big]  \big\}\\
&= \int_{\Hcal^1}\|u\|^2_{H^2}+m\|v\|^2_{H^1}+\Psitwom(u,v)^p +\exp\big\{\beta \big[\Psionem(u,v) +2m\|\Phi_1(u)\|_{L^1}\big]  \big\}\num(\d u,\d v).
\end{align*}
In light of \eqref{ineq:moment-bound:nu^m:H^1+H^2}, it is clear that $U_0^m=(u^0_m,v^0_m)\sim \num$ satisfies the condition~\eqref{cond:U_0} for all $p\ge1$ and $\beta>0$ sufficiently small. By virtue of Corollary~\ref{cor:m->0:dtnb(u^m,u^0)}, we obtain
\begin{align*}
\E\,\dtnbtwo(\um(t),\uo(t)) \to 0,\quad \text{as }\,\, m\to 0,
\end{align*}
whence
\begin{align*}
\W_{\dtnb}\big( \pi_1\num,\nu^0    \big) \to 0,\quad \text{as}\,\,m\to 0.
\end{align*}
The proof is thus finished.

\end{proof}

Turning back to Proposition~\ref{prop:m->0:|u^m-u^0|}, we will make use of the following result from \cite[Section 6.1]{cerrai2020convergence} giving the convergence~\eqref{lim:m->0:|u^m-u^0|} under the additional condition that $\f$ is a Lipschitz function.
\begin{lemma}{ \cite[Limit (6.4)]{cerrai2020convergence} } \label{lem:m->0:|u^m-u^0|:Lipschitz}
Suppose that $Q$ satisfies Assumption~\eqref{cond:Q} and that $\f$ is Lipschitz. Let $U_0^m=(u_0^m,v_0^m)\in \Hcal^1$ be a random variable satisfying
\begin{align*}
\limsup_{m\to 0}\E\big[ \|u^m_0\|^2_{H^1}+m\|v^m_0\|^2_H\big]<\infty.
\end{align*}
Suppose that $(\um(t),\vm(t))$ and $\uo(t)$ are respectively the solutions of~\eqref{eqn:wave} and~\eqref{eqn:react-diff} with initial conditions $U_0^m$ and $u_0^m$. Then, for all $T>0$, Then, for all $T>0$,
\begin{align} \label{lim:m->0:|u^m-u^0|:Lipschitz}
\sup_{t\in[0,T]}\E \|\um(t)-\uo(t)\|_H\to 0,\quad{as}\,\,m\to 0. 
\end{align}
\end{lemma}

We now provide the proof of Proposition~\ref{prop:m->0:|u^m-u^0|} by removing the Lipschitz condition making use of suitable energy estimates in Section \ref{sec:moment-bound} and Appendix~\ref{sec:react-diff}.

\begin{proof}[Proof of Proposition \ref{prop:m->0:|u^m-u^0|} ]
To remove the Lipschitz condition as in Lemma~\ref{lem:m->0:|u^m-u^0|:Lipschitz}, we shall employ an argument similarly to that in \cite[Section 6.2]{cerrai2020convergence} tailored to our settings in dimensions $d=2,3$. 

We note that since the initial condition $U^m_0\in \Hcal^2$, in view of Lemma \ref{lem:moment-bound:H^2:Psi_2m:sup_[0,T]:random-initial-cond} and Lemma \ref{lem:react-diff:moment-bound:H^2:sup_[0,T]}, $\um(t)$ and $u^0(t)$ are elements in $H^2\subset L^\infty$, by Sobolev embedding. So, for $R>0$, we introduce the stopping times $\tau^m_R$ and $\tau^0_R$ given by
\begin{equation*}
\tau^m_R=\inf\{t\ge 0:\|\um(t)\|_{L^\infty}> R   \} ,
\end{equation*} 
and
\begin{align*}
\tau^0_R=\inf\{t\ge 0:\|u^0(t)\|_{L^\infty}>R\}.
\end{align*}

It is clear that for $0\le t\le \tau^m_R$, $ (\um(t),\vm(t))  =(\um_R(t),\vm_R(t))$ where $(\um_R(t),\vm_R(t))$ is the solution of the following truncating version of~\eqref{eqn:wave}
\begin{align}
\d \um(t)&=\vm(t)\d t,\nt \\
m\,\d \vm(t)&=-A\um(t)\d t-\vm(t)\d t+\f(\um(t))\theta_R(\um(t))\d t+ Q\d w(t),\label{eqn:wave:truncate}\\
(\um(0),\vm(0))&=(\um_0,\vm_0),\nt
\end{align}
where $\theta_R$ is the cut--off function introduced in~\eqref{form:theta_N}. Likewise, for $0\le t\le \tau^0_R$, $\uo(t)=\uo_R(t)$ which solves
\begin{equation} \label{eqn:react-diff:truncate}
\d \uo(t)=-A\uo(t)\d t+\f(\uo(t))\theta_R(\uo(t))\d t+ Q\d w(t),\quad\uo(0)=\um_0.
\end{equation}
Observe that the nonlinearities in~\eqref{eqn:wave:truncate} and~\eqref{eqn:react-diff:truncate} are Lipschitz functions. 

Turning to~\eqref{lim:m->0:|u^m-u^0|}, we note that
\begin{align*}
&\E\|\um(t)-\uo(t)\|_{H}\\
&=\E\Big(\|\um(t)-\uo(t)\|_H \mathbf{1}\{\tau_R^m\mi \tau^0_R>t\}\Big)+\E\Big(\|\um(t)-\uo(t)\|_{H} \mathbf{1}\{\tau_R^m\mi \tau^0_R\le t\}\Big)\\
&=I_{1,R}^{m}(t)+I^m_{2,R}(t).
\end{align*}
With regard to $I^m_{1,R}$, since $U^m_0$ satisfies the hypothesis of Lemma~\ref{lem:m->0:|u^m-u^0|:Lipschitz} (by virtue of condition~\eqref{cond:U_0}), Lemma~\ref{lem:m->0:|u^m-u^0|:Lipschitz} then implies the limit
\begin{align*}
\sup_{t\in[0,T]}I_{1,R}^{m}(t)\le \sup_{t\in[0,T]}\E \|\um_R(t)-\uo_R(t)\|_H \to 0,\quad\text{as}\,\, m\to 0.
\end{align*}
Concerning $I_{2,R}^m$, we invoke Holder's and Markov's inequalities to infer
\begin{align*}
I^m_{2,R}(t) &\le \Big(2\E\|\um(t)\|^2_H+2\E\|\uo(t)\|^2_{H}\Big)^{1/2}\Big(\P\big(\tau_R^m\le t\big)+\P\big(\tau_R^0\le t\big)\Big)^{1/2}\\
&\le \Big(2\E\|\um(t)\|^2_H+2\E\|\uo(t)\|^2_{H}\Big)^{1/2}\cdot \tfrac{1}{R}\Big(\E\sup_{r\in[0,t]}\|\um(r)\|_{L^\infty}^2+\E\sup_{r\in[0,t]}\|\uo(r)\|_{L^\infty}^2\Big)^{1/2}.
\end{align*}
In view of~\eqref{ineq:moment-bound:H:exponential:random-initial-cond:sup_[0,T]} and~\eqref{ineq:react-diff:exponential-bound}, for $\beta$ sufficiently small,
\begin{align*}
\E\|\um(t)\|^2_H+\E\|\uo(t)\|^2_{H} &\le C\, \E \exp\big\{\beta [\Psionem(u_0^m,v_0^m)+2m\|\Phi_1(u_0^m)\|_{L^1}]\big\}.
\end{align*}
On the other hand, we invoke Agmon's inequality making use of~\eqref{ineq:moment-bound:H^2:|Au|^2:sup_[0,T]:random-initial-cond} and \eqref{ineq:react-diff:moment-bound:H^2:sup_[0,T]} to obtain for some $p>0$ and $\beta$ sufficiently small
\begin{align*}
&\E\sup_{r\in[0,t]}\|\um(r)\|_{L^\infty}^2+\E\sup_{r\in[0,t]}\|\uo(r)\|_{L^\infty}^2 \\
&\le C\Big(\E \big[\|u_0^m\|^2_{H^2}+m\|v_0^m\|^2_{H^1}\big]+\E\Psitwom(u_0^m,v_0^m)^p +\E\exp\big\{\beta \big[\Psionem(u_0^m,v_0^m) +2m\|\Phi_1(u_0^m)\|_{L^1}\big]  \big\}+1\Big).
\end{align*}
It follows that
\begin{align*}
&\sup_{t\in[0,T]} I^m_{2,R}(t)\\
& \le \tfrac{C}{R}\Big(\E \big[\|u_0^m\|^2_{H^2}+m\|v_0^m\|^2_{H^1}\big]+\E\Psitwom(u_0^m,v_0^m)^p +\E\exp\big\{\beta \big[\Psionem(u_0^m,v_0^m) +2m\|\Phi_1(u_0^m)\|_{L^1}\big]  \big\}+1\Big)^{1/2},
\end{align*}
for some positive constant $C=C(T)$ independent of $m$ and $R$. By condition~\eqref{cond:U_0}, we deduce for all $m$ sufficiently small
\begin{align*}
\sup_{t\in[0,T]} I^m_{2,R}(t) \le \tfrac{C}{R}.
\end{align*}
Altogether, recalling
\begin{align*}
\E\|\um(t)-\uo(t)\|_{H}=I_{1,R}^{m}(t)+I^m_{2,R}(t),
\end{align*}
we establish limit~\eqref{lim:m->0:|u^m-u^0|} by first taking $R$ sufficiently large and then sending $m$ to 0. The proof is thus complete.

\end{proof}

Next, we prove Corollary~\ref{cor:m->0:dtnb(u^m,u^0)}, making use of Proposition~\ref{prop:m->0:|u^m-u^0|}.

\begin{proof}[Proof of Corollary~\ref{cor:m->0:dtnb(u^m,u^0)}]
By the expression~\eqref{form:d.tilde_(N,beta)} of $\dtnb$, we have
\begin{align*}
\E\, \dtnb(\um(t),\uo(t))& \le  \sqrt{N} \E \sqrt{\|\um(t)-\uo(t)\|_H\big[1+e^{\beta \|\um(t)\|^2_H}+e^{\beta \|\uo(t)\|^2_H}\big]}\\
&\le \sqrt{N}\sqrt{\E\|\um(t)-\uo(t)\|_H}\sqrt{1+\E\,e^{\beta \|\um(t)\|^2_H}+\E\, e^{\beta \|\uo(t)\|^2_H}}.
\end{align*}
For $\beta$ sufficiently small, we invoke~\eqref{ineq:moment-bound:H:exponential:random-initial-cond:sup_[0,T]} and~\eqref{ineq:react-diff:exponential-bound} to see that
\begin{align*}
&\E\,e^{\beta \|\um(t)\|^2_H}+\E\, e^{\beta \|\uo(t)\|^2_H}\\
&\le  C\, \E \exp\big\{\beta [\Psionem(u_0^m,v_0^m)+2m\|\Phi_1(u_0^m)\|_{L^1}]\big\} +C .
\end{align*}
As a consequence of condition~\eqref{cond:U_0} , the above right--hand side is a finite constant independent of $m$. From Proposition~\ref{prop:m->0:|u^m-u^0|}, we therefore obtain 
\begin{align*}
\sup_{t\in[0,T]} \E\, \dtnb(\um(t),\uo(t)) \to 0,
\end{align*}
as $m \to 0$, by virtue of~\eqref{lim:m->0:|u^m-u^0|}. The proof is thus finished.

\end{proof}

We finish this subsection by presenting the proof of Proposition~\ref{prop:regularity}.

\begin{proof}[Proof of Proposition~\ref{prop:regularity}]

We first claim that $\num$ satisfies uniform exponential bounds in $\Hcal^1$. That is for all $\beta$ sufficiently small,
\begin{align} \label{ineq:moment-bound:nu^m:exponential:H^1}
\int_{\Hcal^1}\exp\big\{\beta[\Psionem(u,v)+2m\|\Phi_1(u)\|_{L^1}]\big\}\num(\d u,\d v) < C,
\end{align}
for some positive constant $C$ independent of $m$ and $\num$. The proof of~\eqref{ineq:moment-bound:nu^m:exponential:H^1} is quite standard making use of the exponential bound~\eqref{ineq:moment-bound:H:exponential}. To see this, for $R$, consider the set
\begin{align*}
B_R = \{(u,v)\in \Hcal^1: \Psionem(u,v)+2m\|\Phi_1(u)\|_{L^1}>R  \}.
\end{align*}
Recalling that $V_m$ and $\Psionem(u,v)+2m\|\Phi_1(u)\|_{L^1}$ are equivalent (see \eqref{cond:V_m<Psi_1m<V_m}), there exists $R=R(\varepsilon,m)>0$ large enough 
 such that
\begin{align*}
\num( B_R^c )<\varepsilon.
\end{align*}
Given $N>0$, we set $\phi_N(u,v)=e^{\beta[\Psionem(u,v)+2m\|\Phi_1(u)\|_{L^1}]}\mi N$. By invariance of $\num$, since $\phi_N$ is bounded,
\begin{align*}
\int_{\Hcal^1} P_t^m\phi_N(u,v)\num(\d u,\d v)=\int_{\Hcal^1} \phi_N(u,v)\num(\d u,\d v).
\end{align*}
Also, by the choice of $B_R$, we have
\begin{align*}
\int_{\Hcal^1} P_t^m\phi_N(u,v)\num(\d u,\d v)& =\int_{B_R} P_t^m\phi_N(u,v)\num(\d u,\d v)+\int_{B^c_R} P_t^m\phi_N(u,v)\num(\d u,\d v)\\
&\le \int_{B_R} P_t^m\phi_N(u,v)\num(\d u,\d v)+N\varepsilon.
\end{align*}
To estimate the first term on the above right--hand side, we invoke~\eqref{ineq:moment-bound:H:exponential} and obtain for all $\beta>0$ sufficiently small
\begin{align*}
P_t^m\phi_N(u,v)& =\E\big(e^{\beta[\Psionem(\um(t),\vm(t))+2m\|\Phi_1(\um(t))\|_{L^1}]}\mi N\big) \\
&\le   Ce^{-ct}e^{\beta[\Psionem(u,v)+2m\|\Phi_1(u)\|_{L^1}]}+C,
\end{align*}
whence
\begin{align*}
\int_{B_R} P_t^m\phi_N(u,v)\num(\d u,\d v)\le C e^{-ct}e^{\beta R}+C.
\end{align*}
In the above, $C,c$ only depend on $\beta$. It follows that for all $N$ we have the bound
\begin{align*}
\int_{\Hcal^1} \phi_N(u,v)\num(\d u,\d v) \le C e^{-ct}e^{\beta R^2}+N\varepsilon+C.
\end{align*}
We may take $\varepsilon$ small and then take $t$ sufficiently large to arrive at the following uniform bound in $N$ and $m$:
\begin{align*}
\int_{\Hcal^1} \phi_N(u,v)\num(\d u,\d v)\le C<\infty.
\end{align*}
We therefore establish~\eqref{ineq:moment-bound:nu^m:exponential:H^1}, by virtue of the Monotone Convergence Theorem.

Next, considering $\Psitwom$ as in~\eqref{form:Psi_2m}, we claim that for all $p\ge 1$
\begin{align} \label{ineq:moment-bound:nu^m:H^2:p-moment}
\limsup_{m\to 0}\int_{\Hcal^1} \Psitwom(u,v)^p\num(\d u,\d v)<\infty.
\end{align}
To see this, recall the well--known time--averaged measure $\num_T$ defined as
\begin{align*}
\num_T(\cdot)=\frac{1}{T}\int_0^T\close P^m_t(0,\cdot)\d t.
\end{align*}
In the above, $P^m_t$ is the transition probability associated with $(\um(t;0),\vm(t;0))$, i.e., zero initial condition. By the Krylov--Bogoliubov procedure, we know that $\num_T$ converges weakly to $\num$, the unique invariant probability measure of~\eqref{eqn:wave}. In particular, for each $m$, the following limit holds
\begin{align*}
\int_{\Hcal^1}\big[\Psitwom(u,v)^p\mi N\big]\num_T(\d u,\d v)\to \int_{\Hcal^1}\big[\Psitwom(u,v)^p\mi N\big]\num(\d u,\d v),
\end{align*}
as $T\to\infty$. Note that
\begin{align*}
&\int_{\Hcal^1}\Psitwom(u,v)^p\mi N\num_T(\d u,\d v)\\
& \le \frac{1}{T}\int_0^T \E\Psitwom(\um(t),\vm(t))^p\d t\\
&\le \frac{1}{T}\int_0^T C_{2,p} e^{-c_{2,p} t}  \Big(\Psi_{2,m}(0)^p + \big[\Psionem(0)+2m\|\Phi_1(0)\|_{L^1}\big]^{q_{2,p}} \Big) +C_{2,p}\d t\le C,
\end{align*}
where in the second to last estimate, we invoked~\eqref{ineq:moment-bound:H^2:n-moment}. As a consequence, we obtain
\begin{align*}
\int_{\Hcal^1}\big[\Psitwom(u,v)^p\mi N\big]\num(\d u,\d v) \le C,
\end{align*}
implying~\eqref{ineq:moment-bound:nu^m:H^2:p-moment}, by virtue of the Monotone Convergence Theorem. 

We now turn to the bound
\begin{align} \label{ineq:moment-bound:nu^m:H^2:1-moment}
\limsup_{m\to 0}\int_{\Hcal^1}\|u\|^2_{H^2}+m\|v\|^2_{H^1}\num(\d u,\d v)<\infty.
\end{align}
For $(u_0,v_0)\in\Hcal^2$ and $N>0$, we invoke~\eqref{ineq:d.E.Psi_2m} to see that
\begin{align*}
&\E \Psitwom(\um(t),\vm(t)) +\int_0^t\big(\tfrac{1}{2}\E\|A\um(r)\|^2_H+\tfrac{1}{2}m\E\|A^{1/2}\vm(r)\|^2_H\big)\mi N\d r\\
&\le  \Psitwom(u_0,v_0)   +C\int_0^t \E\Psionem(\um(r),\vm(r))^q\d r+Ct\\
&\le \Psitwom(u_0,v_0)   +C\big[\Psionem(u_0,v_0)+2m\|\Phi_1(u_0)\|_{L^1}\big]  ^q+Ct.
\end{align*}
In the last estimate above, we employed~\eqref{ineq:moment-bound:H:n-moment}. Since $\Psitwom, \Psionem+2m\|\Phi_1(\cdot)\|_{L^1}\in L^1(\num)$ by virtue of~\eqref{ineq:moment-bound:nu^m:exponential:H^1} and~\eqref{ineq:moment-bound:nu^m:H^2:1-moment}, respectively, we invoke the invariance property of $\num$ to obtain
\begin{align*}
&t\int_{\Hcal^1}\Big[ \big(\tfrac{1}{2} \|Au_0\|^2_H+\tfrac{1}{2}m\E\|A^{1/2}v_0\|^2_H\big)\mi N\Big] \num(\d u_0,\d v_0)\\
&= \int_{\Hcal^1} \int_0^t\Big[ \big(\tfrac{1}{2}\E\|A\um(r)\|^2_H+\tfrac{1}{2}m\E\|A^{1/2}\vm(r)\|^2_H\big)\mi N\Big]\d r\,\num(\d u_0,\d v_0)\\
&\le C\int_{\Hcal^1}\big[\Psionem(u_0,v_0)+2m\|\Phi_1(u_0)\|_{L^1}\big]^q \num(\d u_0,\d v_0)+Ct,
\end{align*}
whence
\begin{align*}
\int_{\Hcal^1} \Big[ \big(\tfrac{1}{2} \|Au_0\|^2_H+\tfrac{1}{2}m\|A^{1/2}v_0\|^2_H\big)\mi N \Big]\num(\d u_0,\d v_0) \le C,
\end{align*}
for some positive constant $C$ independent of $N$ and $m$. By virtue of the Monotone Convergence Theorem, we establish~\eqref{ineq:moment-bound:nu^m:H^2:1-moment}.

Altogether, \eqref{ineq:moment-bound:nu^m:exponential:H^1},~\eqref{ineq:moment-bound:nu^m:H^2:p-moment} and \eqref{ineq:moment-bound:nu^m:H^2:1-moment} produce~\eqref{ineq:moment-bound:nu^m:H^1+H^2}, as claimed.

\end{proof}

\subsection{Proof of Theorem~\ref{thm:m->0:f}} \label{sec:small-mass:f}

The proof of Theorem~\ref{thm:m->0:f} follows along the lines of \cite[Appendix A]{glatt2021mixing} tailored to our settings. See also \cite[Theorem 3.13]{glatt2022short}.

\begin{proof}[Proof of Theorem~\ref{thm:m->0:f}] 1. With regard to~\eqref{lim:m->0:delta_(u^m)}, fixing $N$ large and $\beta$ small enough such that Theorem~\ref{thm:geometric-ergodicity:mass} holds, let $T^*=T^*(N,\beta)$ be the time constant as in \eqref{ineq:geometric-ergodicity:mass}. We note that
\begin{align} \label{ineq:triangle:0}
\sup_{t\in[0,nT^*]}\W_{\dtnbhalf}\big( \pi_1(P^m_t)^*\delta_{U_0^m},(P^0_t)^*\delta_{u^m_0}  \big)   & \le   \sup_{s\in[0,nT^*]} \E \,\dtnbhalf\big( \um(t),\uo(t)  \big).
\end{align}
In the above, we recall that $(\um(t),\vm(t))$ and $\uo(t)$ respectively are the solutions of~\eqref{eqn:wave} and~\eqref{eqn:react-diff} with initial conditions $U_0^m$ and $u^m_0$. Since $\|U^m_0\|_{\Hcal^2}<R$, $U^0_m$ satisfies the condition~\eqref{cond:U_0}. In particular, Corollary~\ref{cor:m->0:dtnb(u^m,u^0)} implies
\begin{align} \label{ineq:triangle:0a}
\sup_{s\in[0,nT^*]} \E \,\dtnbhalf\big( \um(t),\uo(t)  \big)\to 0,\quad\text{as  }m\to 0.
\end{align}

On the other hand, for all $t\ge nT^*$, by the triangle inequality~\eqref{ineq:W_dtnbhalf<W_dtnb}, we have
\begin{align}
& \W_{\dtnbhalf}\big( \pi_1(P^m_t)^*\delta_{U_0^m},(P^0_t)^*\delta_{u^m_0}  \big)  \nt \\
& \le C\big[\W_{\dtnb}\big( \pi_1(P^m_t)^*\delta_{U_0^m},\pi_1\num \big)+\W_{\dtnbtwo}\big(\pi_1 \num,\nu^0  \big)+\W_{\dtnbtwo}\big( \nu^0,(P^0_t)^*\delta_{u^m_0}  \big)\big]. \label{ineq:triangle:a}
\end{align}
To estimate the first term on the above right--hand side, we employ~\eqref{ineq:W_dmtnb>W_dtnb}
\begin{align*}
\W_{\dtnb}\big( \pi_1(P^m_t)^*\delta_{U_0^m},\pi_1\num \big) \le \W_{\dmtnb}\big( (P^m_t)^*\delta_{U_0^m},\num \big),
\end{align*}
where $\dmtnb$ is the function--like distance as in~\eqref{form:d.tilde^m_(N,beta)}.  For $t\ge nT^*$, we invoke invariance property of $\num$ and \eqref{ineq:geometric-ergodicity:mass} to estimate
\begin{align*}
\W_{\dmtnb}\big( (P^m_t)^*\delta_{U_0^m},\num \big) &= \W_{\dmtnb}\big( (P^m_t)^*\delta_{U_0^m},(P^m_t)^*\num \big)\\
&\le Ce^{-cn}\W_{\dmtnb}\big( (P^m_{t-\lfloor t/T^*\rfloor})^*\delta_{U_0^m},(P^m_{t-\lfloor t/T^*\rfloor})^*\num \big)\\
&\le Ce^{-cn}\sup_{s\in[0,T^*]}\W_{\dmtnb}\big( (P^m_s)^*\delta_{U_0^m},(P^m_s)^*\num \big).
\end{align*} 
Recalling $\dmtnb$ as in~\eqref{form:d.tilde^m_(N,beta)}, let $(X,Y)$ be an arbitrarily coupling of $\big( (P^m_s)^*\delta_{U_0^m},(P^m_s)^*\num \big)$, we invoke~\eqref{ineq:varrho^m_beta<|U-U.tilde|} and obtain
\begin{align*}
\W_{\dmtnb}\big( (P^m_s)^*\delta_{U_0^m},(P^m_s)^*\num \big) &\le \E\dmtnb (X,Y)\\
&\le  \sqrt{N\E\Big(\varrho^m_\beta(X,Y)\big[1+e^{\beta V_m(X)} + e^{\beta V_m(Y)}\big]\Big)}\\
&\le C+ C\E e^{c\beta V_m(X)} +C \E e^{c\beta V_m(Y)},
\end{align*}
for some positive constants $C,c$ independent of $\beta$, $m$, and $(X,Y)$. In view of~\eqref{ineq:moment-bound:H:exponential:random-initial-cond} and \eqref{cond:V_m<Psi_1m<V_m}, we further estimate
\begin{align*}
&\E e^{c\beta V_m(X)} + \E e^{c\beta V_m(Y)}\\
& \le \E e^{c\beta [\Psionem(X)+\|\Phi_1(\pi_1 X)\|_{L^1}]} + \E e^{c\beta [\Psionem(Y)+\|\Phi_1(\pi_1 Y)\|_{L^1}]}\\
& \le C e^{c\beta [\Psionem(U^m_0)+\|\Phi_1(\pi_1 u^m_0)\|_{L^1}]} +C \int_{\Hcal^1}  e^{c\beta [\Psionem(u,v)+\|\Phi_1(u)\|_{L^1}]}\num(\d  u,\d v)+C.
\end{align*}
In the above, we recall $\Phi_1$ and $\Psionem$ defined in~\eqref{cond:Phi_1} and \eqref{form:Psi_1m}, respectively. As a consequence, 
\begin{align*}
&\sup_{s\in[0,T^*]}\W_{\dmtnb}\big( (P^m_s)^*\delta_{U_0^m},(P^m_s)^*\num \big) \\
&\le C  e^{c\beta [\Psionem(U^m_0)+\|\Phi_1(\pi_1 u^m_0)\|_{L^1}]} +C \int_{\Hcal^1}  e^{c\beta [\Psionem(u,v)+\|\Phi_1(u)\|_{L^1}]}\num(\d  u,\d v)+C.
\end{align*}
Since $\|U_0^m\|_{\Hcal^2}<R$ and $\num$ satisfies the uniform exponential bound~\eqref{ineq:moment-bound:nu^m:exponential:H^1}, we infer
\begin{align*}
\sup_{s\in[0,T^*]}\W_{\dmtnb}\big( (P^m_s)^*\delta_{U_0^m},(P^m_s)^*\num \big)  \le C,
\end{align*}
for some positive constant $C=C(T^*)$ independent of $m$. Hence, for all $t\ge nT^*$,
\begin{align}
\W_{\dtnb}\big( \pi_1(P^m_t)^*\delta_{U_0^m},\pi_1\num \big)& \le \W_{\dmtnb}\big( (P^m_t)^*\delta_{U_0^m},\num \big)   \nt \\
&\le Ce^{-cn}\sup_{s\in[0,T^*]}\W_{\dmtnb}\big( (P^m_s)^*\delta_{U_0^m},(P^m_s)^*\num \big)  \nt \\
&\le Ce^{-cn}. \label{ineq:triangle:b}
\end{align}
Concerning the last term on the right--hand side of~\eqref{ineq:triangle:a}, we employ~\eqref{ineq:react-diff:geometric-ergodicity} to see that
\begin{align*}
\W_{\dtnbtwo}\big( \nu^0,(P^0_t)^*\delta_{u^m_0}  \big)&=\W_{\dtnbtwo}\big( (P^0_t)^*\nu^0,(P^0_t)^*\delta_{u^m_0}  \big)\\
& \le Ce^{-ct}\W_{\dtnbtwo}\big( \nu^0,\delta_{u^m_0}  \big),
\end{align*}
which holds for all $t\ge \tilde{T}$ where $\tilde{T}$ is as in~\eqref{ineq:react-diff:geometric-ergodicity}. Recalling $\dtnb$ from \eqref{form:d.tilde_(N,beta)}, we see that
\begin{align*}
\W_{\dtnbtwo}\big( \nu^0,\delta_{u^m_0}  \big) \le C\int_{H}\|u\|^2+ e^{2\beta \|u\|^2_H}\nu^0(\d u)+C\big(\|u^m_0\|^2_H+e^{2\beta\|u^m_0\|^2_H}\big)+C.
\end{align*}
Since $\|U_0^m\|_{\Hcal^2}<R$ and $\nu^0$ satisfies the exponential bound~\eqref{ineq:react-diff:exponential-bound:nu^0}, we infer
\begin{align*}
\W_{\dtnbtwo}\big( \nu^0,\delta_{u^m_0}  \big)  \le C,
\end{align*}
for some positive constant $C=C(R,\beta)$ independent of $m$ and $t$. It follows that
\begin{align} \label{ineq:triangle:c}
\W_{\dtnbtwo}\big( \nu^0,(P^0_t)^*\delta_{u^m_0}  \big) \le Ce^{-ct}, \quad t\ge \tilde{T}.
\end{align}
We collect~\eqref{ineq:triangle:a}, ~\eqref{ineq:triangle:b} and ~\eqref{ineq:triangle:c} to deduce for all $t\ge nT^*\ge \tilde{T}$
\begin{align*}
\sup_{t\ge nT^*}\W_{\dtnbhalf}\big( \pi_1(P^m_t)^*\delta_{U_0^m},(P^0_t)^*\delta_{u^m_0}  \big)  \le Ce^{-cn}+C\W_{\dtnbtwo}\big(\pi_1 \num,\nu^0  \big),
\end{align*}
for some positive constant $C=C(T^*,\beta,R)$ independent of $n$ and $m$. This together with~\eqref{ineq:triangle:0} implies
\begin{align*}
\sup_{t\ge 0}\W_{\dtnbhalf}\big( \pi_1(P^m_t)^*\delta_{U_0^m},(P^0_t)^*\delta_{u^m_0}  \big)  \le  \sup_{s\in[0,nT^*]} \E \,\dtnbhalf\big( \um(t),\uo(t)  \big)+Ce^{-cn}+C\W_{\dtnbtwo}\big( \pi_1\num,\nu^0  \big),
\end{align*}
In view of~\eqref{ineq:triangle:0a} and Theorem~\ref{thm:nu^m->nu^0}, we establish
\begin{align*}
\sup_{t\ge 0}\W_{\dtnbhalf}\big( \pi_1(P^m_t)^*\delta_{U_0^m},(P^0_t)^*\delta_{u^m_0}  \big)  \to 0,\quad\text{as }\,\,m\to 0,
\end{align*}
by first taking $n$ large and then sending $m$ to 0. In turn, this produces the desired limit~\eqref{lim:m->0:delta_(u^m)} for all $\beta$ small and $N$ large enough, as claimed. 

2. Let $f\in C_b(H;\rbb)$ be an observable satisfying $L_f=[f]_{\text{Lip},\dtnb}<\infty$ (see~\eqref{form:Lipschitz}), we invoke~\eqref{ineq:W_d(nu_1,nu_2):dual} to infer
\begin{align*}
\big|\E\, f(\um(t))-\E\,f(\uo(t))\big| \le L_f\W_{\dtnb}\big( \pi_1(P^m_t)^*\delta_{U_0^m},(P^0_t)^*\delta_{u^m_0}  \big) .
\end{align*}
By virtue of~\eqref{lim:m->0:delta_(u^m)}, we establish~\eqref{lim:m->0:f(u^m)}, thereby finishing the proof.

\end{proof}

\section*{Acknowledgment}
The author thanks Nathan Glatt-Holtz and Vincent Martinez for fruitful discussions on the topic of this paper. The author also would like to thank the anonymous reviewer for their valuable comments and suggestions.

\appendix

\section{Stochastic convolution} \label{sec:stochastic-convolution}

We consider the stochastic convolution $\Gamma^m(t)=(\pi_1\Gamma^m(t),\pi_2\Gamma^m(t))$ solving the linear system
\begin{equation} \label{eqn:wave:linear}
\begin{aligned}
\d u(t)&=v(t)\d t,\\
m\,\d v(t)&=-Au(t)\d t-v(t)\d t+ Q\d w(t),\\
(u(0),v(0))&=0\in \Hcal^1.
\end{aligned}
\end{equation}
In Lemma~\ref{lem:irreducibility:Gamma^m(t)} stated below, we assert two properties of $\Gammau(t)$, namely, the small ball probabilities and moment bounds in $H^2$. The former was employed in Section~\ref{sec:geometric-ergodicity:irreducibility} whereas the latter appeared in the proof of Lemma~\ref{lem:moment-bound:H^2:|Au|^2:sup_[0,T]:random-initial-cond}.
\begin{lemma} \label{lem:irreducibility:Gamma^m(t)}
Let $\Gamma^m(t)$ be the solution of~\eqref{eqn:wave:linear}. Then, the followings hold:

1. For all $r>0$ and $T>1$,
\begin{align}\label{ineq:irreducibility:Gamma^m(t)}
\P\Big(\sup_{t\in[0, T]} \|\pi_1\Gamma^m(t)\|_{H^2}^2+m^2\|\pi_2\Gamma^m(t)\|_{H^1}^2<r \Big) \ge c>0,
\end{align}
for some positive constant $c=c(T,r)$ independent of $m$.

2. For all $p\ge 1$ and $T>0$
\begin{align}\label{ineq:moment-bound:Gamma^m(t):H^2:sup_[0,T]}
\E\sup_{t\in[0,T]}\|\Gammau(t)\|^{p}_{H^2}\le C,
\end{align}
for some positive constant $C=C(T,p)$ independent of $m$.
\end{lemma}

In order to prove Lemma~\ref{lem:irreducibility:Gamma^m(t)}, we will compare $\Gammau(t)$ with the process $\pi_1\eta^m(t)$ where $\eta^m(t)$ solves the Langevin equation
\begin{equation} \label{eqn:Langevin:linear}
\begin{aligned}
\d u(t)&=v(t)\d t,\\
m\,\d v(t)&=-v(t)\d t+ Q\d w(t),\\
(u(0),v(0))&=0\in \Hcal^1.
\end{aligned}
\end{equation}
In what follows, we assert similar properties of $\pi_1\eta^m(t)$ that will be employed in the proof of Lemma~\ref{lem:irreducibility:Gamma^m(t)}.
\begin{lemma} \label{lem:irreducibility:eta^m(t)}
Let $\eta^m(t)=(\pi_1\eta^m(t),\pi_2\eta^m(t))$ be the solution of~\eqref{eqn:Langevin:linear}. Then, the followings hold:

1. For all $r\in(0,1)$ and $T>1$,
\begin{align} \label{ineq:irreducibility:eta^m(t)}
\P\Big(\sup_{0\le t\le T} \|\pi_1\eta^m(t)\|^2_{H^3}+m^2\|\pi_2\eta^m(t)\|_{H^3}^2< 8 \emph{Tr}(QA^3Q^*)r \Big) \ge e^{-cT/r}>0,
\end{align}
for some positive constant $c$ independent of $m$, $T$ and $r$.

2. For all $p\ge 1$ and $T>0$,
\begin{align}\label{ineq:moment-bound:eta^m(t):H^2:sup_[0,T]}
\E\sup_{t\in[0,T]}\|\pi_1\eta^m(t)\|^{p}_{H^3}\le C,
\end{align}
for some positive constant $C=C(T,p)$ independent of $m$.
\end{lemma}

\begin{proof} 1. By the variation constants formula, $\pi_2\eta^m(t)$ is explicitly given by
\begin{equation*}
\pi_2\eta^m(t)=\tfrac{1}{m}\int_0^te^{-\frac{1}{m}(t-r)}Q\d w(r),\quad t\ge 0.
\end{equation*}
Given $0\le s\le t$, we use It\^o's formula to compute
\begin{align*}
\E\| \pi_2\eta^m(t)-\pi_2\eta^m(s)\|^2_{H^3} &= \E\Big\|\tfrac{1}{m}\int_s^t e^{-\frac{1}{m}(t-r)}Q\d w(r)+\tfrac{1}{m}\big(e^{-\frac{1}{m}t}-e^{-\frac{1}{m}s}\big)\int_0^s e^{\frac{1}{m}r}Q\d w(r) \Big\|^2_{H^3} \\
&=\frac{1}{m^2}\Big(\int_s^t e^{-\frac{2}{m}(t-r)}\d r+\big(e^{-\frac{1}{m}t}-e^{-\frac{1}{m}s}\big)^2\int_0^s e^{\frac{2}{m}r}\d r\Big) \Tr(QA^3Q^*)\\
&= \frac{\Tr(QA^3Q^*)}{2m}\big(1 - e^{-\frac{2}{m}(t-s)}\big)\big[ \big(1-e^{-\frac{2}{m}s}\big)\big(1-e^{-\frac{2}{m}(t-s)}\big)+1  \big].
\end{align*}
Recall from~\eqref{cond:Q:Tr(QA^3Q)} that
$
\Tr(QA^3Q^*)<\infty.
$
So, we use the elementary inequality 
\begin{align*}
1-e^{-x}\le x,\quad x\ge 0,
\end{align*}
to produce
\begin{align*}
\E\|\pi_2 \eta^m(t)-\pi_2\eta^m(s)\|^2_{H^3}  \le 2\frac{\Tr(QA^3Q^*)}{m^2}(t-s).
\end{align*}
Also, from~\eqref{eqn:Langevin:linear}, we see that
\begin{align*}
\E\|\pi_1 \eta^m(t)-\pi_1\eta^m(s)\|^2_{H^3} &\le 2m^2\E\|\pi_2 \eta^m(t)-\pi_2\eta^m(s)\|^2_{H^3} +2\E\Big\|\int_s^t Q\d w(r)\Big\|_{H^3}^2\\
&\le 6 \Tr(QA^3Q^*)(t-s).
\end{align*}

Turning to~\eqref{ineq:irreducibility:eta^m(t)}, we shall employ the result from \cite[Lemma 6.2]{csorgo1994almost} applied to the Gaussian process $(\pi_1\eta^m(t),m\pi_2\eta^m(t))$ and establish~\eqref{ineq:irreducibility:eta^m(t)}. Indeed, by virtue of \cite[Lemma 6.2]{csorgo1994almost}, for $r\in(0,1)$ and $T>1$,
\begin{align*}
\P\Big( \sup_{t\in[0,T]} \|\pi_1\eta^m(t)\|_{H^3}^2+m^2\|\pi_2\eta^m(t)\|_{H^3}^2< 8 \Tr(QA^3Q^*)r\Big) \ge e^{-cT/r},
\end{align*}
holds for some positive constant $c$ independent of $T$, $r$ and $m$. This clearly establishes~\eqref{ineq:irreducibility:eta^m(t)}, as claimed.

2. With regard to~\eqref{ineq:moment-bound:eta^m(t):H^2:sup_[0,T]}, we shall employ \cite[Lemma 6.3]{csorgo1994almost} to obtain the moment bound in sup norm. To see this, for $0\le t\le T$, setting 
\begin{align*}
\Lambda(t)= \sqrt{6 \Tr(QA^3Q^*)t},\quad \text{ and }\quad \Gamma = \sqrt{6 \Tr(QA^3Q^*)T}.
\end{align*}
It is clear that
\begin{align*}
\int_1^\infty \Lambda(e^{-y^2}T)\d y<\infty.
\end{align*}
In particular, all the hypothesis of~\cite[Lemma 6.3]{csorgo1994almost} are met. In turn, this implies the bound for all $R>0$
\begin{align*}
\P\Big(\sup_{0\le t\le T} \|\pi_1\eta^m(t)\|_{H^3}> R\Big(\Gamma+ \int_1^\infty \Lambda(e^{-y^2}T)\d y  \Big)  \Big) \le C e^{-R^2/2},
\end{align*}
holds for some positive constant $C$ independent of $m$, $T$ and $R$. As a consequence, for $p\ge 1$
\begin{align*}
\E \big[\sup_{0\le t\le T} \|\pi_1\eta^m(t)\|_{H^3}^p\big]= \int_0^\infty \close R^{p-1}\P\Big(\sup_{0\le t\le T} \|\pi_1\eta^m(t)\|_{H^3}> R  \Big)\d R\le C,
\end{align*}
where $C=C(T,p)$ does not depend on $m$. The proof is thus finished.

\end{proof}

We now provide the proof of Lemma~\ref{lem:irreducibility:Gamma^m(t)}.

\begin{proof}[Proof of Lemma~\ref{lem:irreducibility:Gamma^m(t)}]
Letting $\eta^m(t)$ be the solution of~\eqref{eqn:Langevin:linear}, denote
\begin{align*}
\ubar(t)=\pi_1\Gamma^m(t)-\pi_1\eta^m(t),\quad \vbar(t)=\pi_2\Gamma^m(t)-\pi_2\eta^m(t).
\end{align*}
From~\eqref{eqn:wave:linear} and~\eqref{eqn:Langevin:linear}, observe that $(\ubar(t)_,\vbar(t)$ solve the system
\begin{align*}
\tfrac{\d}{\d t} \ubar(t)&=\vbar(t),\\
m\ddt \vbar(t)&=-A\ubar(t)-\vbar(t)-A\pi_1\eta^m(t),\\
(u(0),v(0))&=0.
\end{align*}
A routine calculation together with Cauchy--Schwarz inequality gives
\begin{align*}
\ddt (\| \ubar(t)\|^2_{H^2}+m\|\vbar(t)\|^2_{H^1})&= -2\|\vbar(t)\|^2_{H^1}-2\la A^{3/2}\pi_1\eta^m(t),A^{1/2}\vbar(t)\ra_{H}\\
&\le \|\pi_1\eta^m(t)\|^2_{H^3}.
\end{align*}
It follows that (recalling $m<1$)
\begin{align*}
\sup_{t\in[0,T]}\|\ubar(t)\|^2_{H^2}+m^2\|\vbar(t)\|^2_{H^1}\le T\sup_{t\in[0,T]}\|\pi_1\eta^m(t)\|^2_{H^3},
\end{align*}
whence by Sobolev embedding
\begin{align}
&\sup_{t\in[0,T]}\|\Gammau(t)\|^2_{H^2}+m^2\|\Gammav(t)\|^2_{H^1}  \nt \\
&\le 2\sup_{t\in[0,T]}\|\pi_1\eta^m(t)\|^2_{H^2}+m^2\|\pi_2\eta^m(t)\|^2_{H^1}+\|\ubar(t)\|^2_{H^2} +m^2\|\vbar(t)\|^2_{H^1}\nt \\
& \le   2\big(\tfrac{1}{\alpha_1}+T\big)\sup_{t\in[0,T]}\|\pi_1\eta^m(t)\|^2_{H^3}+2\tfrac{1}{\alpha_1^2}\sup_{t\in[0,T]}m^2\|\pi_2\eta^m(t)\|^2_{H^3}  \nt \\
&\le 2\big(\tfrac{1}{\alpha_1}+\tfrac{1}{\alpha_1^2}+T\big)\sup_{t\in[0,T]}\|\pi_1\eta^m(t)\|^2_{H^3}+m^2\|\pi_2\eta^m(t)\|^2_{H^3} . \label{ineq:Gamma^m<eta^m}
\end{align}
In light of Lemma~\ref{lem:irreducibility:eta^m(t)}, cf~\eqref{ineq:irreducibility:eta^m(t)}, for $r\in(0,1)$ and $T>1$, we have
\begin{align*}
&\P\Big(\sup_{0\le t\le T} \|\pi_1\Gamma^m(t)\|_{H^2}^2+m^2\|\pi_2\Gamma^m(t)\|_{H^1}^2 < 16 \Tr(QA^3Q^*)\big(\tfrac{1}{\alpha_1}+\tfrac{1}{\alpha_1^2}+T\big)r \Big) \\
&\ge \P\Big(\sup_{0\le t\le T} \|\pi_1\eta^m(t)\|_{H^3}^2+m^2\|\pi_2\eta^m(t)\|^2_{H^3}  < 8\Tr(QA^3Q^*)r\Big) \ge e^{-cT/r}>0.
\end{align*}
In turn, this produces~\eqref{ineq:irreducibility:Gamma^m(t)} for all $r>0$.

2. With regard to~\eqref{ineq:moment-bound:Gamma^m(t):H^2:sup_[0,T]}, we invoke~\eqref{ineq:Gamma^m<eta^m} again to see that for all $p\ge 1$
\begin{align*}
\E\sup_{t\in[0,T]}\|\Gammau(t)\|^p_{H^2} \le C \E\sup_{t\in[0,T]}\|\Gammau(t)\|^p_{H^3} <C(T,p),
\end{align*}
where the last estimate follows from~\eqref{ineq:moment-bound:eta^m(t):H^2:sup_[0,T]}. This finishes the proof.
\end{proof}

\section{Estimates on~\eqref{eqn:react-diff}} \label{sec:react-diff}

In this section, we collect several estimates on the reaction--diffusion equation~\eqref{eqn:react-diff}. The first three results describe the long time statistics of~\eqref{eqn:react-diff} and can be found in \cite{glatt2022short}.

\begin{theorem}{\cite[Theorem 8.1]{glatt2022short}} \label{thm:react-diff:geometric-ergodicity}
Suppose that Assumption \ref{cond:phi:well-posed} and Assumption \ref{cond:Q} hold. Then, for all $N$ sufficiently large and $\beta$ sufficiently small, there exists a positive constant $\tilde{T}=\tilde{T}(N,\beta)$ large enough such that the following holds for all $\nu_1,\nu_2\in \Pcal r(H)$,
\begin{equation} \label{ineq:react-diff:geometric-ergodicity:beta/2}
\W_{\dtnbtwo} \big((P^0_t)^*\nu_1,(P^0_t)^*\nu_2\big)\le Ce^{-c t}\W_{\dtnb}(\nu_1,\nu_2),\quad t\ge \tilde{T},
\end{equation}
for some positive constants $c=c(N,\beta), \,C=C(N,\beta)$ independent of $\nu_1,\nu_2$ and $t$. Here, $\dtnb$ and $\dtnbhalf$ are defined in~\eqref{form:d.tilde_(N,beta)}.

As a consequence,
\begin{equation} \label{ineq:react-diff:geometric-ergodicity}
\W_{\dtnb} \big((P^0_t)^*\nu_1,(P^0_t)^*\nu_2\big)\le Ce^{-c t}\W_{\dtnb}(\nu_1,\nu_2),\quad t\ge \tilde{T}.
\end{equation}
\end{theorem}

\begin{lemma}{\cite[Lemma 8.3]{glatt2022short}} \label{lem:react-diff:exponential-bound}
Let $u_0$ be a random variable in $L^2(\Omega; H)$. Then, for all $\beta$ sufficiently small, there exist positive constants $c$ and $C$ independent of $u_0$ and $t$ such that
\begin{equation} \label{ineq:react-diff:exponential-bound}
\E\, e^{\beta\|u^0 (t)\|^2_H}\le e^{-ct} \E\, e^{\beta\|u_0\|^2_H}+C.
\end{equation}
As a consequence, it holds that
\begin{align} \label{ineq:react-diff:exponential-bound:nu^0}
\int_H e^{\beta\|u\|^2_H}\nu^0(\emph{d} u)<\infty,
\end{align}
where $\nu^0$ is the unique invariant probability measure of~\eqref{eqn:react-diff}.
\end{lemma}

\begin{lemma}{\cite[Lemma 8.5]{glatt2022short}} \label{lem:react-diff:H^1+H^2:nu^0}
Let $\nu^0$ be the unique invariant probability measure of~\eqref{eqn:react-diff}. Then, 
\begin{align} \label{ineq:react-diff:H^1+H^2:nu^0}
\int_H\|u\|^p_{H^1}+\|u\|^2_{H^2}\nu_0(\emph{d} u)<\infty.
\end{align}
\end{lemma}

We conclude this section by Lemma~\ref{lem:react-diff:moment-bound:H^2:sup_[0,T]} giving an estimate in $H^2$ on any finite time interval. We have employed this result in Section~\ref{sec:small-mass:nu^m->nu^0}.

\begin{lemma} \label{lem:react-diff:moment-bound:H^2:sup_[0,T]}
Let $u_0$ be a random variable in $L^2(\Omega; H^2)$. Then, there exists a positive constant $p$ such that for all $T>0$, the following holds 
\begin{equation} \label{ineq:react-diff:moment-bound:H^2:sup_[0,T]}
\E\sup_{t\in[0,T]}\|u^0(t)\|^2_{H^2} \le  C\big(\E \|u_0\|^{p}_{H^1}+\E \|u_0\|^{2}_{H^2}+1\big),
\end{equation}
for some positive constant $=C(T,p)$ independent of $u_0$.
\end{lemma}

The proof of Lemma~\ref{lem:react-diff:moment-bound:H^2:sup_[0,T]} is similarly to that of Lemma~\ref{lem:moment-bound:H^2:|Au|^2:sup_[0,T]:random-initial-cond}. For the sake of completeness, we sketch the main steps without going into detail.

\begin{proof}[Sketch of the proof of Lemma~\ref{lem:react-diff:moment-bound:H^2:sup_[0,T]}]
\emph{Step 1}: Let $\Gamma^0(t)$ be the solution of the linear equation
\begin{align*}
\d u(t)&=-Au(t)\d t+Q\d w(t),\quad u(0)=0\in H.
\end{align*}
which is relatively well--known \cite[Chapter 5]{da2014stochastic}. By It\^o's formula, we see that
\begin{align*}
\d \|\Gamma^0(t)\|_{H^2}^2=-2 \|\Gamma^0(t)\|^2_{H^3}\d t+2 \la \Gamma^0(t),Q\d w(t)\ra_{H^2}+\Tr(QA^2Q^*)\d t.
\end{align*}
Employing a strategy similarly to the proof of Lemma~\ref{lem:moment-bound:H}, part 5, while making use of the exponential Martingale inequality, it is not difficult to verify that
\begin{align} \label{ineq:Gamma^0}
\E\exp\big\{\sup_{t\in[0,T]}\|\Gamma^0(t)\|^2_{H^2} \big\}\le C(T).
\end{align}

\emph{Step 2}: Denoting $\ubar(t)=u^0(t)-\Gamma^0(t)$,
observe that
\begin{align*}
\ddt \ubar(t)=-A\ubar(t)+\f(u^0(t)),\quad \ubar(0)=u_0.
\end{align*}
A computation gives
\begin{align*}
\ddt \|\ubar(t)\|^2_{H^1} = -\|A\ubar(t)\|^2_H+\la \f'(u^0(t))\grad u^0(t),\grad\ubar(t)\ra_H.
\end{align*}
From~\eqref{ineq:<phi'.grad(u^m),grad(u^m-Gamma^m)>} applying to the cross term on the above right--hand side, we see that
\begin{align*}
\la \f'(u^0(t))\grad u^0(t),\grad\ubar(t)\ra_H \le c \|\ubar(t)\|^2_{H^1}+c\|\Gamma^0(t)\|^{2}_{H^1}+c\|\Gamma^0(t)\|_{H^2}^{2\lambda}+c\|\Gamma^0(t)\|_{H^2}^{\frac{2}{2-\lambda}}.
\end{align*}
It follows that for any $p\ge 2$, we have the estimate
\begin{align*}
\ddt \|\ubar(t)\|^{p}_{H^1} \le c \|\ubar(t)\|^{p}_{H^1} +c\|\Gamma^0(t)\|_{H^2}^{q},
\end{align*}
for some positive constant $q=q(p,\lambda)$. Together with~\eqref{ineq:Gamma^0}, we deduce
\begin{align} \label{ineq:react-diff:moment-bound:H^1:sup_[0,T]}
\E\sup_{t\in[0,T]}\|u^0(t)\|^{p}_{H^1} \le C(\E \|u_0\|^{p}_{H^1}+1).
\end{align}

\emph{Step 3}: Similarly to Step 2, we use Agmon's inequality to see that
\begin{align*}
\ddt \|\ubar(t)\|^2_{H^2}& = -\|A^{3/2}\ubar(t)\|^2_H+\la \f'(u^0(t))\grad u^0(t),\grad \triangle\ubar(t)\ra_H\\
&\le c \|\f'(u^0(t))\grad u^0(t)\|^2_H\\
&\le c\|u^0(t)\|^{2(\lambda-1)}_{L^\infty}\|u^0(t)\|^2_{H^1}+c\|u^0(t)\|^2_{H^1}\\
&\le c\|u^0(t)\|^2_{H^2}+c\|u^0(t)\|^{\frac{2}{2-\lambda}}_{H^1}+c\|u^0(t)\|^2_{H^1}\\
&\le c\|\ubar(t)\|^2_{H^2}+c\|\Gamma^0(t)\|^2_{H^2}+c\|u^0(t)\|^{\frac{2}{2-\lambda}}_{H^1}+c\|u^0(t)\|^2_{H^1}.
\end{align*}
In view of~\eqref{ineq:Gamma^0} and~\eqref{ineq:react-diff:moment-bound:H^1:sup_[0,T]}, we obtain
\begin{align*}
\E\sup_{t\in[0,T]}\|u^0(t)\|^{2}_{H^2} \le C(\E \|u_0\|^{p}_{H^1}+\E \|u_0\|^{2}_{H^2}+1),
\end{align*}
for some positive constant $C=C(T)$ independent of $u_0$. This produces~\eqref{ineq:react-diff:moment-bound:H^2:sup_[0,T]}, as claimed.

\end{proof}

\section{Auxiliary results} \label{sec:auxiliary-result}

We start with the following auxiliary result in Lemma~\ref{lem:phi} about the nonlinearity $\f$. We have employed Lemma~\ref{lem:phi} to establish irreducibility of~\eqref{eqn:wave} in Section~\ref{sec:geometric-ergodicity:irreducibility}.

\begin{lemma} \label{lem:phi}
Let $\f$ and $\lambda$ be as in Assumption~\ref{cond:phi:well-posed}. Then, the followings hold for all $x,y\in\rbb$:

1. \begin{align} \label{cond:phi:phi<x+x^lambda}
|\f(x)|\le \max\Big\{\sup_{|z|\le 1}|\f'(z)|,2a_1\Big\} (|x|+|x|^\lambda),
\end{align}
where $a_1$ is as in~\eqref{cond:phi:phi(x)=O(x^lambda)}.

2. \begin{align} \label{cond:phi:|phi(x)-phi(y)|}
|\f(x)-\f(y)| \le 2^\lambda a_4|x-y|\big(|x-y|^{\lambda-1}+|y|^{\lambda-1}+1  \big),
\end{align}
where $a_4$ is as in~\eqref{cond:phi:phi'=O(x^(lambda-1))}.

3. For all $\varepsilon$ sufficiently small, 
\begin{align} \label{cond:phi:xphi(x)<x^2-x^(lambda+1)}
x\f(x)\le (a_\f+\varepsilon)|x|^2-\frac{\varepsilon}{\big(\frac{2a_3}{a_2}\big)^{\frac{\lambda-1}{\lambda+1}}}|x|^{\lambda+1},
\end{align}
where $a_2,a_3$ are as in~\eqref{cond:phi:x.phi(x)<-x^(lambda+1)} and $a_\f=\sup_{x}\f'(x)$ is as in ~\eqref{cond:phi:sup.phi'<a_f}.

4. Recalling $\Phi_2(x)=-\int_0^x\f(r)\emph{d} r$ as in~\eqref{form:Phi_2}, for all $\varepsilon$ sufficiently small,
\begin{align} \label{cond:Phi_2:a}
  -\tfrac{1}{2}(a_\f+\varepsilon)x^2+\frac{\varepsilon}{(\lambda+1)\big(\frac{2a_3}{a_2}\big)^{\frac{\lambda-1}{\lambda+1}}}|x|^{\lambda+1}  \le \Phi_2(x),
\end{align}
and
\begin{align}\label{cond:Phi_2:b}
\Phi_2(x)  \le\max\Big\{\sup_{|z|\le 1}|\f'(z)|,2a_1\Big\}(x^2+|x|^{\lambda+1})+\tfrac{1}{2}a_\f x^2.
\end{align}

\end{lemma}
\begin{proof} 1. Since $\f(0)=0$, we employ the Mean Value Theorem to see that for $|x|\le 1$
\begin{align*}
|\f(x)|=|\f(x)-\f(0)|\le |x|\sup_{|z|\le 1}|\f'(z)|.
\end{align*}
On the other hand, for all $|x|\ge 1$, by \eqref{cond:phi:phi(x)=O(x^lambda)}, we infer 
\begin{align*}
|\f(x)|\le a_1(|x|^\lambda+1)\le 2a_1|x|^p.
\end{align*}
Altogether, we arrive at 
\begin{align*}
|\f(x)|\le \max\Big\{\sup_{|z|\le 1}|\f'(z)|,2a_1\Big\}  (|x|+|x|^{\lambda}).
\end{align*}
This proves~\eqref{cond:phi:phi<x+x^lambda}, as claimed.

2. Concerning~\eqref{cond:phi:|phi(x)-phi(y)|}, we invoke condition~\eqref{cond:phi:phi'=O(x^(lambda-1))} to estimate the difference $\f(x)-\f(y)$ as follows:
\begin{align*}
\f(x)-\f(y)&=\int_0^1 \f'\big(tx+(1-t)y\big)\d t(x-y)\\
&\le a_4|x-y|\sup_{t\in[0,1]}(|t(x-y)+y|^{\lambda-1}+1)\\
& \le  a_4|x-y|(2^{\lambda}|x-y|^{\lambda-1}+2^\lambda|y|^{\lambda-1}+1)  \\
&\le 2^\lambda a_4|x-y|\big( |x-y|^{\lambda-1}+|y|^{\lambda-1}+1  \big),
\end{align*}
which produces~\eqref{cond:phi:|phi(x)-phi(y)|}.

3. Turning to~\eqref{cond:phi:xphi(x)<x^2-x^(lambda+1)}, from the condition~\eqref{cond:phi:x.phi(x)<-x^(lambda+1)}, we see that for all
$|x|\ge \big(\tfrac{2a_3}{a_2}\big)^{\frac{1}{\lambda+1}},$
it holds that
\begin{align*}
x\f(x)\le -a_2|x|^{\lambda+1}+a_3 \le -\tfrac{1}{2}a_2|x|^{\lambda+1}.
\end{align*}
Also, for all $x\in\rbb$, since $\f(0)=0$, we invoke~\eqref{cond:phi:sup.phi'<a_f} to obtain
\begin{align*}
x\f(x)=x\int_0^x\close \f'(r)\d r\le a_\f x^2.
\end{align*}
In particular, for $|x|<\big(\tfrac{2a_3}{a_2}\big)^{\frac{1}{\lambda+1}}$,
\begin{align*}
x\f(x)\le (a_\f+\varepsilon)x^2-\frac{\varepsilon}{\big(\frac{2a_3}{a_2}\big)^{\frac{\lambda-1}{\lambda+1}}}|x|^{\lambda+1}.
\end{align*}
By choosing $\varepsilon$ sufficiently small, e.g.,
\begin{align*}
\frac{\varepsilon}{\big(\frac{2a_3}{a_2}\big)^{\frac{\lambda-1}{\lambda+1}}} <\tfrac{1}{2}a_2,
\end{align*}
we also deduce 
\begin{align*}
x\f(x)\le -\tfrac{1}{2}a_2|x|^{\lambda+1}\le (a_\f+\varepsilon)x^2-\frac{\varepsilon}{\big(\frac{2a_3}{a_2}\big)^{\frac{\lambda-1}{\lambda+1}}}|x|^{\lambda+1},
\end{align*}
for all $|x|\ge\big(\tfrac{2a_3}{a_2}\big)^{\frac{1}{\lambda+1}}$. This establishes~\eqref{cond:phi:xphi(x)<x^2-x^(lambda+1)}.

4. Concerning~\eqref{cond:Phi_2:a}, there are two cases depending on the sign of $x$. If $x$ is positive, we invoke~\eqref{cond:phi:xphi(x)<x^2-x^(lambda+1)} to see that
\begin{align*}
-x\f(x)\ge -(a_\f+\varepsilon)x^2+\frac{\varepsilon}{\big(\frac{2a_3}{a_2}\big)^{\frac{\lambda-1}{\lambda+1}}}|x|^{\lambda+1},
\end{align*}
whence
\begin{align*}
-\f(x) \ge -(a_\f+\varepsilon)x+\frac{\varepsilon}{\big(\frac{2a_3}{a_2}\big)^{\frac{\lambda-1}{\lambda+1}}}|x|^{\lambda}.
\end{align*}
It follows that
\begin{align*}
\Phi(x)=-\int_0^x \f(r)\d r\ge -\tfrac{1}{2}(a_\f+\varepsilon)x^2+\frac{\varepsilon}{(\lambda+1)\big(\frac{2a_3}{a_2}\big)^{\frac{\lambda-1}{\lambda+1}}}|x|^{\lambda+1}.
\end{align*}

Otherwise, if $x$ is non-positive, from~\eqref{cond:phi:phi<x+x^lambda}, we see that
\begin{align*}
\frac{-\f(x)}{\max\{\sup_{|z|\le 1}|\f'(z)|,2a_1\}}&\ge -  (|x|+|x|^{\lambda})= x-(-x)^{\lambda},
\end{align*}
implying
\begin{align*}
\Phi(x)=-\int_0^x \f(r)\d r\ge \max\Big\{\sup_{|z|\le 1}|\f'(z)|,2a_1\Big\}\big(\tfrac{1}{2}x^2 + |x|^{\lambda+1} \big).
\end{align*}
Altogether, by taking $\varepsilon$ sufficiently small, we arrive at the bound for all $x$
\begin{align*}
\Phi(x) \ge  -\tfrac{1}{2}(a_\f+\varepsilon)x^2+\frac{\varepsilon}{(\lambda+1)\big(\frac{2a_3}{a_2}\big)^{\frac{\lambda-1}{\lambda+1}}}|x|^{\lambda+1}.
\end{align*}

Turning to~\eqref{cond:Phi_2:b}, an integration by parts together with~\eqref{cond:phi:phi<x+x^lambda} and~\eqref{cond:phi:sup.phi'<a_f} yields
\begin{align*}
\Phi(x) =-\int_0^x \f(r)\d r&= -x\f(x)+\int_0^x r\f'(r)\d r\\
&\le -x\f(x)+a_\f\int_0^x r\d r\\
&\le  \max\Big\{\sup_{|z|\le 1}|\f'(z)|,2a_1\Big\}(x^2+|x|^{\lambda+1})+\tfrac{1}{2}a_\f x^2.
\end{align*}
This proves part 4.

\end{proof}

Next, we turn to Wasserstein distances and collect useful estimates on $\W_{\dmtnb}$ and $\W_{\dtnb}$. We have employed these bounds to study the small mass limits in Section~\ref{sec:small-mass}.

In Lemma~\ref{lem:W_dmtnb>W_dtnb} below, we assert that $\W_{\dmtnb}$ dominates $\W_{\dtnb}$. Since the proof of Lemma~\ref{lem:W_dmtnb>W_dtnb} is short, we include it here for the sake of completeness.
\begin{lemma} \label{lem:W_dmtnb>W_dtnb}
Let $\dmtnb$ and $\dtnb$ respectively be the distance--like functions in $\Hcal^1$ and $H$ defined in~\eqref{form:d.tilde^m_(N,beta)} and \eqref{form:d.tilde_(N,beta)}. Then, for all $\nu,\tilde{\nu}\in \Pcal r(\Hcal^1)$ and $m>0$
\begin{align} \label{ineq:W_dmtnb>W_dtnb}
\W_{\dmtnb}(\nu,\tilde{\nu}) \ge \W_{\dtnb}(\pi_1\nu,\pi_1\tilde{\nu}).
\end{align}
\end{lemma}
\begin{proof}
By the definition of Wasserstein distances as in~\eqref{form:W_d}, it suffices to prove that
\begin{align*}
\dmtnb(U,\Ut) \ge \dtnb(\pi_1 U,\pi_1\Ut).
\end{align*}
To see this, letting $\gamma(\cdot):[0,1]\to\Hcal^1$ be a path connecting $U$ and $\Ut$, we have
\begin{align*}
&\int_0^1 e^{\beta V_m(\gamma(s))}\big(m\|\pi_1\gamma'(s)\|^2_{H^1}+m^2\|\pi_2\gamma'(s)\|^2_{H}+\|\pi_1\gamma'(s)\|^2_{H}\big)^{1/2}\d s\\
&\ge \int_0^ 1 \|\pi_1\gamma'(s)\|_H\d s \ge \|\pi_1 U-\pi_1\Ut\|_H.
\end{align*}
Since the above inequality holds for any such path, from~\eqref{form:varrho^m_beta}, we obtain
\begin{align*}
\varrho^m_{\beta}(U,\Ut)\ge \|\pi_1 U-\pi_1\Ut\|_H.
\end{align*}
As a consequence, in view of expressions~\eqref{form:d.tilde^m_(N,beta)} and \eqref{form:d.tilde_(N,beta)}, we deduce
\begin{align*}
\dmtnb (U,\Ut)&=\sqrt{(N\varrho^m_\beta(U,\Ut)\mi 1)\big[1+e^{\beta V_m(U)}   +e^{\beta V_m(\Ut)}\big]}\\
&\ge \sqrt{(N \|\pi_1 U-\pi_1\Ut\|_H\mi 1)\big[1+e^{\beta \|\pi_1 U\|^2_H}   +e^{\beta \|\pi_1 \Ut\|^2_H}\big]}= \dtnb(\pi_1 U,\pi_1 \Ut).
\end{align*}
In turn, this establishes~\eqref{ineq:W_dmtnb>W_dtnb}.

\end{proof}

We conclude this section by Lemma~\ref{lem:W_dtnbhalf<W_dtnb} concerning a generalized triangle estimate between $\W_{\dtnb}$ and $\W_{\dtnbhalf}$. The argument of Lemma~\ref{lem:W_dtnbhalf<W_dtnb} follows along the lines of the proof of \cite[Lemma B.2]{glatt2022short} and thus is omitted.
\begin{lemma} \label{lem:W_dtnbhalf<W_dtnb}
Let $\dtnb$ be the distance--like function in $H$ defined in~\eqref{form:d.tilde_(N,beta)}. Then, for all $\nu_1,\nu_2,\nu_3\in \Pcal r(H)$,
\begin{align} \label{ineq:W_dtnbhalf<W_dtnb}
\W_{\dtnb}(\nu_1,\nu_3) \le C\Big(\W_{\dtnbtwo}(\nu_1,\nu_2)+\W_{\dtnbtwo}(\nu_2,\nu_3)\Big),
\end{align}
for some positive constant $C=C(N,\beta)$ independent of $\nu_1,\nu_2$ and $\nu_3$.
\end{lemma}

\bibliographystyle{abbrv}
{\footnotesize\bibliography{wave-bib}}

\end{document}